\documentclass[a4paper,10pt]{amsart}
\usepackage[utf8]{inputenc}
\usepackage{babelbib}
\usepackage[USenglish]{babel}
\usepackage{amsmath}
\usepackage{amssymb}
\usepackage{hyperref}
\usepackage{amsthm}
\usepackage[all]{xy}

\theoremstyle{plain}
\newtheorem{lem}{Lemma}[section]
\newtheorem{prop}[lem]{Proposition}
\newtheorem{cor}[lem]{Corollary}
\newtheorem{thm}[lem]{Theorem}
\newtheorem{theorem-definition}[lem]{Theorem-Definition}

\newtheorem*{thm*}{Theorem}
\newtheorem*{prop*}{Proposition}
\newtheorem*{cor*}{Corollary}
\newtheorem*{lem*}{Lemma}

\theoremstyle{definition} 
\newtheorem{ex}[lem]{Example}

\newtheorem{nota}[lem]{Notation}

\newtheorem{rem}[lem]{Remark}

\newtheorem{defn}[lem]{Definition}

\DeclareMathOperator{\Aut}{\mathrm{Aut}}
\DeclareMathOperator{\Hom}{\mathrm{Hom}}

\DeclareMathOperator{\Spec}{\mathrm{Spec}}
\DeclareMathOperator{\Gal}{\mathrm{Gal}}
\DeclareMathOperator{\Sm}{\mathrm{Sm}}

\DeclareMathOperator{\Var}{\mathrm{Var}}

\DeclareMathOperator{\Sch}{\mathrm{Sch}}
\DeclareMathOperator{\Sets}{\mathrm{Sets}}

\DeclareMathOperator{\Id}{\mathrm{id}}

\DeclareMathOperator{\Gr}{\mathrm{Gr}}
\DeclareMathOperator{\Jac}{\mathrm{Jac}}
\DeclareMathOperator{\Ord}{\mathrm{ord}}
\DeclareMathOperator{\Spf}{\mathrm{Spf}}
\DeclareMathOperator{\Sym}{\mathrm{Sym}}

\DeclareMathOperator{\Pol}{{\![\![\!}}
\DeclareMathOperator{\Por}{{\!]\!]}}

\newcommand{\LL}{\ensuremath{\mathbb{L}}}
\newcommand{\N}{\ensuremath{\mathbb{N}}}
\newcommand{\A}{\ensuremath{\mathbb{A}}}
\newcommand{\Z}{\ensuremath{\mathbb{Z}}}

\usepackage{xcolor}
\usepackage{lipsum}
\usepackage[draft]{todonotes}   

\title[Equivariant motivic integration on formal schemes]{Equivariant motivic integration on formal schemes and the motivic zeta function}
\author{Annabelle Hartmann}

\begin{document}

\begin{abstract}
For a formal scheme 
over a complete discrete valuation ring
with a good action of a finite group,
we define equivariant motivic integration, and we prove a change of variable formula for that.
To do so, we construct and examine an induced group action
on the Greenberg scheme of 
such a formal scheme.
Using this equivariant motivic integration,
we define an equivariant volume Poincar\'e series, from which we deduce Denef and Loeser's motivic zeta function including the
action of the profinite group of roots of unity.
\end{abstract}

\maketitle


\section{Introduction}
\label{introduction}

\noindent
Let $R$ be a complete discrete valuation ring of equicharacteristic zero with residue field $k$,
and let $X_\infty$ be a $sftf$ formal $R$-scheme, i.e.~a separated formal scheme which is  topologically of finite type over $R$.
Let $m$ be the relative dimension of $X_\infty$ over $R$,
and denote by $X_0$ its special fiber and by $X_\eta$ its generic fiber in the category of rigid varieties, which we assume to be smooth.
Let $\omega$ be a gauge form on $X_\eta$, i.e.~a global section of $\Omega^m_{X_\eta}$.
Under these assumptions
the volume Poincar\'e series $S(X_\infty,\omega;T)$ of
the pair $(X_\infty,\omega)$ was defined in \cite[Definition~7.2]{NiSe} by
\[
S(X,\omega;T):=\sum_{d>0}(\int_{X_\infty(d)}\hspace{-15pt}\lvert \omega(d)\rvert)T^d\in \mathcal{M}_{X_0}\Pol T \Por.
\]
Here $X_\infty(d):=X_\infty\times_R R(d)$, where $R(d)$ is a totally ramified extension of $R$ of degree $d$,
$\omega(d)$ is the pullback of $\omega$ to $X_\infty(d)$, and
\[
 \int_{X_\infty(d)}\hspace{-15pt}\lvert \omega(d)\rvert \in \mathcal{M}_{X_0}
\]
is the motivic integral of the gauge form $\omega(d)$ on $X_\infty(d)$, which was defined in \cite[Theorem-Definition~4.1.2]{motrigid}.
It takes values in the localization $\mathcal{M}_{X_0}$ with respect to the
class $\mathbb{L}$ of the affine line of the Grothendieck ring $K_0(\Var_{X_0})$ of varieties over $X_0$.
This ring is as group generated by classes $[V]$ of separated schemes $V$ of finite type over $X_0$, and
whenever $V$ is a closed subscheme of $W$, we ask $[W]$ to be equal to the sum of $[V]$ and $[V\setminus W]$;
the product is the fiber product over $X_0$.

Assume now that $X$ is a smooth irreducible algebraic variety of dimension $m+1$ over $k$,
let $f: X\to \mathbb{A}_k^1$ be a non-constant map, and assume that
$X_\infty$ is actually the completion of $X$ along $X_0:=f^{-1}(0)$.
Then using an explicit formula of the volume Poincar\'e series by means of an embedded resolution,
it was shown in \cite[Theorem 9.10]{NiSe} that
\begin{align}\label{main result}
 S(X_\infty,\frac{\omega}{df};T)=\mathbb{L}^{-m}Z(f,\mathbb{L}T)\in \mathcal{M}_{X_0}\Pol T \Por.
\end{align}
Here $Z(f,T)$ is Denef and Loeser's motivic zeta function, see \cite{DL3}.
It is given by
\[
 Z(f;T):=\sum_{d>0}[\mathcal{X}_{d,1}]\mathbb{L}^{(m+1)d}T^d\in \mathcal{M}_{X_0}[\![T]\!],
\]
where $\mathcal{X}_{d,1}$ is the subscheme of the $d$-th jet scheme of $X$ whose $k$-points are given by 
\[
 \{\psi:\text{Spec}(k[\![t]\!]/(t^{d+1}))\to X\mid f(\psi(t))=t^d \!\!\!\!\mod t^{d+1}\}.
\]
The motivic zeta function serves as a universal zeta function,
because it specializes to both the (twisted) topological zeta function and to Igusa's $p$-adic zeta function (with characters) for almost all $p$, see \cite[Section 2.3 and 2.4]{MR1618144}.
For all these zeta functions we can formulate a monodromy conjecture connecting the poles of the zeta function with the eigenvalues of the 
monodromy action on the Milnor fiber of $f$.
There is some evidence that these conjectures hold, but in general they are still open.
For more information on the different zeta functions and monodromy conjectures we refer to \cite{MR2647606}.
Apart from the connection with the other zeta function, the motivic zeta function also provides
fine invariants of hypersurface singularities, see for example \cite[Section~4.4]{DL3}.

Now observe that $\hat{\mu}$, the profinite group of roots of unity, acts, assuming that $k$ contains all roots of unity,
on $\mathcal{X}_{d,1}$
by multiplication with a primitive $d$-th root of unity.
Hence in fact we have
\[
Z(f,T)\in\mathcal{M}_{X_0}^{\hat{\mu}}\Pol T \Por,
\]
where $\mathcal{M}_{X_0}^{\hat{\mu}}$ is the localization with respect to the class $\mathbb{L}$ of the affine line
of the $\hat{\mu}$-equivariant Grothendieck ring $K_0^{\hat{\mu}}(\Var_{X_0})$ over $X_0$,
the profinite limit of the $\mu_d$-equivariant Grothendieck rings $K_0^{\mu_d}(\Var_{X_0})$.
Those rings are generated by classes $[V]$ of $X_0$-varieties $V$ with a good action
of the group of $d$-th roots of unity $\mu_d$.
Here an action on $V$ is called good if every orbit of the action lies in an affine subscheme of $V$.
We ask that $[V]+[W\setminus V]=[W]$ whenever $V\hookrightarrow W$
is a $\mu_d$-equivariant closed immersion,
and that the class of an affine bundles with affine $\mu_d$-action
only depends on its rank and base.
The product is given by the fiber product with induced $\mu_d$-action.

This means that using the volume Poincar\'e series, we do not recover
the motivic zeta function completely, but we lose the information of this group action,
which one needs in fact for the specialization to the topological and $p$-adic zeta function.
Moreover, this $\hat{\mu}$-action is closely related to monodromy, which is in particular very important with respect to the monodromy conjecture,
see \cite[Section 5.4]{MR2647606}.

The content of this paper is the construction of an equivariant version of the motivic Poincar\'e series
with values in $\mathcal{M}_{X_0}^{\hat{\mu}}\Pol T\Por$ instead of $\mathcal{M}_{X_0}\Pol T\Por$.
We also show that with this construction
we can recover the motivic zeta function including the $\hat{\mu}$-action,
i.e.~that Equation (\ref{main result}) actually holds in $\mathcal{M}_{X_0}^{\hat{\mu}}\Pol T \Por$.

\medskip
\noindent
In the first part of the paper, up to Section \ref{motivic integration},
we establish a theory of motivic integration of formal schemes
taking values in an equivariant Grothendieck ring.
To do so, we
fix a smooth $sftf$ formal scheme $X_\infty$ of relative dimension $m$
over a complete discrete valuation ring $R$ with perfect residue field $k$,
and a finite group $G$ with a good action on $X_\infty$,
which is compatible with a nice $G$-action on $R$,
i.e.~an action on $R$ with trivial induced action on $k$.

For motivic integration on formal schemes,
one measures subsets $A$ of the Greenberg scheme $\Gr(X_\infty)$ of $X_\infty$,
which replaces the arc space in the world of formal schemes, see Section \ref{greenberg definition}.
As the arc space comes along with $n$-th jet schemes,
there are $n$-th Greenberg scheme $\Gr_n(X_\infty)$ for all $n\in \mathbb{N}$, together with truncation maps $\theta_n:\Gr(X_\infty)\to\Gr_n(X_\infty)$
and $\theta_m^n:\Gr_n(X_\infty)\to\Gr_m(X_\infty)$ for $n\geq m$.
As $\Gr(X_\infty)$ is not of finite type, one uses the finite type schemes $\Gr_n(X_\infty)$ to define measures in the Grothendieck ring.

To get elements in an equivariant Grothendieck ring,
we need to deduce from the $G$-action on $X_\infty$ a good $G$-action
on $\Gr(X_\infty)$ and $\Gr_n(X_\infty)$ such that the truncation maps are $G$-invariant,
which we do in Section \ref{greenberg action}.
In particular we ask the action on $\Gr_0(X_\infty)=X_0$ to agree with the action induced by the given $G$-action on $X_\infty$.
Analogously to \cite[Proposition-Definition 3.6.1]{motrigid},
we then
define a $G$-stable cylinder $A$ of degree $n$
to be the inverse image of a G-invariant constructable subscheme $C$ of $\Gr_n(X_\infty)$,
and we set its measure to be 
\[
\mu^G(A):=[C]\mathbb{L}^{-nm}\in \mathcal{M}^G_{X_0},
\]
see Definition \ref{defn cylinder} and Definition \ref{defn measure}.
As a $G$-stable cylinder $A$ of degree $n$ is also a $G$-stable cylinder of degree $m$ for $m\geq n$,
we ask in addition that for $m\geq n$ the truncation map $\theta_{m+1}^m: \theta_{m+1}(A)\to \theta_m(A)$ is piecewisely a $G$-equivarinat affine bundle of rank $m$
with affine $G$-action,
which implies using the second relation in the equivariant Grothendieck ring that the measure of $A$ is well defined.
We show that this assumption is automatic in the case that $X_\infty$ is smooth:
already in the non-equivariant case, it was shown that 
$\Gr_{n+1}(X_\infty)$ is an affine bundle of rank $m$ over $\Gr_n(X_\infty)$.
We can show in addition that the action on this affine bundle
is affine over the action on the base, see Proposition \ref{prop structure map}.
To do so, we use a description of $\Gr_{n+1}(X_\infty)$ in terms of derivations over elements in $\Gr_n(X_\infty)$,
and an explicit $G$-action on these derivations.

Similarly to the non-equivariant case, we call a function
$\alpha: \Gr(X_\infty)\to \mathbb{Z}$ with finite image
naively $G$-integrable
if all fibers are $G$-stable cylinders, and set
\[
\int_{X_\infty}\hspace{-10pt} \mathbb{L}^{-\alpha}:=\sum_{i\in \mathbb{Z}}\mu^G(\alpha^{-1}(i))\mathbb{L}^{-i}\in\mathcal{M}^G_{X_0}.
\]
To be able to compute such an integral,
we need in particular a way to change variables.
Hence assume that we have another smooth formal $R$-scheme $Y_\infty$
with the same properties as $X_\infty$,
and a $G$-equivariant $R$-morphism
$h:Y_\infty\rightarrow X_\infty$
such that the map $Y_\eta\rightarrow X_\eta$ on the generic fibers
is an open immersion, and 
${Y_\eta (K')\to X_\eta(K')}$ is a bijection for all unramified extensions $K'$ of $K$,
where $K$ denotes the fraction field of $R$.
For this setup, we can show the following theorem:
\begin{thm*}[Change of variables formula,  Theorem \ref{changevar}]
Assume that $G$ is abelian and acts tamely on $R$, i.e.~$\lvert G\rvert$ is prime to the characteristic of the residue field $k$ of $R$,
and that $R$ has equal characteristic and $k$ contains all roots of unity. Then
\[
 \int_{X_\infty}\hspace{-10pt} \LL^{-\alpha}d\mu^G_{X_0} = \int_{Y_\infty}\hspace{-10pt} \LL^{-(\alpha\circ h + \Ord(\Jac_h))}d\mu^G_{X_0}\in\mathcal{M}^G_{X_0}. 
\]
 \end{thm*}
\noindent
Here $\Ord(\Jac_h)$ is the order of the Jacobian, which measures the relative sheave of differentials of $h$,
see Definition~\ref{def ordjac}.
This theorem also holds in the non-equivariant case,
see \cite[Th\'eor\`eme 7.3.3]{sebag1}.

To show the change of variables formula, we need to
compare
$\Gr_n(Y_\infty)$ and $\Gr_n(X_\infty)$ in the equivariant Grothendieck ring.
Note that
$h$ induces a map $\Gr_n(h)$ between these two rings,
which we study in Section \ref{greenberg h}.
We can show that if $n$ is big enough,
the reduced subscheme of the inverse image under $\Gr_n(h)$ of every point $x_n$ in $\Gr_n(X_\infty)$
is an $G_x$-equivariant affine bundle of rank depending on the order of the Jacobian with affine $G_x$-action,
where $G_x$ denotes the stabilizer of $x_n$, see Proposition \ref{structure h}.
Using some spreading out argument in Lemma \ref{lemma local global},
 we can compute from this $\Gr_n(Y_\infty)$
in terms of $\Gr_n(X_\infty)$ in $\mathcal{M}_{X_0}^{G}$, which implies
the change of variables formula.

Note that while we can define $G$-integrable functions and describe the truncation map for general $R$ and $G$,
we can proof Proposition \ref{structure h} and hence the change of variables formula only in the case
that $R$ has equal characteristic and $G$ is abelian and acts tamely on $R$.
This is in particular due to the fact that we use a concrete description of the action on $R$, which we do not get in the non-abelian or wild case.
Moreover, in the case of mixed characteristic, one gets problems with non-separable extensions already in the non-equivariant case,
see \cite[Section 2.4]{MR2885338}.

\medskip
\noindent
Based on the developed theory of equivariant motivic integration,
 we generalize in the second part of this paper the definitions of the integral of a gauge form
and the volume Poincar\'e series, from which we finally deduce Denef and Loeser's motivic zeta function 
including the $\hat{\mu}$-action.

Let $R$  be a complete discrete valuation
ring of equal characteristic with residue field $k$ containing all roots of unity and a nice tame action of a finite abelian group $G$.
Let $X_\infty$ be a $sftf$ formal $R$-scheme with a good $G$-action compatible with the $G$-action on $R$, generically smooth
 but not necessarily smooth,
with a gauge form $\omega$ on its generic fiber $X_\eta$.
As in the non-equivariant case,
we associate a function $\Ord(\omega): \Gr(X_\infty)\to \mathbb{N}$ to
this gauge form, see Definition~\ref{dfn ordgauge}.

In order to integrate $\Ord(\omega)$,
we need a smooth scheme to integrate over.
Here we use, as in the non-equivariant case, a weak N\'eron model $U_\infty$ of $X_\eta$,
i.e.~$U_\infty$ is a smooth formal $sftf$ scheme,
whose generic fiber is an open rigid subspace of the generic fiber of $X_\infty$,
and the induced map $U_\infty(R')\to X_\infty(K')$ is a bijection for every unramified extension $R'$  of $R$ with quotient field $K'$,
see Definition \ref{def wnm}.
More precisely, we show in Theorem \ref{neronsmooth} that, under our assumptions, there exists always a $G$-equivariant
N\'eron smoothening $f:U_\infty \to X_\infty$ of $X_\infty$, meaning that $U_\infty$ is a weak N\'eron model of $X_\eta$ with an action of $G$,
and there is a $G$-equivariant isomorphism $h: X_\infty'\to X_\infty$ inducing an isomorphism on the generic fibers,
such that $f$ factors through an open $G$-equivariant immersion $U_\infty \hookrightarrow X_\infty'$.
Using such a smoothening $f:U_\infty \to X_\infty$, we define in Section \ref{integral of gauge}
\[
\int_{X_\infty}\hspace{-10pt}\lvert\omega\rvert:=\int_{U_\infty}\hspace{-10pt}\LL_{X_0}^{-\Ord(f^*\omega)}d\mu_{X_0}\in\mathcal{M}^G_{X_0}.
\]
As a weak N\'eron smoothening is not unique,
we need to show that this is well defined,
for which we use the change of variables formula, Theorem \ref{changevar}.

Using a G-equivariant N\'eron smoothening of $X_\infty$,
we also define the equivariant motivic Serre invariant of $X_\infty$ to be the class of the special fiber of such a weak N\'eron model
in $K_0^G(\Var_{X_0})/(\mathbb{L}-1)$,
see Section \ref{eq serre}.
This generalizes the Serre invariant, see
\cite[Definition 6.2]{NiSe2},
which is closely connected to the existence of rational points.
Some concrete applications of the motivic Serre invariant can be found for example in \cite{1009.1281}.

Now we can look at a $sftf$ formal $R$-scheme $X_\infty$, which is generically smooth.
We now assume that $R$
has equal characteristic zero.
Note that $\mu_d$, the group of $d$-th roots of unity,
acts on $R(d)$ and hence on $X_\infty(d)$.
Let $\hat{\mu}$ be again the profinite limit of the $\mu_d$,
hence we can define the equivariant volume Poincar\'e series by
\[
S(X,\omega;T):=\sum_{d>0}(\int_{X_\infty(d)}\hspace{-15pt}\lvert \omega(d)\rvert)T^d\in \mathcal{M}_{X_0}^{\hat{\mu}}\Pol T \Por,
\]
see Definition \ref{volume}.
Similarly, one can define the equivariant Serre Poincar\'e series
by summing over the equivariant Serre invariants of the $X_\infty(d)$.

To compute these series, we need a concrete $\mu_d$-equivariant N\'eron smoothening of $U_\infty(d)\to X_\infty(d)$ for all $d$.
The induced action on the special fiber of $U_\infty(d)$ agrees then with the action on $\Gr_0(U_\infty(d))$,
and can be used to compute the corresponding integral.
To get the desired smoothening we fix an embedded resolution of singularities $h: X_\infty'\to X_\infty$,
i.e.~a morphism of $sftf$ formal schemes inducing an isomorphism on the generic fibers,
such that $X_\infty'$ is regular,
and its special fiber is a simple normal crossing divisor $\sum_{i\in I} N_iE_i$.
Let $\widetilde{X_\infty'(d)}$ be the normalization of $X_\infty\times_R R(d)$ with induced $\mu_d$-action.
In Theorem~\ref{neron} we show that the induced map $U_\infty(d):=\Sm(\widetilde{X_\infty'(d)})\to X_\infty(d)$ is a $\mu_d$-equivarinat N\'eron smoothening
if $d$ is not $X_0'$-linear, see Definition \ref{defX0lin}.
Using this N\'eron smoothening, some local computations in Theorem \ref{neron}, and Lemma \ref{correction nonlin} to get rid of the $X_0'$-linearity,
we can show the following formula, which was shown in \cite{NiSe} without $\hat{\mu}$-action:
\begin{thm*}[Theorem \ref{explicit}]
\[
S(X_\infty,\omega;T)
=\LL^{-m}\hspace{-5pt}\sum_{\emptyset\neq J\subset I}\hspace{-5pt}(\LL-1)^{|J|-1}[\widetilde{E}_J^{o}]
\prod_{i\in J}\frac{\LL^{-\mu_i}T^{N_i}}{1-\LL^{-\mu_i}T^{N_i}}\in \mathcal{M}^{\hat{\mu}}_{X_0}\Pol T \Por.
\]
\end{thm*}
\noindent
Here we use the following notation:
for any subset $J\subset I$,
${E_J^o:=\bigcap _{j\in J}E_j\!\setminus\! \bigcup_{i\in I \setminus J} E_i}$, and
$m_J:=\text{gcd}\{N_i\mid i\in J\}$.
For each non-empty subset $J\subset I$, we
can cover $E_J^o\subset X_\infty$  by finitely many affine open
formal subschemes $U_\infty=\mathrm{Spf}(V)$ of $X_\infty$, such
that on $U_\infty$, $t=u\prod _{i\in J}x_i^{N_i}$, with $t$ a uniformizing parameter of $R$ and $u$ a unit in $V$, and the $x_i$ are local coordinates.
 The
restrictions over $E_J^o$ of the \'etale covers
$U_\infty':=\mathrm{Spf}(V\{T\}/(uT^{m_J}-1))$
of $U_\infty$
glue together to an \'etale cover $\widetilde{E}_J^o$ of $E_J^o$, on which
$\hat{\mu}$ acts 
by multiplying $T$ with a $m_J$-th root of unity on every chart.

\medskip
\noindent
Using a similar formula for Denef and Loeser's motivic zeta function $Z(f;T)$,
we can deduce from this formula the following theorem:
\begin{thm*}[Theorem \ref{comparzeta}]
Let $X$ be a smooth irreducible variety of dimension $m+1$ over a field $k$ containing all roots of unity,
let $f:X\rightarrow \A^1_k$ be a
 dominant morphism,
 and let $X_\infty$ be the completion of $X$ along the special fiber $X_0:=f^{-1}(0)$.
 Assume that there exists a global gauge form $\omega$ on the generic fiber of $X_\infty$. Then
\[
 S(X_\infty, \frac{\omega}{df};T)=\LL^{-m}Z(f;\LL T)\in\mathcal{M}^{\hat{\mu}}_{X_0}\Pol T\Por.
\]
\end{thm*}
\noindent
Hence we finally recover the motivic zeta function with $\hat{\mu}$-action from the equivariant volume Poincar\'e series.
This implies in particular that if we want to show something about the motivic zeta function, for example the motivic monodromy conjecture,
we can also prove it for the equivariant volume Poincar\'e series.

Without $\hat{\mu}$-action,
the corresponding monodromy conjecture for the volume Poincar\'e series was proven in the case of Abelian varieties,
see \cite[Theorem~8.5]{MR2782205}.
As already remarked in \cite[Section 2.5
]{MR2782205},
the non-equivariant version of the conjecture
does not imply Denef and Loeser's conjecture completely,
because one still misses the $\hat{\mu}$-action.
Hence it would be very nice to generalize their proof to the equivariant volume Poincar\'e series.

\medskip
\noindent
Finally, in Section \ref{rec mnc}, we can also recover from the equivariant volume Poincar\'e series the motivic nearby cycles $\mathcal{S}_f$, 
which are defined by formally taking the
limit of $-Z(f,T)$ for $T$ to $\infty$.
This invariant was defined in \cite{DL3}
and investigated further for example in \cite{MR2106970}.
Here we do not need to assume the existence of a global gauge form on $X_\infty$,
see Definition \ref{eq mot vol2}.
In fact we can also define an equivariant motivic volume $\mathcal{S}_{X_\infty}$ for all formal $k\Pol t\Por$-schemes $X_\infty$,
which agrees with $\mathcal{S}_f$ in the case that $X_\infty $ comes from a map $f: X\to \mathbb{A}^1_k$.
Using Theorem \ref{changevar} we get a formula for $\mathcal{S}_{X_\infty}$ in terms of an embedded resolution of $X_\infty$,
from which, together with
a result from \cite{abi2} on the existence of a quotient map on the equivariant Grothendieck ring of varieties,
the following corollary follows:
\begin{cor*}[Corollary \ref{special fiber modulo l}]
 Let $X_\infty$ be a $sftf$ formal scheme of relative dimension $m$ over $R$
 with smooth generic fiber.
Then the class of 
$ X_0'$ modulo $\mathbb{L}$ in $\mathcal{M}_{X_0}$ does not depend on the choice of 
an embedded resolution $h:X_\infty'\to X_\infty$.
\end{cor*}

\subsubsection*{Acknowledgments}
During the research for this article, I was supported by a research fellowship
 of the \emph{DFG} (Aktenzeichen HA 7122/1-1).
 I would like to thank Alberto Bellardini and in particular Emmanuel Bultot for discussing many technical problems with me.
Moreover, I am thankful to Johannes Nicaise for his discussions, ideas and suggestions.

\section{Preliminaries}
\label{preliminaries}

\subsection{Complete discrete valuation rings}
\label{complete dvr}
Throughout this article, $R$ always denotes a complete discrete valuation 
ring, with residue field $k$, and quotient field $K$.
In order to avoid problems in positive characteristics, we assume that $k$ is always perfect.
We fix a uniformizing parameter $t$,
i.e.~a generator for the
maximal ideal of $R$.
Moreover, if $R$ has equal characteristic, we fix a $k$-algebra structure
$\mu:k\rightarrow R$; this yields an isomorphism $R\cong
k\Pol t \Por$. For any integer $n\geq 0$, we put $R_n:=R/(t^{n+1})$.

For any integer $d>0$ prime
to $p$, we put $K(d):=K[t(d)]/(t(d)^d-t)$. This is a totally
ramified extension of degree $d$ of $K$.
Note that if $k$ is not
algebraically closed, such an extension is not necessarily unique.
We denote by $R(d)$ the
normalization $R[t(d)]/(t(d)^d-t)$ of $R$ in $K(d)$, and for
each $n>0$, we embed $R(d)$ in $R(nd)$ by putting
$t(d)=t(nd)^n$.

\subsection{Formal schemes and rigid varieties}
\label{formal schemes}
An $stft$ formal $R$-scheme $X_\infty$
is a separated formal scheme,
topologically of finite type over $R$.
We denote the category of $stft$ formal $R$-schemes
by $(stft/R)$.
For every $X_{\infty}\in (stft/R)$, 
we denote its special fiber
by $X_0$, and its generic fiber (in the category of separated
quasi-compact rigid $K$-varieties) by $X_\eta$.
For any integer $n\geq 0$, we put $X_n:=X_\infty\times_R
R_n$, which is a separated $R_n$-scheme of finite type.

We say that $X_\infty$
is generically smooth, if $X_\eta$ is a smooth rigid $K$-variety.
We denote by $\Sm(X_\infty)$ the smooth part of $X_\infty$ over
$R$.

\subsection{Group actions}
\label{group actions}
Fix a finite group $G$. 
We say that a left action of $G$ on a scheme $S$ is \emph{good} if every
orbit of this action is contained in an affine open subscheme of $S$.
By \cite[Expos\'e V, Proposition 1.8]{MR0217087} this is the same as requiring a cover of $U$ by affine, open, $G$-invariant subschemes.
By requiring the action to be good, one makes sure that the quotient exists in the category of schemes, see \cite[Expos\'e~V.1]{MR0217087}.
If not mentioned otherwise, all group actions on schemes will be left actions.

For a given separated scheme $S$ with a good $G$-action,
we denote by $(\Sch_{S,G})$ the category whose objects are separated
schemes of finite type over $S$ with a good $G$-action such that the structure map is $G$-equivariant, and whose morphisms are
$G$-equivariant morphisms of $S$-schemes.
One can check that the fiber product exists in this category
by constructing a good $G$-action on the fiber product in the category of separated schemes of finite type.

Let $R$ be a complete discrete valuation ring as in Section \ref{complete dvr}.
A \emph{nice action of $G$ on $R$} is a right action of $G$ on $R$, such
that the induced action on the residue field $k$ is trivial.
In the case of equal characteristic we also assume that $G$ respects the
chosen $k$-algebra structure. 
We say that $G$ acts \emph{nicely} on $R$.
Note that a nice $G$-action on $R$ induces a unique $G$-action on $R_n$ for all $n>0$,
with the property that the quotient maps $R\to R_n$ and $R_n\to R_m$  for $n\geq m>0$ are $G$-equivariant.

We call a $G$-action on $R$ \emph{tame} if the characteristic of the residue field $k$ is prime to the order of $G$,
and \emph{wild} otherwise.

\begin{ex}
\label{ex action on r}
 Let $R$ be a complete discrete valuation ring, and consider $R(d)$,
 a finite totally ramified extension of $R$  of degree $d$,
 with quotient field $K(d)$.
 Then $G:=\Gal(K(d)/K)$ acts on $R(d)$, and because the extension is totally ramified the induced action on the residue field $k$ of $R(d)$ is trivial.
 
 Assume $R$ has equal characteristic with residue field $k$ containing all roots of unity,
 and that $d$ is prime to the characteristic $p$ of $k$.
 Then we have that $R\cong k\Pol t \Por$,
 $R(d)\cong R\Pol t(d)\Por$ with $t(d)^d=t$,
 and the action of $G=\Gal(K(d)/K)\cong \mu_d$, where $\mu_d$ is the group of $d$-th roots of unity, on $R(d)$ is
 given by sending $t(d)$ to $\xi t(d)$ for all $\xi\in \mu_d$.
 \end{ex}

 \noindent
 A \emph{good} $G$-action on a formal scheme $X_\infty$ is a left
action of $G$ on $X_\infty$, such that any orbit is contained in an
affine open formal subscheme of $X_\infty$.
If not mentioned otherwise, all actions on formal schemes will be left actions.

For a given complete discrete valuation ring $R$ with a nice $G$-action, we denote by
$(stft/R,G)$ the category of flat, $stft$ formal $R$-schemes
$X_\infty$, endowed with a good $G$-action compatible with the $G$-action on $R$,
i.e.~the structure morphism $X_\infty\rightarrow
\mathrm{Spf}\,R$ is $G$-equivariant.
Morphisms are
$G$-equivariant morphisms of formal $R$-schemes.
Note that such a $G$-action on a formal scheme $X_\infty$ induces a $G$-action on the $R_n$-scheme $X_n=X_\infty \times_R R_n$ with $G$-equivariant structure map.
Moreover, for all $n\geq m \geq 0$ the restriction maps $X_m\to X_n$ are $G$-equivariant.

\begin{ex}
\label{stan ex}
Consider $R(d)$ with the nice $G$-action as in Example \ref{ex action on r}.
Let $X_\infty$ be a $stft$ formal $R$-scheme, and put
$X_\infty(d):=X_\infty\times_R R(d)$.
Using the universal property of the fiber product, the nice $G$-action on $R(d)$
induces a good $G$-action on
$X_\infty(d)$ such that the structural morphism $X_\infty(d)\rightarrow
\mathrm{Spf}\,R(d)$ is $G$-equivariant.
Hence in particular $X_\infty(d)\in (stft/R(d),G)$.
\end{ex}

\section{Greenberg schemes with group actions}
\label{greenberg}
\noindent
Throughout this section, let $G$ be an abstract finite group, and let $R$ be a complete discrete valuation ring with perfect residue field $k$,
endowed with a nice $G$-action.

\subsection{The Greenberg scheme of a formal scheme}
\label{greenberg definition}
In this subsection we give a short summary of the construction of the Greenberg scheme of a formal scheme,
and fix notations.
We do this in consideration of the nice group action on $R$.
Details, proofs, and more references can be found for example in \cite[Chapter 2.2]{MR2885336}.

\subsubsection{The ring scheme $\mathcal{R}_n$}
Let $n\in \mathbb{N}$.
If R has equal characteristic, set
  \[
   \mathcal{R}_n:(k-alg)\to (rings);\ A\mapsto A\otimes_kR_n.
  \]
If $R$ has mixed characteristic, 
  then let $\mathcal{R}_n$ be the sheafification in the fpqc-topology of the functor
  \[
   \tilde{\mathcal{R}}_n:(k-alg)\to (rings);\ A\mapsto W(A)\otimes_{W(k)}R_n,
  \]
where $W(A)$ is the ring of Witt vectors with coefficient in $A$.
In both cases, $\mathcal{R}_n$ is represented by a ring scheme. 
We also denote this scheme by $\mathcal{R}_n$.

Note that the quotient maps $q_m^n: R_n\to R_m$ induce maps of functors
by sending $f\in\mathcal{R}_n(A)$ to ${(id\otimes q^n_m)\circ f}$,
and thus of schemes $\mathcal{R}_n\to \mathcal{R}_m$ for all $n\geq m \geq 0$.
We define $\mathcal{R}$ to be the $k$-scheme representing the limit of the projective
system $(\mathcal{R}_n)_{n\in \N}$.

\begin{rem}
\label{action rn}
Every automorphism $g_{R_n}$ of $R_n$ inducing the identity on $k$
gives rise to a morphism of the functor $\mathcal{R}_n$
by sending $f\in \mathcal{R}_n(A)$ to $(id\otimes g_{R_n})\circ f$, and
hence we get an automorphism of the scheme $\mathcal{R}_n$.
Thus the right $G$-action on $R_n$ induced by the right $G$-action on $R$
 gives us naturally a right $G$-action on $\mathcal{R}_n$.
 As for all $n\geq m\geq 0$ the quotient maps $q^n_m: R_n \to R_m$ are $G$-equivariant,
 the same holds by construction for the induced maps $\mathcal{R}_n\to \mathcal{R}_m$.
\end{rem}

\subsubsection{The ideal schemes $\mathcal{J}^m_n$}\label{ideal scheme}
In the proofs in Section~\ref{greenberg truncation}
and Section \ref{greenberg h},
we will need to consider ideal schemes, which
can be found for example in \cite[Chapter~4, 2.3.1]{CNS}:
define for all $m\geq n\geq 0$ the functor
\[
 \mathcal{J}^m_n: (k-alg)\to (\Sets);\  A\to \ker(\mathcal{R}_{m}(A)\to \mathcal{R}_n(A)).
\]
It is representable by a closed subscheme of $\mathcal{R}_{m}$,
which we call the \emph{ideal scheme} $\mathcal{J}^m_n$.
If $n\leq m \leq 2n+1$,
the square of $\mathcal{J}^m_n$ in $\mathcal{R}_{m}$ is zero, hence we can view $\mathcal{J}_n^m$ as a module over $\mathcal{R}_n$.

In the case of equal characteristic, for every choice of an uniformizer $t\in R$ we have that for every $k$-algebra $A$
\begin{align}\label{structure Jmn}
 \mathcal{J}^m_n(A)=\{a_{n+1}t^{n+1}+\dots+a_m t^m\mid a_i\in A\}.
\end{align}
Hence we get a functorial bijection $i(A): \mathcal{J}_n^m(A)\to A^{m-n}$ by sending an element of the form $a_{n+1}t^{n+1}+\dots+a_mt^m$ to $(a_{n+1},\dots,a_m)\in A^{m-n}$.

\begin{rem}\label{action jn}
Note that the $G$-action on $\mathcal{R}_{m}$ constructed in Remark \ref{action rn} restricts to $\mathcal{J}_n^m$,
because the map $\mathcal{R}_{m}\to \mathcal{R}_n$ is $G$-equivariant.

Let $A$ be a $k$-algebra.
Take any $g\in G$, and denote by $g_{m}\in \Aut(\mathcal{J}_n^m(A))$
and ${g_n\in \Aut(\mathcal{R}_n(A))}$ the corresponding automorphisms.
As both the $G$-action on $\mathcal{J}_n^m$ and $\mathcal{R}_n$
come from the same $G$-action on $R$,
 the action on $\mathcal{J}_n^m(A)$
is compatible with the $\mathcal{R}_n(A)$-module structure,
which we have in the case of $n\leq m \leq 2n+1$,
i.e.~for all $r\in \mathcal{R}_n(A)$ and $x\in \mathcal{J}_n^m(A)$ we have
 $g_{m}(r x)=g_n(r)g_{m}(x)$.
\end{rem}

\begin{ex}\label{ex action on Jmn}
 Assume that $R$ has equal characteristic and $k$
 contains all roots of unity.
As $G$ acts nicely on $R$,
$G$ acts trivially on the chosen lifting of $k$.
Using this lifting of $k$ we get an $A$-module structure on $\mathcal{J}_n^m(A)$ with the property
that for all $a\in A$ and $x\in \mathcal{J}_n^m(A)$ we have
$g_m(ax)=ag_m(x)$.
Assume now in addition that $ G $ acts tamely on $R$.
Then after maybe changing the uniformizer $t$ of $R$, we may assume
that $g\in G$ acts on $\mathcal{J}_n^m(A)$, which is given as in (\ref{structure Jmn}),
 by sending ${a_{n+1}t^{n+1}+\dots+a_mt^m}$
 to $a_{n+1}\xi^{n+1}t^{n+1}+\dots+a_m\xi^mt^m$,
 where $\xi\in k$ is a $\lvert g\rvert$-th root of unity.
 If $G$ is abelian, we can chose a $t$ not depending on $g$.
\end{ex}

\noindent
Now introduce the notation  $\mathcal{J}_n$ for $ \mathcal{J}_n^{n+1}$,
which has a canonical structure as a vector space.
Let $\mathfrak{m}\subset R$ be the maximal ideal,
and denote by $V$ the one dimensional $k$-vector space $\mathfrak{m}/\mathfrak{m}^2$.
Set $V(i):=V^{\otimes i}$ for $i\geq 0$, and for $i<0$ set $V(i):=V(-i)^*$, the dual of $V(-i)$.
Let $A$ be again a $k$-algebra.
Then, as explained in \cite[Chapter 4, 2.3.1]{CNS},
the map
\[
 \Psi_{\text{eq}}:V(n+1)\otimes A\to \mathcal{J}_n(A);\  v_0\otimes \dots \otimes v_n\otimes a\mapsto  v_0\dots v_n a
\]
in the case of equal characteristic, and
\[
 \Psi_{\text{mix}}:V(\beta)\otimes~^{p^{\alpha}}\!\!A\to \mathcal{J}_n(A);\  v_0\otimes \dots \otimes v_n\otimes a\mapsto  v_0\dots v_n(0,\dots,0,a)
\]
in the case of mixed characteristic $(0,p)$ of absolute ramification index $e\geq 0$,
are isomorphism of $A$-modules.
In the case of mixed characteristic,
$\alpha$ is the integer, such that $R_{n+1}$ has characteristic $p^{\alpha}$,
$\beta$ is the remainder of the Euclidean division of $n+1$ by $e$,
and
$(0,\dots,0,a)\in W(A)$
is the $(\alpha-1)$-th Verschiebung of $a$.

\begin{rem}\label{gaction Psi}
 Note that the $G$-action on $R_1=R/\mathfrak{m}^2$ restricts to $V$,
 because automorphisms map maximal ideals to maximal ideals.
As the action of $G$ on $R$ is nice,
and hence the induced action on the residue field $k$ is trivial,
the action on $V$ is given by multiplication with an element $\xi_g\in k$
for every $g\in G$.
For $i\geq 0$, let $G$ act on $V(i)$ by acting on the factors separately. 
It follows in particular that the automorphisms of $V(i)$
defining the action of $G$ are linear maps.

It is easy to see that $\Psi_{\text{eq}}$ and $\Psi_{\text{mix}}$ are in fact $G$-invariant,
for the considered $G$-action on $\mathcal{J}_n$
and on $V(i)$ and the trivial action on $A$ and on $~^{p^\alpha}\!\!A$, respectively.

For $i\leq 0$, we associate for every $g\in G$ the dual $g_V^*$ of the corresponding automorphism $g_V$ of $V(-i)$,
i.e.~$f\in V(i)=V^*(-i)$ gets send to $f\circ g_V$.
With this actions the canonical map $V(i)\otimes V(-i)\to k$ sending $(f,v)$
to $f(v)\in k$ is $G$-equivariant, if we equip $k$ with the trivial action of $G$.

\end{rem}

\subsubsection{The Greenberg scheme}

\begin{defn}
Let $X_n$ be an $R_n$-scheme of finite type.
By \cite{Greenberg2} the functor
 \[
 (k-alg)\rightarrow (Sets);\  A\mapsto\Hom_{R_n}(\Spec(\mathcal{R}_n(A)),X_n)=X_n(\mathcal{R}_n(A))
\]
is representable by a $k$-scheme of
finite type.
We call this scheme the \emph{$n$-th Greenberg scheme} $\Gr_n(X_n)$ of $X_n$.
If $X_\infty$ is a $stft$ formal $R$-scheme, we put for each $n\geq 0$
$\Gr_n(X_\infty):=\Gr_n(X_n)$,
with $X_n=X_\infty \times_R R_n$.
\end{defn}
\noindent
For any pair of
integers $n\geq m\geq 0$, and any $R_n$-scheme of finite type $X_n$,
the morphisms $\mathcal{R}_n(A)\to \mathcal{R}_m(A)$ for all $k$-algebras $A$ induce a canonical morphism of $k$-schemes
\[
\theta^n_m: \Gr_n(X_\infty)\rightarrow \Gr_m(X_\infty).
\]
As explained in \cite[Chapter 2.2]{MR2885336}, this morphism is affine. 
Hence we can take
the projective limit in the
category of $k$-schemes.

\begin{defn}
 Let $X_\infty$ be a $stft$ formal $R$-scheme. Then
\[
 \Gr(X_\infty):=\lim_{\stackrel{\longleftarrow}{n}}\Gr_n(X_\infty)
 \]
 is called the \emph{Greenberg scheme} of $X_\infty$.
\end{defn}

\noindent
For all $n\geq 0$, $\Gr(X_\infty)$ is endowed with natural truncation maps
\[
\theta_n:\Gr(X_\infty)\rightarrow \Gr_n(X_\infty).
\]

\noindent
 Let $h:Y_\infty \to X_\infty$ be a morphism of formal schemes,
i.e.~we have compatible morphisms $h_n: Y_n \to X_n$ for all $n\in \mathbb{N}$.
 The $h_n$ induce maps 
 \[
 \Gr_n(h): \Gr_n(Y_\infty)\to \Gr_n(X_\infty),
 \]
 which are, on the level of functors, given
 by sending a map $\gamma: \Spec(\mathcal{R}_n(A))\to Y_n$
 to ${h_n \circ \gamma: \Spec(\mathcal{R}_n(A)) \to X_n}$ for all $k$-algebras $A$.
 By construction, these maps are compatible with the truncation maps,
 so we also get a map
 \[
  \Gr(h): \Gr(Y_\infty)\to \Gr(X_\infty),
 \]
and the following diagram commutes:
\begin{equation}
 \label{com dia}
 \xymatrix{
 \Gr(Y_\infty)\ar[r]^{\theta_n}\ar[d]_{\Gr(h)}&\Gr_n(Y_\infty)\ar[r]^{\theta^n_m}\ar[d]^{\Gr_n(h)}&\Gr_m(Y_\infty)\ar[d]^{\Gr_m(h)}\\
 \Gr(X_\infty)\ar[r]_{\theta_n}&\Gr_n(X_\infty)\ar[r]_{\theta^n_m}&\Gr_m(X_\infty)
 }
\end{equation}
Note that $\theta_n$ and $\theta_m^n$ depend on $Y_\infty$ and $X_\infty$, respectively.
To keep the notation simple we do not indicate this dependence.
If it is clear from the context which map we mean, we will write $h$ instead of $\Gr(h)$ or $\Gr_n(h)$.


%
%

\begin{rem}\label{R(F)}
Note that every point $x\in\Gr(X_\infty)$ with residue field $F$ corresponds to a section
$\psi\in X_\infty(\mathcal{R}(F))$.
If $R$ has equal characteristic,
then it is easy to see that $\mathcal{R}(F)$
is a complete discrete valuation ring with residue field $F$ and ramification index one over $R$.
\end{rem}

\begin{rem}\label{rem hn}
Let $Y$ be a $k$-scheme, and consider
$h_n(Y):=(\lvert Y \rvert, \Hom_k(Y, \mathcal{R}_n))$,
the locally ringed space with underlying topological space $\lvert Y\rvert$ and structure sheaf $\Hom_k(Y,\mathcal{R}_n)$.
If $Y=\Spec(A)$ is affine, $h_n(Y)$ is isomorphic to the affine scheme
$\Spec(\mathcal{R}_n(A))$.
With this notation we have for all formal $sftf$ schemes $X_\infty$ over $R$ and all $n\in \mathbb{N}$ that
\[
 \Gr_n(X_\infty)(Y)=\Hom_k(Y,\Gr_n(X_\infty))=\Hom_{R_n}(h_n(Y),X_n).
\]
Some more information on $h_n(Y)$ can be found for example in \cite[Section 3.1]{sebag1}.
\end{rem}

%

\subsection{Construction of the group action}
\label{greenberg action}
The aim of this subsection is to construct a $G$-action on the Greenberg scheme
of a stft formal $R$-scheme with $G$-action,
such that the maps in Diagram (\ref{com dia}) are $G$-equivariant.
Note that
the following construction was already done in the case of Example \ref{stan ex} in
\cite[6.1.2]{NiSe-weilres};
here we show how this result extends in a more general setting.

\begin{prop}\label{Gaction}
For every $X_\infty \in (stft/R,G)$,
there are good actions of $G$ on the $k$-schemes
$\Gr_n(X_\infty)$ for every integer $n\geq 0$, and
on $\Gr(X_\infty)$,
such that the action on $\Gr_0(X_\infty)\cong X_0$ coincides with the $G$-action induced by the action on $X_\infty$,
and such that for
$n\geq m\geq 0$, the truncation maps
\[
\theta^n_m:\Gr_n(X_\infty)\rightarrow \Gr_m(X_\infty)\text{ and }
\theta_n:\Gr(X_\infty)\rightarrow\Gr_n(X_\infty)
\]
are $G$-equivariant.
Moreover, if $h: Y_\infty \to X_\infty$ is a morphism in $(stft/R,G)$,
then the induced maps $\Gr_n(h)$ and $\Gr(h)$ are $G$-equivariant, too.
\end{prop}

\begin{proof}
Note that
it suffices to construct the action of $G$ on $\Gr_n(X_\infty)$ for any integer $n\geq 0$, and to show
that the truncation maps $\theta^m_n$ 
are equivariant for any $n\geq m\geq 0$. The
action on $\Gr(X_\infty)$ is then obtained by passing to the
projective limit $n\to \infty$, and the $\theta_n$ are $G$-equivariant by construction.

Take any $k$-algebra $A$.
By Remark \ref{action rn} there is a right $G$-action on $\mathcal{R}_n(A)$ which is compatible with the $G$-action on $R_n$.
Hence we get a left $G$-action on $\Spec(\mathcal{R}_n(A))$
such that the structure map to $\Spec(R_n)$ is $G$-invariant.
Fix a $g\in G$. 
Let ${g}_{\mathcal{R}_n(A)}\in \Aut(\Spec(\mathcal{R}_n(A)))$ be the corresponding automorphism.
Consider the $G$-action on $X_n$ induced by the $G$-action on $X_\infty$,
and let $g_{X_n}\in \Aut(X_n)$ be the automorphism corresponding to $g$.
We define a map
\begin{align*}\label{eq gaction}
 g_F\!:\!\Hom_{R_n}(\Spec(\mathcal{R}_n(A)),X_n)\!=\!\Gr_n(X_\infty)(\Spec(A))&\rightarrow \Hom_{R_n}(\Spec(\mathcal{R}_n(A)),X_n);\\
 f&\mapsto g_{X_n} \circ f \circ g_{\mathcal{R}_n(A)}^{-1}.
\end{align*}
Here $g_{X_n} \circ f \circ g_{\mathcal{R}_n(A)}^{-1}$  is an $R_n$-morphism,
because the structure map of the two $R_n$-schemes $\Spec(\mathcal{R}_n(A))$ and $X_n$ are $G$-equivariant.
Hence $g_F$ is well defined.
For every morphism of $k$-algebras $A'\to A$ the induced map ${\Spec(\mathcal{R}_n(A))\to \Spec(\mathcal{R}_n(A'))}$ is $G$-equivariant,
so $g_F$ yields a natural transformation of the functor
\[
F:(Sch/k)^{opp}\rightarrow (Sets);\ Y\mapsto \Gr_n(X_\infty)(Y).
\]
Here we use that $F$ is a sheaf in the Zariski topoogy and hence it suffices to give maps  on affine schemes $Y=\Spec(A)$.
By Joneda's lemma we get an automorphism of the $k$-scheme
$\Gr_n(X_\infty)$. Doing the same construction for every $g\in G$ we obtain a group action of
$G$ on $\Gr_n(X_\infty)$.

Note that for $n=0$, the action on $\mathcal{R}_0(A)=A$ is trivial for all $k$-algebras $A$, and hence
the action on $\Gr_0(X_\infty)\cong X_0$ is just the action on $X_0$ induced by the action on $X_\infty$.
For any pair of integers $m\geq n\geq0$,
the truncation morphism $\theta^m_n$ is equivariant, since for
any $k$-algebra $A$ the natural morphism $\mathcal{R}_m(A)\rightarrow \mathcal{R}_n(A)$ is equivariant, see Remark~\ref{action rn},
and the same holds for $X_m \to X_n$ by construction of the group action.

As the maps $\theta_m^n$ are affine and $G$-equivariant,
a cover of $X_0$ by affine $G$-invariant open subsets gives rise to a similar cover of $\Gr_n(X_\infty)$ and $\Gr(X_\infty)$,
thus a good $G$-action on $X_\infty$ induces good
$G$-actions on the Greenberg schemes.

Now take a $G$-equivariant morphism of formal schemes $h:Y_\infty \to X_\infty$.
Then for every affine $k$-scheme $Y=\Spec(A)$,
$f\in \Gr_n(Y_\infty)(Y)=\Hom_{R_n}(\Spec(\mathcal{R}_n(A)),Y_\infty)$ and 
$g\in G$ we have that
$h_n\circ (g_{Y_n}\circ f \circ g_{\mathcal{R}_n(A)}^{-1})=g_{X_n}\circ (h_n\circ f)\circ g_{\mathcal{R}_n(A)}^{-1}$,
because $h$ is $G$-equivariant.
Hence the $\Gr_n(h)$ are $G$-equivariant, too.
Taking the limit gives us the same result for $\Gr(h)$.
\end{proof}

\begin{rem}\label{action on closed points}
Let $x$ be a point of $\Gr(X_\infty)$ corresponding canonically to an unramified
extension $R'$ of $R$, and a section $\psi$ in $X_\infty(R')$, see Remark \ref{R(F)}.
Since the residual action of $G$ on $k$ is trivial, the $G$-action
on $R$ extends canonically to a $G$-action on $R'$ which induces
the trivial action on the residue field.
For any $g\in G$, let $g_{R'}\in \Aut(\Spf(R'))$ and $g_{X_\infty}\in \Aut(X_\infty)$
be the corresponding automorphisms.
Then $g$ maps $x$
to the point corresponding to the section $g_{X_\infty}\circ \psi \circ g_{R' }^{-1} \in X_\infty(R')$.
\end{rem}

%
%

\subsection{The structure of the truncation maps}
\label{greenberg truncation}
The aim of this subsection is to study the truncation maps
on the Greenberg scheme under consideration of the  $G$-action constructed in the previous subsection.
All considered $G$-actions on Greenberg schemes are those constructed in Proposition \ref{Gaction}.

%
\begin{defn}\label{affine bundle}
Let $B$ be an $S$-scheme.
An \emph{affine bundle over $B$ of rank $d$} is a $B$-scheme $V$
with a vector bundle $E\to B$ of rank $d$ and a $B$-morphism $\varphi: E\times_BV\to V$ such that 
$\varphi\times p_V: E\times _BV \to V\times_B V$, where $p_V$ denotes the projection to $V$, is an isomorphism of $B$-schemes.
We call $E$ the \emph{translation space} of $V$.

An affine bundle $V$ over $B$ is called \emph{$G$-equivariant}, if $V$ and $B$ are in $(\Sch_{S,G})$, and $V\to B$ is $G$-equivariant.
The $G$-action on $V\to B$ is called \emph{affine} if there is a $G$-action on $E$, linear over the action on $B$,
such that $\varphi$ is $G$-equivariant.
An action on $E$ is \emph{linear over the action on $B$}
if for all $g\in G$ the
map $g':E \to g_B^*E$ induced by the following Cartesian diagram
\begin{equation*}\label{diagram1}
  \xymatrix{
 E\ar@/^/[drr]^{g_E}\ar@/_/[ddr]\ar@{-->}[rd]^{g'}\\
 &g_B^*(E)\ar[r]\ar[d]& E\ar[d]\\
 & B\ar[r]^{g_B}& B
 }
\end{equation*}
 is a morphism of vector bundles.
 Here $g_B\in \Aut(B)$ and $g_E\in \Aut(E)$ are the automorphisms of $B$ and $E$ induced by $g$.
\end{defn}

%

\noindent
For a discussion of equivariant affine bundles with affine group action
we refer to \cite[Section 3]{abi2}.
We are now using the definition to describe the truncation maps.

\begin{prop}\label{prop structure map}
Let $X_\infty\in (stft/R,G)$ be smooth of pure
relative dimension $m$ over $R$.
Then
for every integer $n\geq 0$, the truncation map
\[
\theta^{n+1}_n:\Gr_{n+1}(X_\infty)\rightarrow\Gr_{n}(X_\infty)
\]
is a $G$-equivariant affine bundle of rank $m$ with affine $G$-action.
\end{prop}

\begin{proof}
Locally on $\Gr_n(X_\infty)$, this proposition was shown in the non-equivariant case for example in \cite[Proposition 2.10]{MR2885336},
using \'etale covers.
In this proof we will use a proof of the non-equivariant case from \cite[Chapter 4, Theorem~2.4.4]{CNS},
because there the translation space is constructed explicitly using derivations.
We will start explaining this construction, and then construct a $G$-action on the translation space and examine it.
Note that all the steps in the prove which do not correspond to the $G$-action are taken from \cite[Chapter 4, Section 2]{CNS},
where one can also find more explanations and proofs.


\medskip
\noindent
\textit{Construction of the affine bundle structure.}
Let $\gamma: h_n(\Gr_n({X_\infty}))\to X_\infty$ be the morphism corresponding to the identity morphism on $\Gr_n(X_\infty)$,
see Remark \ref{rem hn}.
Consider the sheave
\[
(\Sch_{\Gr_n(X_\infty)}) \to (\text{Ab});\ (f:S\to \Gr_n(X_\infty))  \mapsto \Hom_{\mathcal{O}_{h_n(S)}}(h_n(f)^*\gamma^*\Omega_{X_\infty/R},\mathcal{J}_n),
\]
and
denote it by $\mathcal{J}_{X_\infty}^n$.
Here $h_n(f): h_n(S)\to h_n(\Gr(X_\infty))$ is the morphism induced by $f: S\to \Gr(X_\infty)$,
and $\mathcal{J}_n$ is the $\mathcal{R}_{n}$-module
defined in Section \ref{ideal scheme},
which becomes a sheaf of $\mathcal{O}_{h_n(S)}$-modules by tensoring it over $R_n$.

We will show later how this sheaf is represented by a vector bundle $V_{X_\infty}^n$ of rank~$m$.
Now we will explain 
on the level of sheaves the construction of a map 
${\varphi: V_{X_\infty}^n\times_{\Gr_n(X_\infty)}\Gr_{n+1}(X_\infty)\to \Gr_{n+1}(X_\infty)}$,
which makes $\Gr_{n+1}(X_\infty)$ an affine bundle of rank $m$ over $\Gr_n(X_\infty)$
with translation space $V_{X_\infty}^n$, see \cite[Chapter~4, Theorem 2.4.4]{CNS}.
Note that it is sufficient to give maps on sheaves over $\Gr_n(X_\infty)$ for affine $\Gr_n(X_\infty)$-schemes only.
Thus using that the truncation maps are affine,
we can without loss of generality replace $X_\infty$ by an open affine subspace $\Spf(B)$.
Note that we may assume that the action of $G$ on $X_\infty$ restricts to $\Spf(B)$,
because the action on $X_\infty$ is good by assumption,
and therefore $X_\infty$ is covered by $G$-invariant affine open subspace.

Let  $S=\Spec(A)$ be an affine point of $\Gr_n(X_\infty)$,
which corresponds to an $R$-morphism ${h: B\to \mathcal{R}_n(A)}$.
With this notation we have
\[
 \mathcal{J}_{X_\infty}^n(S)=\Hom_{\mathcal{R}_n(A)}(\Omega_{B/R}\otimes_{B,h} \mathcal{R}_n(A), \mathcal{J}_n(A)).
\]
Hence an element in $\mathcal{J}_{X_\infty}^n(S)$
corresponds to an $R$-derivation $D$ of $B$
with values in the $\mathcal{R}_n(A)$-module $\mathcal{J}_n(A)\subset \mathcal{R}_{n+1}(A)$
over the morphism $h:B\to \mathcal{R}_n(A)$.
This means by definition that for all $r\in R\subset B$ we have that $D(r)=0$,
and for all $b_1,b_2\in B$ we have
\[
 D(b_1+b_2)=D(b_1)+D(b_2) \text{ and } D(b_1b_2)=h(b_1)D(b_2)+h(b_2)D(b_1).
\]
Recall furthermore that
$ \Gr_{n+1}(X_\infty)(S)=\Hom_R(B,\mathcal{R}_{n+1}(A))$.
The required maps
\[
 \varphi(S):\mathcal{J}_{x_\infty}^n(S)\times\Gr_{n+1}(X_\infty)(S)\to \Gr_{n+1}(X_\infty)(S)
\]
are given by sending $(D,h')$ to $h'+D$.

\medskip
\noindent
\textit{Construction of the group action.}
Take any $g\in G$.
Denote by $g_B$ the corresponding automorphism of $B$,
and let $g_n$ be the corresponding automorphism of $\mathcal{R}_n(A)$ for all $n\geq 0$.
Denote by $g_n$ also the restriction of the $G$-action on $\mathcal{R}_n(A)$ to $\mathcal{J}_{n-1}(A)$.
Take any derivation $D\in \mathcal{J}_{X_\infty}^{ n}(S)$, and look at $\tilde{D}:=g_{n+1}^{-1} \circ D\circ g_B$.
For all $r\in R$ we have that
$\tilde{D}(r)=g_{n+1}^{-1} (D( g_B(r)))= g_{n+1}^{-1}(0)=0$.
This is due to the fact that $g_B(r)\in R$, because $R\to B$ is $G$-equivariant.
 As $g_{n+1}^{-1}$ and $g_B$ are ringhomomorphisms, and $D$ is a derivation, 
and hence all are additive,
$\tilde{D}$ is additive, too.
Moreover we have for $b_1,b_2\in B$ that
\begin{align*} 
\tilde{D}(b_1b_2)
&=g_{n+1}^{-1}(h(g_B(b_1))D(g_B(b_2))+h(g_B(b_2))D(g_B(b_1)))\\
&=\tilde{h}(b_1)\tilde{D}(b_2)+\tilde{h}(b_2)\tilde{D}(b_1)
\end{align*}
with $\tilde{h}:=g_{n+1}^{-1}\circ h \circ g_B$.
So $\tilde{D}$ is a derivation over $\tilde{h}$.
By Remark \ref{action jn}, we have that $\tilde{h}=g_n^{-1}\circ h \circ g_B$.
Note that the action on $\Gr_n(X_\infty)$ sends the $\Gr_n(X_\infty)$-scheme
$S=\Spec(A)$ which corresponds to the morphism
$h: B\to \mathcal{R}_n(A)$
to the $\Gr_n(X_\infty)$-scheme corresponding to the morphism $g_n^ {-1}\circ h \circ g_B=\tilde{h}$.
Hence $D\mapsto \tilde{D}$ gives rise to a well defined map from
$\mathcal{J}_{X_\infty}^n(S)$ to $\mathcal{J}_{X_\infty}^n(g_{\Gr_n}(S))$,
where $g_{\Gr_n}$ denotes the automorphism of $\Gr_n(X_\infty)$ corresponding to $g\in G$.

Doing the same for every affine scheme $Y$ over $\Gr_n(X_\infty)$,
we get a morphism of the sheaf $\mathcal{J}_{X_\infty}^n$, and hence an automorphism of the scheme $V_{X_\infty }^n$
representing it, over $g_{\Gr_n}$.
Doing the same for every $g\in G$,
we get a well defined $G$-action on ${V}_{X_\infty}^n$ over the $G$-action on $\Gr_n(X_\infty)$.
As 
\[
 g_{n+1}^{-1}\circ h'\circ g_B+g_{n+1}^{-1}\circ D \circ g_B =g_{n+1}^{-1}\circ (h'+D) \circ g_B
\]
for all $g\in G$, $\varphi$ is $G$-equivariant with the considered $G$-actions.

\medskip
\noindent
\textit{Description of the group action without derivations.}
Note now that for all $g\in G$ the ring homomorphism $g_B: B\to B$
induces an additive map ${g_\Omega : \Omega_{B/R}\to \Omega_{B/R}}$
by sending $b'db$ to $g_B(b')d(g_B(b))$,
with $b,b'\in B$. 
Here $d: B\to \Omega_{B/R}$ is the canonical map, which is $G$-equivariant by construction.
Let $D$ be again a derivation of $B$ with image in $\mathcal{J}_n(A)$ over $h$, and
$\tilde{D}=g_{n+1}^{-1}\circ D \circ g_B$.
By the universal property of $\Omega_{B/R}$ there exists a unique morphism of $B$-modules
$f:\Omega_{B/R}\to \mathcal{J}_n(A)$ with $f\circ d =D$, and a unique morphism of $B$-modules
$\tilde{f}:\Omega_{B/R}\to \mathcal{J}_n(A)$ with $\tilde{f}\circ d =\tilde{D}$.
Note that in the first case $\mathcal{J}_n(A)$
is a $B$-module via $h$,
and in the second case via $\tilde{h}$.
Set $\tilde{f}':=g_{n+1}^{-1}\circ f \circ g_\Omega$,
which is additive, and for
all $b\in B$ and $\omega\in \Omega_{B/R}$
\[
 \tilde{f}'(b\omega)=g_{n+1}^{-1}\circ f(g_B(b)g_\Omega(\omega))=g_{n+1}^{-1}((h \circ g_B)(b)(f\circ g_\Omega )(\omega))=\tilde{h}(b)\tilde{f}'(\omega).
\]
Hence $\tilde{f}'$ is a morphism of $B$-modules (via $\tilde{h}$).
Moreover we have
\[
 \tilde{f}'\circ d =(g_{n+1}^{-1}\circ f \circ g_\Omega)\circ d =\tilde{D}.
\]
As $\tilde{f}$ is unique with this properties, it follows that $\tilde{f}=g_{n+1}^{-1}\circ f \circ g_\Omega$.
Let
\[
g_\Omega \otimes g_n : \Omega_{B/R}\otimes _{B,\tilde{h}} \mathcal{R}_n(A)\to \Omega_{B/R}\otimes _{B,{h}} \mathcal{R}_n(A)
\]
be the map given by sending $\omega\otimes_{B,\tilde{h}}a$ to $g_\Omega(\omega)\otimes_{B,{h}} g_n(a)$.
With this notation the $G$-action on $\mathcal{J}^ n_{X_\infty}$ is given as follows:
\begin{align*}
\Hom_{\mathcal{R}_n(A)}(\Omega_{B/R}\otimes _{B,h} \mathcal{R}_n(A), \mathcal{J}_n(A))&\to \Hom_{\mathcal{R}_n(A)}(\Omega_{B/R}\otimes _{B,\tilde{h}} \mathcal{R}_n(A), \mathcal{J}_n(A));\\
f&\mapsto g_{n+1}^{-1}\circ f\circ (g_\Omega \otimes g_n).
\end{align*}

\medskip
\noindent
\textit{The vector bundle structure.} 
We now explain the construction of the vector bundle $V_{X_\infty}^n$ representing $\mathcal{J}_{X_\infty}^n$,
see \cite[Chapter 4, 2.4.3]{CNS}.
We do this under consideration of the constructed group action.
Therefore, we restrict ourselves again to the case that $X_\infty=\Spf(B)$ is affine,
which implies that also $\Gr_n(X_\infty)=\Spec(C)$ and $X_0=\Spec(B_0)$ are affine.

As the $\mathcal{R}_{n+1}$-module structure of $\mathcal{J}_n$ factors through the quotient $\mathcal{R}_0$,
we actually have for every affine $\Gr_n(X_\infty)$-scheme $S=\Spec(A)$ with structure map $f$ that
\[
\mathcal{J}_{X_\infty}^n(S)=\Hom_{\mathcal{O}_S}(f^*{\theta_0^n}^*\Omega_{X_0/k},\mathcal{J}_n)=\Hom_{A}(\Omega_{B_0/k}\otimes_{B_0,h_0}A,\mathcal{J}_n(A)).
\]
Here $X_0=\Spec(B_0)$,
and $h_0: A\to B_0$ is the map corresponding to $\theta_0^n\circ f$.
For all $g\in G$,
denote by $g_B$ also the restriction of $g_B$ to $B_0$,
and by $g_\Omega$ also the restriction of $g_\Omega $ to $\Omega_{B_0/k}$. 
After restricting all involved maps,
the $G$-action on $\mathcal{J}^ n_{X_\infty}$ is given
by sending $f\in \mathcal{J}^n_{X_\infty}(S)$
to 
\[
g_{n+1}^{-1}\circ f\circ (g_\Omega\otimes \Id)\in \mathcal{J}^n_{X_\infty}(g_{\Gr_n}(S))=\Hom_A(\Omega_{B_0/k}\otimes _{B_0,g_B\circ {h_0}}A, \mathcal{J}_n(A)).
\]
Recall that
by Remark \ref{gaction Psi}, there is a $G$-equivariant isomorphism between the $A$-modules
$ \mathcal{J}_n(A)$ and $  V(\beta)\otimes_k {\mathcal{F}^R_A} _*(A)$.
Here
$\mathcal{F}^R_A$ is the identity if $R$ has equal characteristic,
and some power depending on $R$ of the absolute Frobenius on $A$ if $R$ has mixed characteristic,
and $V(\beta)$ is a one-dimensional vector space.
Moreover, the $G$-action on $V(\beta)\otimes_k {\mathcal{F}^R_A} _*(A)$ is given by automorphisms $ g_V\otimes \Id$ for all $g\in G$,
where $g_V$ a linear map on $V(\beta)$.
Using this $G$-equivariant isomorphism, we get that
\begin{align*}
 \mathcal{J}_{X_\infty}^n(S)&\cong\Hom_A(\Omega_{B_0/k}\otimes_{B_0,h_0}A,V(\beta)\otimes_k {\mathcal{F}^R_A} _*(A))\\
 &\cong\Hom_A(V(-\beta)\otimes_k\Omega_{B_0/k}\otimes_{B_0,h_0}A,{\mathcal{F}^R_A} _*(A)),
\end{align*}
where $V(-\beta)$ is the dual of $V(\beta)$.
We get from the first to the second line by using the isomorphism $l$, which
sends $f$ with
$f(\omega \otimes a)=f_V(\omega\otimes a)\otimes f_A(\omega\otimes a)$
to
$l(f)$ with 
$l(f)(v\otimes \omega \otimes a)=v(f_V(\omega\otimes a))f_A(\omega\otimes a)$.
Note that if the $G$-action on the last is given by sending
$\tilde{f}$ to $\tilde{f}\circ ({g_V^*}^{-1}\otimes g_\Omega \otimes \Id)$,
where $g_V^*$ is the dual morphism of $g_V$ for all $g\in G$, then $l$ is $G$-equivariant.

To get rid of the Frobenius, we pull back both sides via $\mathcal{F}^R_A$, and get that
\[
 \mathcal{J}_{X_\infty}^n(S)=\Hom_{A}(V(-\beta)\otimes_k \Omega_{X_0/k}^R\otimes_{B_0,h_0}A,A),
\]
with $\Omega^R_{X_0/k}:={\mathcal{F}_{X_0}^R}^*\Omega_{X_0/k}=\Omega_{X_0/k}\otimes_{B_0,\mathcal{F}^R_{B_0}}B_0$.
Denote by $g_\Omega$ also the automorphism of $\Omega_{X_0/k}^R$ we get by pulling back $g_\Omega$ via $\mathcal{F}^R_k$.
Using that
by \cite[Lemma~3.2.22.]{MR1917232} the absolute Frobenius commutes with morphism of schemes over $\mathbb{F}_p$,
it is also given by sending $b\omega$ to $g_B(b)g_\Omega(\omega)$.
Recall that we assumed that
$\Gr_n(X_\infty)=\Spec(C)$ for some $C$,
and denote by ${\tau: B_0\to C}$ the map induced by $\theta_0^n$,
and by  $c_A:C\to A$ the map induced by $f$.
For all $g\in G$, denote by $g_C$ the automorphism of $C$ induced by the $G$-action on $\Gr_n(X_\infty)$.
With this notation, the $G$-action on $\mathcal{J}_{X_\infty}^n$ is given as follows:
\begin{align*}
J_{X_\infty}^n(S)=&\Hom_{A}(V(-\beta)\otimes_k\Omega_{B_0/k}^R\otimes_{B_0,\tau} C \otimes_{C,c_A} A, A)\\
&\to J_{X_\infty}^n(g_{\Gr_n}(S))=\Hom_A(V(-\beta)\otimes_k \Omega_{B_0/k}^R\otimes _{B_0,\tau }C\otimes_{C, c_A\circ g_C} A, A);\\
f&\mapsto f\circ ({g_V^*}^{-1}\otimes g_\Omega\otimes g_C \otimes \Id).
\end{align*}
Note that we used that $\tau$ is actually $G$-equivariant.

As in \cite[Proposition~1.7.11]{MR0217084},
let $V$ be the contravariant functor sending a quasi-coherent $\mathcal{O}_{\Gr_n(X_\infty)}$-module $\xi$
to the affine $\Gr_n(X_\infty)$-scheme
$\Spec(\Sym(\xi))$, where $\Sym(\xi)$ is the symmetric $\mathcal{O}_{\Gr_n(X_\infty)}$-algebra,
see \cite[1.7.4]{MR0217084}.
Using this notation we get that
$ V_{X_\infty}^n=V(V(-\beta)\otimes_k\Omega^R_{B_0/k}\otimes_{B_0,\tau}C)$
represents $\mathcal{J}_{X_\infty}^n$.
As $X_\infty$ is smooth of relative dimension $m$ over $R$,
$\Omega^R_{B_0/k}$ is a locally free sheaf or rank $m$,
and hence the same holds for $ V(-\beta)\otimes_k\Omega^R_{B_0/k}\otimes_{B_0,\tau}C$. 
Thus $V_{X_\infty}^n$ is a vector bundle of rank $m$. 
Moreover, for every $g\in G$ the automorphism of the scheme $V_{X_\infty}^n$
is given by the automorphism ${g_V^*}^{-1}\otimes g_\Omega \otimes g_C$
on the corresponding $C$-module 
${V(-\beta)\otimes_k\Omega_{X_0/k}^R\otimes _{B_0,\tau}C}$.

\medskip
\noindent
\textit{The group action is affine.}
In order to check whether the considered action on $V_{X_\infty}^n$ is linear over the base,
we tensor $V_{X_\infty}^n$ with $\Gr_n(X_\infty)$ over $g_{\Gr_n}$ for every $g\in G$.
By \cite[Proposition~1.7.11]{MR0217084}
we have that 
\begin{align*}
g_{\Gr_n}^*(V_{X_\infty}^n)&= V_{X\infty}^n\times_{\Gr_n(X_\infty)}\Gr_n(X_\infty)=V(V(-\beta)\otimes \Omega^R_{B_0/k}\otimes_{B_0,\tau}C\otimes_{C,g_{C}} C)\\
&=V(V(-\beta)\otimes_k\Omega_{X_0/k}^R\otimes _{B_0,g_B}B_0\otimes_{B_0,\tau}C).
\end{align*}
Here we used again that $\tau$ is $G$-equivariant.
The induced map ${g':V_{X_\infty}^n\to g_{\Gr_n}^*(V_{X_\infty}^n)}$, see Definition \ref{affine bundle},
is given on the level of $C$-modules by
the map $\tilde{g}$ defined by ${\tilde{g}(v\otimes \omega\otimes b\otimes c)= {g_V^*}^{-1}(v)\otimes g_{\Omega}(\omega)\otimes b\otimes c}$. 
We now want to show that $\tilde{g}$ is a morphism of $C$-modules. Therefore it suffices to show that
\begin{align*}
\tilde{g}': V(-\beta)\otimes_k \Omega^R_{B_0/k}\otimes_{B_0,g_{B}}B_0&\to V(-\beta)\otimes_k\Omega^R_{B_0/k}\otimes_{B_0,\Id}B_0;\\
 v\otimes \omega\otimes b  &\mapsto {g_V^*}^{-1}(v)\otimes g_{\Omega}(\omega)\otimes b
\end{align*}
is a morphism of $B_0$-module.
Let $v_0\in V(-\beta)$ be a basis of this vector space. Hence for every element $v\in V(-\beta)$ there is a $\bar{v}\in k$ such that $v=\bar{v}v_0$.
Take now any $v_1=\bar{v}_1v_0,v_2=\bar{v}_2v_0\in V(-\beta)$, $\omega_1,\omega_2\in \Omega_{B_0/k}^R$, and $b_1,b_2\in B_0$.
Then we have
\begin{align*}
 \tilde{g}'(v_1\otimes_k \omega_1 &\otimes_{B_0,g_{B}}b_1 + v_2\otimes_k \omega_2 \otimes_{B_0,g_{B}}b_2)&\\
 &=  \tilde{g}'(v_0\otimes_k  (\bar{v}_1g_B^{-1}(b_1)\omega_1+ \bar{v}_2g_B^{-1}(b_2)\omega_2)\otimes_{B_0,g_{B}}1)\\
 &={g_v^*}^{-1}(v_0)\otimes_k g_\Omega(\bar{v}_1g_B^{-1}(b_1)\omega_1+ \bar{v}_2g_B^{-1}(b_2)\omega_2)\otimes_{B_0,\Id}1\\
 &={g_v^*}^{-1}(v_0)\otimes_k (\bar{v}_1b_1g_\Omega(\omega_1)+ \bar{v}_2b_2g_\Omega(\omega_2))\otimes_{B_0,\Id}1\\
 &=\tilde{g}'(v_1\otimes_k \omega_1 \otimes_{B_0,g_{B}}b_1) + \bar{g}'(v_2\otimes_k \omega_2 \otimes_{B_0,g_{B}}b_2).
\end{align*}
Hence $\tilde{g}'$ is additive. 
It is clear that $\tilde{g}'$ is multiplicative in $B_0$, hence it is a morphism of $B_0$-modules, and
thus $\tilde{g}$ is a morphism of $C$-modules.
Note that by \cite[1.7.14]{MR0217084}
a morphism of $C$-modules $\xi\to \xi'$,
corresponds to a morphism  of $C$-algebras
$\Sym(\xi)\to \Sym(\xi')$.
This implies that the corresponding maps of schemes between 
$V(\xi)=\Spec(\Sym(\xi))$ and $V(\xi')$ is a morphism of vector bundles.
To show this, one uses the construction of the vector bundle structure given in \cite[1.7.10]{MR0217084}.
Thus $\tilde{g}$ corresponds to a morphism of vector bundle, and hence
the action on $V_{X_\infty}^n$ is linear over that on $\Gr_n(X_\infty)$.
Hence altogether $\Gr_{n+1}(X_\infty)$ is an affine bundle over $\Gr_n(X_\infty)$ with translation space $V_{X_\infty}^n$ and affine $G$-action.

\end{proof}

\subsection{Greenberg schemes and equivariant morphisms of formal schemes}
\label{greenberg h}

Throughout this subsection, assume that $R$ has equal characteristic.
Moreover, fix a morphism ${h:Y_\infty\rightarrow X_\infty}$ of
flat $stft$ formal $R$-schemes, both of pure relative
dimension $m$ over $R$.

The aim of this section is to examine the induced map $\Gr_n(h)$ on the corresponding Greenberg schemes with respect to the induced $G$-action
constructed in Proposition \ref{Gaction}.
Before we can state the main result, we first need to introduce the order of the Jacobian of $h$,
as defined for example in \cite[Chapter 4, Definition~3.1.2]{CNS}.

\begin{defn}\label{def jac}
We define the \emph{Jacobian ideal} $\Jac_h\subset \mathcal{O}_{Y_\infty}$ as the $0$-th fitting ideal 
of the sheaf of relative differential forms $\Omega_{Y_\infty/X_\infty}$.
If $X_\infty$ and $Y_\infty$ are smooth over $R$,
${\Jac}_h$ is generated by the determinant of the map
\[
 h^*\Omega_{X_\infty/R}\to \Omega_{Y_\infty/R}.
\]
This holds, because if $X_\infty$ and $Y_\infty$ are both smooth over $R$,
then the modules of differentials are free of rank $d$,
hence the map above defines a free resolution of $\Omega_{Y_\infty/X_\infty}$.
For the general definition of fitting ideals we refer to \cite[Corollary-Definition~20.4]{MR1322960}
\end{defn}

\begin{defn}\label{def ordjac}
Let $y\in\Gr(Y_\infty)$ be any point with residue field $F$,
and
let $\psi$ be the corresponding element in $Y_\infty(R')$,
$R':=\mathcal{R}(F)$, which is a complete discrete valuation ring with residue field $F$
and ramification index one over $R$, see Remark~\ref{R(F)}.
Denote by $\xi: R'\to \mathbb{N}\cup \{\infty\}$ the valuation map,
and let $\psi(0)$ be the image of the unique point of $\Spf(R)$.
Then $\Ord(\Jac_h)$, the \emph{order of the Jacobian of $h$},
is the function sending a point $y\in \Gr(Y_\infty)$ to
\[
 \Ord(\Jac_h)(y)=\min\{\xi(\psi^* (f))\mid f\in ({\Jac_h})_{\psi(0)}\}.
\]
\end{defn}


\begin{rem}\label{compute ordjac}
Assume that $X_\infty$ and $Y_\infty$ are smooth over $R$.
Let $R'$ be a unramified extension of $R$ with
residue field $F$,
and fix a section $\psi$ in
$Y_\infty(R')$.
The canonical morphism
$h^*\Omega^m_{X_\infty/R}\rightarrow \Omega^m_{Y_\infty/R}$
induces a morphism of free rank one $R'$-modules
\[
\psi^*h^*\Omega^m_{X_\infty/R}\rightarrow \psi^*\Omega^m_{Y_\infty/R}.
\]
By definition of the Fitting ideal, this map is just multiplying with the generator $a$ of the fitting ideal of $\psi^*\Omega_{Y_\infty/X_\infty}$.
If $\psi$ corresponds to a point $y\in \Gr(Y_\infty)$ such that
$\Ord(\Jac_h)(y)=e$,
then $a$ has valuation $e$ by the definition of the order of $\Jac_h$, and hence
$e$ is equal to the length of the cokernel of this map.

If $Y_\infty$ is only generically smooth, $\psi^*h^*\Omega^m_{X_\infty/R}$
might have torsion elements.
In this case $e=\Ord(\Jac_h)(y)$ is the length of the cocernel of the map of free rank one $R'$-modules
$\psi^*h^*\Omega^m_{X_\infty/R}/(\text{torsion})\rightarrow \psi^*\Omega^m_{Y_\infty/R}$.
\end{rem}

\noindent
Using the order of the Jacobian of $h$, we have the following proposition describing the structure
of $\Gr_n(h)$ for big enough $n$:

\begin{prop}\label{structure h}
Let $R$ be a complete discrete valuation ring of equal characteristic with residue field $k$ containing all roots of unity,
endowed with a nice and tame action of a finite abelian group $G$.
Take $Y_\infty, X_\infty \in (stft/R,G)$ smooth over $R$,
and a $G$-equivariant morphism $h:Y_\infty \to X_\infty$ of formal schemes,
which is generically an open
embedding.

Take any $y\in \Gr(Y_\infty)$ with $\Ord(\Jac_h)(y)=e$,
and set $x_n:=\theta_n(\Gr(h)(y))$ in $\Gr_n(X_\infty)$ for some $n\geq 2e$.
Denote by $G_x\subset G$ the stabilizer of $x_n$,
and let $B_{x_n}$ be the reduced subscheme of $\Gr_n(h)^{-1}(x_n)$.
Then
\[
\Gr_n(h):B_{x_n}\rightarrow x_n
\]
is a $G_x$-equivariant affine bundle of rank
$e$ with affine
$G_x$-action.
\end{prop}

\begin{proof}
Take any point $x_n=\Spec(F)\in \Gr_n(X_\infty)$ as in the claim with stabilizer $G_x$.
To simplify the notation, we assume that $G_x=G$,
hence $G$ acts in particular on the reduced subscheme $B_{x_n}$ of $\Gr_n(h)^{-1}(x_n)$.
In \cite[Lemme 7.2.2]{sebag1}, it was shown that $B_{x_n}\cong \mathbb{A}_F^e$.
In our proof, we use the construction of a concrete affine bundle structure
in 
\cite[Chapter 4, Theorem~3.2.2]{CNS} using derivations,
to construct a $G$-equivariant isomorphism of $F$-schemes  $\varphi: \mathbb{A}_F^e\to B_{x_n}$ such that the $G$-action on $\mathbb{A}^e_F$
is linear over the action on $x_n$.
Note that the steps in the construction  which do not concern the $G$-action are mainly taken from \cite[Chapter 4, Section~3]{CNS}.

\medskip
\noindent
\textit{Construction of $\varphi$.}
Set $x=\Gr(h)(y)$. Then $\theta_n(x)=x_n$.
For all $m\in \mathbb{N}$, set $x_m:=\theta_m(x)$ and $y_m:=\theta_m(y)$.
By \cite[Lemme 7.2.2]{sebag1}
every point in $B_{x_n}$
is mapped to $y_{n-e}$ under $\theta^n_{n-e}$. Hence $B_{x_n}$
is a closed subset of ${\theta_{n-e}^n}\!\!\!^{-1}(y_{n-e})$.
%

Note first that $\Gr_n(Y_\infty)$ only depends on $Y_n$, as well as
$\Gr_n(X_\infty)$ only depends on $X_n$.
Due to the local nature of the claim, 
we may replace $X_n$ by an affine $G$-invariant neighborhood $U$ of $x_0=\theta_0^n(x_n)\in \Gr_0(X_\infty)=X_0\subset X_n$,
which exists because the action of $G$ on $X_n$ is good. 
Moreover we may replace $Y_n$ by an affine subset $V$ containing $\theta_0^n(B_{x_n})=y_0\in \Gr_0(Y_\infty)=Y_0\subset Y_n$
of the intersection of $h^{-1}(U)$ and an affine subset of $Y_n$ containing $y_0$.
Such a $V$ exists due to \cite[Proposition~3.6.5]{MR1917232}.
Replacing $V$ by $\cap _{g\in G}g(V)$ we may assume that $V$ is $G$-invariant.
Hence from now on we assume that $Y_n=\Spec(B)$ and $X_n=\Spec(C)$ are affine.

We will now describe ${\theta_{n-e}^n}\!\!\!^{-1}(y_{n-e})$.
Let ${\gamma: \Spec( \mathcal{R}_{n-e}')\to Y_\infty}$ be the morphism corresponding to
$y_{n-e}=\Spec(F')\in\Gr_{n-e}(Y_\infty)$, where $R_{m}':=\mathcal{R}_m(F')$.
Consider the sheaf
\[
(\Sch_{F'}) \to (\text{Ab});\ (f:S\to \Spec(F'))  \mapsto \Hom_{\mathcal{O}_{h_{n-e}(S)}}(h_{n-e}(f)^*\gamma^*\Omega_{Y_\infty/R},\mathcal{J}_{n-e}^{n}),
\]
and denote it by $\mathcal{J}_{y_{n-e}}^{n,n-e}$.
We will now construct a $G$-action on $\mathcal{J}_{y_{n-e}}^{n,n-e}$.
This construction works analogously to that in the proof of Proposition~\ref{prop structure map},
so we will be rather short on this.
Again we may give maps only for affine $y_{n-e}$-schemes $S=\Spec(A)$.
Recall that $Y_n=\Spec(B)$ is affine, and denote by $\tau: B\to {R}_{n-e}'$ the morphism of rings corresponding to~$\gamma$.
With this notation we have 
\[
\mathcal{J}_{y_{n-e}}^{n,n-e}(S)=\Hom_{{R}_{n-e}'}(\Omega_{B/R}\otimes_{B,\tau}{R}_{n-e}',\mathcal{J}_{n-e}^n(A)).
\]
For every $g\in G$, let
$g_n\in \Aut(\mathcal{J}^n_{n-e}(A))$,
$g_{n-e}\in\Aut({R}_{n-e}')$,
and $g_B\in\Aut(B)$
be the corresponding automorphisms.
We can define 
a map $g_\Omega: \Omega_{B/R}\to \Omega_{B/R}$
sending $b'db$ to $g_B(b')d(g_B(b))$,
with $b,b'\in B$ and ${d:B\to \Omega_{B/R}}$ the canonical map.
Consider the action on $\mathcal{J}_{y_{n-e}}^{n,n-e}$ given by sending $ f\in \mathcal{J}_{y_{n-e}}^{n,n-e}(S)$ to
\begin{align*}
g_n ^{-1}\circ f\circ (g_\Omega \otimes g_{n-e})\ \in & \ \mathcal{J}_{y_{n-e}}^{n,n-e}(g_{\Gr}(S))\\
&=\Hom_{{R}_{n-e}'}(\Omega_{B/R}\otimes_{B,g_{n-e}^{-1}\circ\tau\circ g_B}{R}_{n-e}',\mathcal{J}_{n-e}^n(A))
\end{align*}
for all affine $y_{n-e}$-schemes $S=\Spec(A)$.
Here $g_{\Gr}$ denotes the automorphism of $y_{n-e}$ induced by the automorphism of $\Gr_{n-e}(Y_\infty)$
 corresponding to $g\in G$.
Exactly as done in the proof of Proposition~\ref{prop structure map} on can show that
doing so we get a well defined
 $G$-action on $\mathcal{J}_{y_{n-e}}^{n,n-e}$
 over the action on $y_{n-e}$.
 Note that we also have a map of sheaves
 \[
 \varphi_{Y_\infty}: \mathcal{J}_{y_{n-e}}^{n,n-e}\times_{y_{n-e}}{\theta_{n-e}^e}\!\!\!^{-1}(y_{n-e}) \to {\theta_{n-e}^e}\!\!\!^{-1}(y_{n-e}),
\]
making ${\theta_{n-e}^n}\!\!\!^{-1}(y_{n-e})$ a principal homogenous space, see \cite[Chapter 4, Proposition 2.4.2]{CNS}.
This map is constructed as $\varphi$ in Proposition \ref{prop structure map},
and with the same proof as there one can show that
$\varphi_{Y_{\infty}}$ is $G$-equivariant with the considered $G$-actions.

\medskip
\noindent
Note that
we can do the same construction also for ${\theta_{n-e}^n}\!\!\!^{-1}(x_{n-e})\subset \Gr_n(X_\infty)$.
Moreover,
we can define a
$G$-equivariant map
$i:  \mathcal{J}_{y_{n-e}}^{n,n-e} \to  \mathcal{J}_{x_{n-e}}^{n,n-e}$
as follows:
recall that $X_n=\Spec(C)$ and $Y_n=\Spec(B)$
are affine, and 
denote by $\psi:C\to B $ the ringmorphism corresponding to $h_n:=h\rvert_{Y_n}$.
Let
$i_\Omega: \Omega_{C/R}\to \Omega_{B/R}$ be the map given by sending $c'dc$
to $\psi(c')d(\psi(c))$. This map is $G$-equivariant for the considered actions on $\Omega_{C/R}$ and $\Omega_{B/R}$,
because $\psi$ is $G$-equivariant by assumption.
Let $S=\Spec(A)$ be again an affine $y_{n-e}$-scheme.
Then $i$ is given by sending
\begin{align*}
f&\in \mathcal{J}_{y_{n-e}}^{n,n-e}(S)=\Hom_{{R}_{n-e}'}(\Omega_{B/R}\otimes_{B,\tau}{R}_{n-e}',\mathcal{J}_{n-e}^n(A))\text{ to}\\
f\circ (i_\Omega \otimes \Id)&\in  \mathcal{J}_{x_{n-e}}^{n,n-e}(S)=\Hom_{{R}_{n-e}'}(\Omega_{C/R}\otimes_{C,\tau\circ \psi}{R}_{n-e}',\mathcal{J}_{n-e}^n(A)).
\end{align*}
Note that $\tau\circ \psi$ is the ring morphism corresponding to $h\circ \gamma: \Spec(R_{n-e}')\to X_\infty$,
which is corresponding to the point $x_{n-e}\in \Gr_{n-e}(X_\infty)$.
As $i_\Omega$ is $G$-equivariant, $i$ is $G$-equivariant for the considered $G$-actions.
Altogether we get the following commutative diagram.
\[
 \xymatrix{
 \mathcal{J}_{y_{n-e}}^{n,n-e}\times_{y_{n-e}} {\theta_{n-e}^e}\!\!\!^{-1}(y_{n-e})\ar[r]^{\ \ \ \ \ \ \varphi_{Y_\infty}} \ar[d]^{i\times \Gr_n(h)} & {\theta_{n-e}^n}\!\!\!^{-1}(y_{n-e})\ar[d]^{\Gr_{n}(h)} \\
 \mathcal{J}_{x_{n-e}}^{n,n-e}\times_{x_{n-e}}{\theta_{n-e}^n}\!\!\!^{-1}(x_{n-e}) \ar[r]^{\ \ \ \ \ \ \varphi_{X_\infty}}& {\theta_{n-e}^n}\!\!\!^{-1}(x_{n-e})
 }
\]
Note that all the maps are $G$-equivariant.
As $B_{x_n}\subset{\theta_{n-e}^n}\!\!\!^{-1}(y_{n-e}) $ is $G$-invariant
and mapped to the fixed point
 $x_n\in{\theta_{n-e}^n}\!\!\!^{-1}(x_{n-e}) $,
we can restrict this diagram to get the following diagram, which is  still $G$-equivariant:
\[
 \xymatrix{
 \mathcal{J}_{y_{n-e}}^{n,n-e}\times_{y_{n-e}} B_{x_n}\ar[r]^{\varphi_{Y_\infty}} \ar[d]^{i\times \Gr_n(h)} & {\theta_{n-e}^n}\!\!\!^{-1}(y_{n-e})\ar[d]^{\Gr_{n}(h)} \\
 \mathcal{J}_{x_{n-e}}^{n,n-e}\times_{x_{n-e}} x_n \ar[r]^{\varphi_{X_\infty}}& {\theta_{n-e}^n}\!\!\!^{-1}(x_{n-e})
 }
\]

\medskip
\noindent
Note that a point of the scheme $\mathcal{J}_{x_{n-e}}^{n,n-e}\times_{y_{n-e}} B_{x_n}$ lies in the inverse image of $B_{x_n}$ in $ {\theta^n_{n-e}}\!\!\!^{-1}(y_{n-e})$
if and only if it is mapped to $(0,x_n)$ by $i\times \Gr_n(h)$.
Hence in order to describe $B_{x_n}$, we need to describe the kernel of $i$.
Denote the corresponding subsheaf of $\mathcal{J}_{y_{n-e}}^{n,n-e}$ by $E$.
Note that $f\in \mathcal{J}^{n,n-e}_{y_{n-e}}(A)$ lies in the kernel of $i$
if and only if for all $c\in \psi(C)\subset B$ we have that $f(dc)=0$.
Hence we get 
\[
E(S)=\Hom_{{R}_{n-e}'}(\Omega_{B/C}\otimes_{B,\tau}{R}_{n-e}',\mathcal{J}_{n-e}^n(A))
\]
for all affine $y_{n-e}$-schemes $S=\Spec(A)$.
Consider now the map
$q_\Omega$ which maps $b'db\in \Omega_{B/R}$ to $b'db\in \Omega_{B/C}$.
 The inclusion map $E\hookrightarrow \mathcal{J}_{y_{n-e}}^{n,n-e}$ is given by sending $f\in E(S)$
 to $f\circ q_\Omega\otimes \Id$ for all $y_{n-e}$-schemes $S$.
 
Consider the $G$-action on $\Omega_{B/C}$
given by sending $b'db$ to $g_B(b')d(g_B(b))$ for all $g\in G$.
Denote these maps also by $g_\Omega$.
These are well defined, because $\psi$ is $G$-equivariant.
By construction, $q_\Omega$ is
$G$-equivariant with the considered $G$-actions.
Let $G$ act on $E$ by sending for all $g\in G$, $f\in E(S)$
to 
$g_n^{-1}\circ f \circ (g_\Omega \otimes g_{n-e})$ in $E(g_{\Gr}(S))$.
Then the inclusion $E\hookrightarrow \mathcal{J}_{y_{n-e}}^{n,n-e}$ is $G$-equivariant with this $G$-action,
thus like this we can describe the induced $G$-action on $E$ over the action on $y_{n-e}$.

\medskip
\noindent
Next we observe that there is always a fixed point $\tilde{y}_{n}=\Spec(F)$ in $B_{x_n}$:
by Proposition~\ref{prop structure map}, ${\theta_n^{n+1}:{\theta^{n+1}_n}^{-1}(x_n)\to x_n}$ is an affine bundle over $x_n=\Spec(F)$
with affine $G$-action.
Hence by \cite[Remark 3.7]{abi2}, there exists a point ${\tilde{x}_{n+1}=\Spec(F)}$ in ${\theta^{n+1}_n}^{-1}(x_n)$
which is fixed by the action of $G$.
To get this fixed point, one needs to assume that the action of $G$ is tame.
Using an induction argument, we get a fixed point
$\tilde{x}_{n+e}=\Spec(F)$ in ${\theta^{n+e}_n}^{-1}(x_n)\subset\Gr_{n+e}(X_\infty)$.

Restricting $X_\infty$ to a suitable open formal subscheme,
we may assume that $h$ is generically an isomorphism,
hence by \cite[Lemma 2.4.1]{MR2885338}, $\Gr(h)$ is surjective.
As in addition the truncation maps are surjective, because $Y_\infty$ and $X_\infty$ are smooth,
$\Gr_{n+e}(h)^{-1}(\tilde{x}_{n+e})$ is not empty.
By \cite[Chapter 4, 3.2.4]{CNS}, $\Ord(\Jac_h)$ is constant on connected components of $Y_\infty$,
which implies that $\tilde{x}_{n+e}$ also lies in the image of a point $\tilde{y}\in \Gr(Y_\infty)$ with $\Ord(\Jac_h)(\tilde{y})=e$.
Hence by \cite[Lemme~7.2.2]{sebag1}, $\Gr_{n+e}(h)^{-1}(\tilde{x}_{n+e})$ is mapped to exactly one point $\tilde{y}_{n}=\Spec(F)$,
which lies in $B_{x_n}\subset \Gr_n(Y_\infty)$.
As $\tilde{x}_{n+e}$ is a fixed point and $\Gr_{n+e}(h)$ is $G$-equivariant,
the action of $G$ on $\Gr_{n+e}(Y_\infty)$ restricts to
$\Gr_{n+1}(h)^{-1}(\tilde{x}_{n+1})$.
As moreover $\theta_n^{n+e}$ is $G$-equivariant, $\tilde{y}_{n}$ is a fixed point.

Now we can restrict $\varphi_{Y_\infty}$ to $E\times_{y_{n-e}}\tilde{y}_n$,
and get a $G$-equivariant morphism $\varphi: E\times_{y_{n-e}}\tilde{y}_n\to B_{x_n}$ over $x_n=\Spec(F)$,
which is an isomorphism, because $\varphi_{Y_\infty}$ is a formally principal homogeneous space.
Note that if $E$ is isomorphic to $\mathbb{A}^e_{F'}$ and the action on $E$ is linear over the action on $y_{n-e}=\Spec(F')$,
then $E\times_{y_{n-e}}\tilde{y}_n$ is isomorphic to $\mathbb{A}^e_{F}$ and the action on $E\times_{y_{n-e}}\tilde{y}_n$ is linear over the action on $\tilde{y}_{n}=\Spec(F)$.
Hence from now on we assume that $y_{n-e}\cong \tilde{y}_n=\Spec(F)$, i.e.~in particular $F=F'$,
and $E=E\times_{y_{n-e}}\tilde{y}_n$.

\medskip
\noindent
 \textit{The vector space structure of $E$.}
Now we recall the construction of the isomorphism $E\to \mathbb{A}_F^e$
from \cite[Chapter~4, Theorem~3.2.2]{CNS}.
As $E$ is isomorphic to $B_{x_n}$ which is reduced,
it suffices to give this isomorphism onreduced affine $F$-schemes $S=\Spec(A)$.
Let $\gamma: \Spec(R_{n-e}')\to Y_\infty$
be again the section corresponding to $y_{n-e}\in \Gr_{n-e}(Y_\infty)$.
As $\Ord(\Jac_h)(y)=e$,
it follows 
that $\gamma^*\Omega_{Y_\infty/X_\infty}$
is a $R_{n-e}'$-module of length $e$.
Hence we can fix an isomorphism of $R_{n-e}'$-modules
\begin{align}\label{j}
 j: \gamma^*\Omega_{Y_\infty/X_\infty}\to \oplus_{i=1}^r R_{e_i}'
\end{align}
with $e_1,\dots,e_r\in\{0,\dots, n-e\}$ such that
$e_1+\dots+e_r=e-r+1$.
Hence 
\[
E(S)=\Hom_{R'_{n-e}}(\oplus_{i=1}^r R_{e_i}',\mathcal{J}_{n-e}^n(A)),
\]
and there is a canonical isomorphism of $R'_{n-e}$-modules
\begin{align*}\label{can iso}
 \Hom_{R'_{n-e}}(\oplus_{i=1}^rR_{e_i}',\mathcal{J}_{n-e}^n(A))\to \oplus_{i=1}^r\mathcal{J}^n_{n-e_i}(A),
\end{align*}
as long as $A$ is reduced, which we are assuming.
For every $i$, let $l_i$ be a generator of $R_{e_i}'$ as an $R_{n-e}'$-module.
Then this isomorphism is given by sending an $f\in  \Hom_{R'_{n-e}}(\oplus_{i=1}^rR_{e_i}',\mathcal{J}_{n-e}^n(A))$
characterized by the images $f(l_i)$ of the $l_i$ in $\mathcal{J}^n_{n-e_i}(A)$,
to $(f({l}_1),\dots,f(l_r))$ in $\oplus_{i=1}^r\mathcal{J}_{n-e_i}^n(A)$.
As we assume that $R$ has equal characteristics, we get, as explain in
Section \ref{ideal scheme}, for every
choice of a uniformizer $t\in R':=\mathcal{R}(F)$ that
\[
\mathcal{J}_{n-e_i}^n(A)=\{a_{i1}t^{n-e_i+1}+\dots + a_{ie_i}t^n\mid a_{ij}\in A\},
\]
see Formula (\ref{structure Jmn}).
This determines a functorial bijection
\[
 i(A): \oplus_{i=1}^r J_{n-e_i}^n(A)\to \oplus_{i=1}^rA^{e_i}\cong A^e\cong \Hom_F(F[x_1,\dots,x_e], A).
\]
As explained in Example \ref{ex action on Jmn},
one can chose $t$ such that for every $g\in G$ there exists a root of unity $\xi\in k\subset F$,
such that the induce automorphism $g_n\in \Aut(\mathcal{J}_{n-e_i}^n(A))$ is given by
\begin{align}\label{action jnA}
 g_n(a_{i1}t^{n-e_i+1}+ \dots + a_{ie_i}t^n)=a_{i1}\xi^{n-e_i+1}t^{n-e_i+1}+ \dots + a_{ie_i}\xi^nt^n.
\end{align}
Here we need that $k$ contains all roots of unity, and that $G$ is abelian.
We fix a $t$ such that Equation (\ref{action jnA}) holds, and
thus an isomorphism
$E\to \mathbb{A}_F^e$.
%

\medskip
\noindent
\textit{The action on $E$.}
Note that
$y_{n-e}$, given by a map $i:\Spec(F)\to \Gr_{n-e}(Y_\infty)$,
is mapped to $i\circ g_{\Gr}: \Spec(F)\to \Gr_{n-e}(Y_\infty)$ for all $g\in G$.
Here $g_{\Gr}$ denotes again the automorphism of  $\Spec(F)$ corresponding to $g$.
Denote by $g_F$ the corresponding automorphism of $F$,
which induces an automorphism of $R_{n-e}'$ given by
\[
 {g}_{F,n-e}: R_{n-e}'=R_{n-e}\otimes_k F\to R_{n-e}';\ r\otimes f\mapsto r\otimes g_F(f).
\]
This map is well defined as $g_F$ is a morphism over $k$.
Hence ${\gamma: \Spec(R_{n-e}')\to Y_\infty}$ corresponding to the point $y_{n-e}$
gets mapped to ${\gamma \circ \tilde{g}_{F,n-e}}$.
Here $\tilde{g}_{F,n-e}$ is the automorphism of $\Spec(R_{n-e}')$ induced by $g_{F,n-e}$.
So on the level of rings we have,
using the concrete construction of the group action on the points of $\Gr_{n-e}(Y_\infty)$,
that $g_{n-e}^{-1}\circ \tau \circ g_B= g_{F,n-e} \circ \tau$.
One computes that
\begin{align*}
 R_{n-e}'\otimes_{R_{n-e}',g_{F,n-e}}R_{n-e}' &= R_{n-e}\otimes_kF \otimes_{R_{n-e}\otimes_k F,\Id\otimes g_F}R_{n-e}\otimes_k F\\
 & = R_{n-e}\otimes_kF \otimes_{ F,g_F} F= R_{n-e}'\otimes_{F,g_F}F,
\end{align*}
hence
\begin{align*}
\Omega_{B/C}\otimes_{B,g_{n-e}^{-1}\circ\tau\circ g_B}{R}_{n-e}'\otimes_{F, g_{F}^{-1}}F &=\Omega_{B/C}\otimes_{B,\tau}{R}_{n-e}'\otimes_{R_{n-e}',g_{F,n-e}}F\otimes_{F, g_{F}^{-1}}F\\
&\cong \Omega_{B/C}\otimes_{B,\tau}{R}_{n-e}'.
\end{align*}
This implies that
\begin{align*}\label{tensor a}
E(g_{Gr}(S))&=\Hom_{R_{n-e}'}(\Omega_{B/C}\otimes_{B,g_{n-e}^{-1}\circ\tau\circ g_B}{R}_{n-e}',\mathcal{J}_{n-e}^n(A)) \nonumber \\
&=\Hom_{{R}_{n-e}'}(\Omega_{B/C}\otimes_{B,\tau}{R}_{n-e}',\mathcal{J}_{n-e}^n(k)\otimes_kA\otimes_{F,g_{F}^{-1}}F)\nonumber \\
&=\Hom_{R_{n-e}'}(\oplus_{i=1}^ rR_{e_i}',\mathcal{J}_{n-e}^n(A\otimes_{F,g_F^{-1}}F)).
\end{align*}
In the last line we used again the isomorphism from Equation (\ref{j}).
Note that if the action on $\Spec(F)$ is trivial, we have that $A\otimes_{F,g_F {-1}}F=A$ as modules over $F$.
With this notation,
the action of $G$ is given by sending $f\in E(S)$ to $\tilde{f}\in E(g_{\Gr}(S))$
with
\begin{align*}
\tilde{f}(\omega\otimes r \otimes s )&= g_{n}^{-1}( f (g_\Omega(\omega) \otimes g_{n-e}(r)\otimes g_{F}(s)))\otimes 1\\
&= g_{n}^{-1}( f (g_\Omega(\omega) \otimes g_{n-e}(r)\otimes 1))\otimes s
\end{align*}
 for all $g\in G$.
Let $l_1,\dots, l_r$ be as before generators of $\gamma^*\Omega_{Y_\infty/X_\infty}\cong \oplus_{i=1}^rR_{e_i}(F)$.
Note that ${g_\Omega\otimes g_{n-e}\otimes g_F}$ sends $l_i$
to $\sum_{j=1}^rc_{ij}l_j$ for some $c_{ij}\in R_{n-e}'$.
Using that $f \in E(S)$ is a morphism over $R_{n-e}'$,
$f$ is mapped to
$\tilde{f}$ with 
\[
\tilde{f}(l_i)=g_n^{-1}\circ f(\sum_{j=1}^rc_{ij}l_j)=g_n^{-1}( \sum_{j=1}^rc_{ij}f({l}_j))=\sum_{j=1}^r g_n^{-1}(c_{ij})g_n^{-1}(f({l}_j)).
\]
Here $\tilde{f}(l_i)$ is an element in $\mathcal{J}^{n}_{n-e}(A\otimes_{F,g_F^{-1}} F)$.
To simplify the notation, we do not indicate that.
Using the explicit description of $\mathcal{J}^n_{n-e_i}(A)$ from above,
 we can write ${f(l_j)=m_{j1}t^{n-e_j+1}+\dots+m_{je_j}t^n}$ with $m_{jl}\in A$.
Here $t$ is the uniformizer of $R$ we fixed before.
Moreover we have 
$c_{ij}=c_{ij0}t^{0}+\dots+c_{ijn-e}t^{n-e}$, with $c_{ijk}\in F$.
Hence for all $i\in \{1,\dots,r\}$ we get
\begin{align*}
\tilde{f}(l_i)
=g_{n}^{-1}(\sum_{j=1}^r((\sum_{k=0}^{n-e}c_{ijk}t^k)(\sum_{l=1}^{e_j}m_{jl}t^{n-e_j+l}))).
\end{align*}
Using that we are actually computing in $\mathcal{J}_{n-e_i}^n(A)$,
we get
\begin{align*}
\tilde{f}(l_i)= g_{n}^{-1}(\sum_{j=1}^r(\sum_{s=1}^{e_i}(\sum_{l=1}^{e_j}{c}_{isjl}m_{jl})t^{n-e_i+s}) )
= g_{n}^{-1}(\sum_{s=1}^{e_i}(\sum_{j=1}^{r}\sum_{l=1}^{e_j}{c}_{isjl}m_{jl})t^{n-e_i+s}),
\end{align*}
where ${c}_{isjl}:=c_{ij(s-l+e_j-e_i)}$ if that is defined by the equation for $c_{ij}$,
and $0$ otherwise.
Using Equation (\ref{action jnA}) we get
\begin{align*}
\tilde{f}(l_i)
= \sum_{s=1}^{e_i}(\sum_{j=1}^{r}\sum_{l=1}^{e_j}{c}_{isjl}\xi^{-(n-e_i+s)}m_{jl})t^{n-e_i+s}
\end{align*}
for some root of unity $\xi\in k\subset F$.
Hence $\tilde{c}_{isjl}:={c}_{isjl}\xi^{-(n-e_i+s)}$ lies in $F$, and
\[
 L_{is}(m_{11},\dots,m_{re_r}):=\sum_{j=1}^{r}\sum_{l=1}^{e_j}\tilde{c}_{isjl}m_{jl}
\]
defines a linear form over $F$. Set $L_{e_1+\dots+e_{i-1}+j}:=L_{ij}$.

\medskip
\noindent
Analogously to $E(S)$,  we can now identify $E(g_{\Gr}(S))$ with
\[
\oplus_{i=1}^r \mathcal{J}_{n-e_i}^n(A\otimes_{F,g_F^{-1}}F)
\cong(A\otimes_{F,g_F^{-1}}F)^e\cong \Hom(F[x_1,\dots,x_e],A\otimes_{F,g_F^{-1}}F).
\]
%
%
Then the induced map on 
$E(S)=\Hom(F[x_1,\dots,x_e],A)$
sends
$f$ with $f(x_i)=a_i$
to 
\[
\tilde{f}\in E(G_{\Gr}(S))=\Hom(F[x_1,\dots,x_e],A\otimes_{F,g_F^{-1}}F)
\]
with $\tilde{f}(x_i)=L_i(a_1,\dots,a_l)$ for all $A$.
Hence the induced automorphism of $\mathbb{A}_F^e\cong E$
is given on ring level by sending $\sum a_{n_1\dots n_e}x_1^{n_1}\dots x_e^{n_e}\in F[x_1,\dots,x_e]$
to 
\[
\sum g_F(a_{n_1\dots n_e})L_1(x_1,\dots,x_e)^{n_1}\dots L_e(x_1,\dots x_e)^{n_e}.
\]
One observes that this map is linear over the map on $\Spec(F)$.

Altogether this means that $B_{x_n}\cong \mathbb{A}^e_F$ and the action on it is linear over the action on $F$,
hence $B_{x_n}$ is an affine bundle of rank $e$
with affine $G$-action.
\end{proof}

\begin{rem}
 If $G$ is not abelian, Proposition \ref{structure h} is probably still true, 
 but we need this assumption to get the explicit action of $G$ on $\mathcal{J}_{n-e_i}^n(A)$.
 The assumption that $G$ is abelian will also be used to prove Lemma \ref{lemma local global}.
 As we will only consider abelian groups for the applications in Section \ref{eq poincare series} and Section \ref{applications}, it seems to be reasonable to restrict to this case.
\end{rem}

\begin{rem}\label{mixed wrong}
As explained in \cite[Section 2.4]{MR2885338},
Proposition \ref{structure h} can not be shown if $R$ has mixed characteristic, even if the $G$-action is trivial.
If one assumes that $F$ is perfect, one gets Proposition \ref{structure h} without G-action, see \cite[Lemma~2.4.4]{MR2885338}.
In addition to the problems in the non-equivariant case,
we need that $R$ has equal characteristic to describe $\mathcal{J}_n^m(A)$ and the action on it explicitly.
 \end{rem}
 
\section{Equivariant motivic integration}
\label{motivic integration}
\noindent
The aim of this section is to establish motivic integration on formal schemes with an action of a finite group $G$,
which will have values in an equivariant Grothendieck ring of varieties, see Section \ref{eq grothendieck}.
The main result in this section is the change of variables formula for this equivariant motivic integrals, Theorem \ref{changevar}.

\subsection{The equivariant Grothendieck ring of varieties}
\label{eq grothendieck}
\noindent
Let $S$ be any separated scheme,
endowed with a good action of a finite group $G$. 

\begin{defn}\label{equivariant Gring}
The \emph{equivariant Grothendieck ring of $S$-varieties}
$K_0^G(\Var_S)$ is defined as follows: as an abelian group, it is
generated by isomorphism classes $[X]$ of elements $X\in(\Sch_{S,G})$. These generators are subject to the following
relations:
\begin{enumerate}
\item $[X]=[Y]+[X\setminus Y]$, whenever $Y$ is a closed $G$-equivariant 
subscheme of $X$ (scissors relation).
\item $[V]=[W]$, whenever 
$V\rightarrow B$ and $W\to B$ are two
 $G$-equivariant affine bundles of rank $r$ over $B$ with affine $G$-action, see Definition \ref{affine bundle}.
\end{enumerate}
For all $X,Y\in (\Sch_{S,G})$, set 
$[X][Y]:=[X\times_S Y]$, where
 the fiber product is taken in $(\Sch_{S,G})$.
 This product extends bilinearly to
 $K_0^G(\Var_S)$ and makes it into a ring.
 
 Let $\LL_S$ be
 the class of the affine line $\A^1_S$ with $G$-action induced by the action on $S$,
 and the trivial action on the affine line.
 We define $\mathcal{M}^G_S$ as the localization
 $K_0^G(\Var_S)[\LL_S^{-1}]$.
  \end{defn}
  
  \noindent
  For a  discussion of the different definitions of the equivariant Grothendieck ring of varieties in the literature, we refer to \cite[Chapter 4]{abi2}.
  
 \begin{nota}
  If $G$ is the trivial group, we write
 $K_0(\Var_S)$ and $\mathcal{M}_S$ instead of $K_0^G(\Var_S)$ and
 $\mathcal{M}^G_S$, receptively.
 Note that in this case Relation (2) becomes trivial.
 If $S=\Spec(A)$,
 we write $K_0^G(\Var_A)$ for $K_0^G(\Var_S)$, $\mathbb{L}_A$ for $\mathbb{L}_S$, and $\mathcal{M}_A^G$ for $\mathcal{M}_S^G$.
   If the base scheme $S$ is clear from the
 context, we write $\LL$ instead of $\LL_S$.
 \end{nota}

\begin{rem}\label{remark}
A morphism of finite groups $G'\rightarrow G$ induces forgetful
ring morphisms
$K^{G}_0(\Var_S)\rightarrow K^{G'}_0(\Var_{S})$
and
$\mathcal{M}^{G}_S\rightarrow \mathcal{M}^{G'}_{S}$.
 If
 $G'\rightarrow G$ is surjective, then these morphisms are
 injections.
\end{rem}

\begin{defn}
 Let $S$ be a separated scheme with an action of a profinite group
 \[
 \widehat{G}=\lim_{\stackrel{\longleftarrow}{i\in I}} G_i
 \]
 factorizing through a good action of some finite quotient $G_j$.
 Then we define
 $$K_0^{\widehat{G}}(\Var_S):=\lim_{\stackrel{\longrightarrow}{i\in
 I}}K_0^{G_i}(\Var_S)\ \mathrm{and}\ \mathcal{M}^{\widehat{G}}_{S}:= \lim_{\stackrel{\longrightarrow}{i\in
 I}}\mathcal{M}^{G_i}_{S}.$$
 \end{defn}

\begin{rem}\label{class constructible}
Take $X\in (\Sch_{S,G})$,
and let $C\subset X$ be a constructable subset, closed under the action of $G$.
Then $C$ defines an element in $K^G_0(\Var_S)$.

To see this,
take a generic point $\eta\in C$, and let $\bar{\eta}$ be its closure in $X$.
As $\eta\in C$, there exists an open  $U\subset \bar{\eta}$ containing $\eta$ such that $U\subset C$.
The orbit $G(\bar{\eta})$ of $\bar{\eta}$ is a closed $G$-invariant subscheme of $X$.
Shrinking $U$ a bit, we may assume that $U$ is also open in $G(\bar{\eta})$.
As for all $g\in G$ the induced map on $G(\bar{\eta})$ is an isomorphism, $C_1:=\cup_{g\in G}g(U)$ is open in $G(\bar{\eta})$,
hence it defines in particular an element in $K_0^G(\Var_S)$.
As $C$ is $G$-invariant, $C_1$ is contained in $C$.
Using Notherian induction on $C\setminus C_1$ the claim follows.
\end{rem}
 
 \begin{rem}\label{compute [V]}
 Note that the trivial bundle $\mathbb{A}^r_S\times_S{B}=\mathbb{A}_B^r\to B$
with the group action induced by that on $B$ and the trivial one on the affine space
is an affine bundle of rank $r$ over $B$
with affine $G$-action.
From this it follows with the second relation in the definition of the equivariant Grothendieck ring 
that for every $G$-equivariant affine bundle $V\to B$ of rank $r$   
with affine $G$-action
\[
 [V]=[B]\mathbb{L}_{S}^r\in K_0^G(\Var_S).
\]
 \end{rem}
 
\noindent
We show now that this formula also holds with less assumptions if $G$ is abelian.
 
 \begin{lem}\label{lemma local global}
Assume that $G$ is abelian.
Take $J,I\in (\Sch_{S,G})$, and let $h: J\to I$ be a $G$-equivariant morphism.
For all $x\in I$ denote by $G_x$ the stabilizer of $x$, and by $J_x$ the underlying reduced subscheme of $h^{-1}(x)$,
on which we get an induced action of $G_x$.
Assume that for all $x\in I$, $J_x$ is a $G_x$-equivarinat affine bundle of rank $e$ with affine $G_x$-action.
Then
\[
  [J]=[I]\mathbb{L}_{S}^e\in K_0^G(\Var_S).
\]
\end{lem}

\begin{proof}
Take any generic point $\eta=\Spec(F)$ of $I$.
Replace, if necessary, $J$ by its reduced underlying subscheme with induced $G$-action.
We can do this, because it does not change the corresponding class in the Grothendieck ring.
Moreover $J^{\text{red}}\times_{I}\eta$ is equal to the reduced subscheme $J_\eta$ of $h^{-1}(\eta)$.
Let $G_\eta\subset G$ be the stabilizer of $\eta$.
By assumption $J_\eta$ is a $G_\eta$-equivariant affine bundle of rank
$e$ with translation space $E\cong \mathbb{A}_F^e$ and affine $G_\eta$-action,
i.e.~there is a $G_\eta$-action on $E$, which is linear over the action on $F$,
and a $G_\eta$-equivariant morphism
 $\varphi: E\times J_\eta\to J_\eta$ inducing an isomorphism $\varphi\times p_{J_\eta}: E\times J_\eta\to J_\eta\times J_\eta $,
 where $p_{J_\eta}$ denotes the projection to $J_\eta$. 

Take a $G_\eta$-invariant affine open $U\subset I$ containing $\eta$, which
exists because the action of $G_\eta$ on $I$ is good.
For all $g\in G_\eta$ with corresponding automorphism $g_F$ of  $\Spec(F)=\eta$, the induced map $g_E':E\to g_F^* (E)$ is linear over $F$, hence given by matrices with coefficients in $F$.
So after maybe shrinking $U$ again, we may assume that these matrices give rise to morphisms of vector bundles
${g_U': E_U:=E\times_F U \to g_U^* (E_U)}$, where $g_U$ denotes the automorphism of $U$ corresponding to $g$.
By replacing $U$ by $\cap_{g\in G}g_U(U)$, we may assume that $U$ is $G$-invariant.
Combining these maps with the projection maps $g_U^*(E_U)\to E_U$, we get a well defined good $G$-action on $E_U$,
which is linear over the action on $U$.

Note that $(E_U\times_U J_U)\times_U \eta=E \times J_\eta$ and $ J_U\times_U \eta=J_\eta$.
Hence it follows from \cite[Theorem~8.8.2]{MR0217086} that after maybe restricting $U$ again,
there is a unique $U$-morphism $\varphi_U: E_U\times_U J_U\to J_U$ such that its restriction to $\eta \in U$ is equal to $\varphi$.
Again we may assume that $U$ is $G$-invariant.
Using a similar argument for $\varphi\times p_{J_\eta}$ and its inverse,
we may assume that $\varphi_U\times p_{J_U}:  E_U\times_U J_U\to J_U\times J_U$
is actually an isomorphism.
Here $p_{J_U}$ denotes the projection to $J_U$.

For $g\in G_\eta$ let $g'_{E_U\times J_U}: E_U\times_U J_U\to g_U^* (E_U\times_U J_U)$, $g'_{J_U}: J_U\to g_U^* (J_U)$,
$g'_{E\times J_\eta}: E\times J_\eta\to g_F^* (E\times J_\eta)$, and $g'_{J_\eta}: J_\eta\to g_F^* (J_\eta)$ be
the maps induced by he actions on $E_U\times_U J_U$, $J_U$, $E\times J_\eta$ and $J_\eta$.
By \cite[Theorem~8.8.2]{MR0217086}, we can restrict $U$ such that for all $g\in G_\eta$,
there is a unique map $ E_U\times J_U\to g_U^*(J_U)$ restricting to $g_{J_\eta}'\circ \varphi=\varphi \circ g_{E\times J_\eta}'$.
As both $g'_{J_U}\circ \varphi_U$ and $\varphi_U\circ g'_{J_U\times E_U}$ have this property,
they are equal, and hence $\varphi_U$ is $G_\eta$-invariant.
Altogether $J_U$ is an affine bundle with affine $G_\eta$-action and translation space $E_U$.

We can now restrict $U$ further such that for all $g\in G\setminus G_\eta$, $g(U)\cap U=\emptyset$.
Set $V=\cup _{g\in G} g(U)$, and let $r$ be the number of connected components $V_i \subset V$.
Note that all such $V_i$ are of the form $g_i(U)$ for some $g_i\in G$.
For all $i$ we fix such a $g_i$. Without loss of generality we may assume that $V_1=U$ and $g_1=\Id$.
Consider the vector bundle $f: \tilde{E}:=\sqcup _{i=1}^r E\cong \mathbb{A}^e_{V}\to V$ over $V$.
Here for all $i$ the $U$-scheme $E$ becomes a $V_i$-scheme using $g_i$.

For every $g\in G$ denote by $g$ also the corresponding morphisms of $J_V$ and $E_U$
(the last of course only exists if $g\in G_\eta\subset G$).
Denote by $\tilde{E}_{i}=E\times V_i$ the inverse image of $V_i$ in $\tilde{E}\times_V J_V$.
We define a morphism $\tilde{\varphi}: \tilde{E}\times J_V\to J_V$
by setting $\tilde{\varphi}\rvert_{\tilde{E}_i}=g_i\circ \varphi \circ (\Id\times g_i^{-1})$.
By construction, the induced map $\tilde{E}\times J_V\to J_V\times J_V$ is an isomorphism.

For every $g\in G$ with $g(V_i)=V_j$,
$g g_i g_{j}^ {-1} \in G_\eta$.
Consider the automorphism of $\tilde{E}$ 
given by sending $c\in f^{-1}(V_i)\cong E_U$
to $g g_i g_{j }^ {-1}(c) \in f^{-1}(V_j)\cong E_U$.
Doing so for all $g\in G$, we get a good $G$-action on $\tilde{E}$.
For any $g\in G$
look at the induced map 
$g_{\tilde{E}}' : \tilde{E}\to g_V^*(\tilde{E})$.
If $g(V_i)=V_j $,
then we have again that $\tilde{g}:=gg_i g_{j }^ {-1}\in G_\eta$, and
it is easy to see that
\[
 g_{\tilde{E}}'\rvert_{f^ {-1}(V_i)}: f^ {-1}(V_i)\cong E_U\to \tilde{g}_V^ * (f^ {-1}(V_j ))\cong \tilde{g}_U^*(E_U),
\]
coincides with the map $\tilde{g}'_{E_U}$, which is a morphism of vector bundles.
Thus the action on $\tilde{E}$ is linear over the action on $V$.
Moreover we have
\begin{align*}
 \tilde{\varphi}\circ g\rvert_{\tilde{E}_i}&=\tilde{\varphi}\rvert_{\tilde{E}_{j }}\circ (g g_ig_{j}^{-1}\times g)=g_{j}\circ \varphi \circ ((g  g_i  g_{j}^{-1})\times gg_{j}^{-1})\\
 &= g_{j}gg_ig_{j}^{-1}\circ \varphi \circ (\Id\times g_i^ {-1})=gg_i\circ \varphi\circ (\Id\times g_i^ {-1})=g\circ \tilde{\varphi}\rvert_{\tilde{E}_i}.
\end{align*}
Here we used that $gg_ig_{j}^{-1}\in G_\eta$, that $\varphi_U$ is $G_\eta$-invariant,
and that $G$ is actually commutative.
This calculation implies that $\tilde{\varphi}$ is $G$-equivariant.
Hence all together we have shown that $J_V$ is a $G$-equivariant affine bundle of rank $e$ over $V$
with translation space $\tilde{E}$ and affine $G$-action.

Now we proceed with $I\setminus V$
until by Notherian induction we found a stratification of $I$ into finitely many locally closed subschemes $C_i$
such that $(h^{-1}(C_i))^{\text{red}}\to C_i$ is an affine bundle of rank $e$ with affine $G$-action. Hence by Remark \ref{compute [V]}
\[
 [J]=[h^{-1}(I)]=\hspace{-3pt}\sum[h^{-1}(C_i)]=\hspace{-3pt}\sum[(h^{-1}(C_i))^{\text{red}}]=\hspace{-3pt}\sum [C_i]\mathbb{L}_S^e=[I]\mathbb{L}_S^e\in K_0^G(\Var_S).
\]

\end{proof}

\subsection{Equivariant motivic measure and integrals}
\label{eq motivic measure}
Let $G$ be a finite group, acting well on a complete discrete valuation ring $R$.
Take $X_\infty\in(stft/R,G)$, and assume that it has pure
relative dimension $m$ over $R$.
Consider the $G$-actions on $\Gr(X_\infty)$ and $\Gr_n(X_\infty)$ as constructed
in Proposition \ref{Gaction}.

\begin{defn}\label{defn cylinder}
Let $n\geq 0$ be an integer. A subset $A$ of $\Gr(X_\infty)$ is
called a \emph{cylinder of degree $n$}, if there exists a constructable
subset $C$ of $\Gr_n(X_\infty)$, such that $A=\theta^{-1}_n(C)$.
We say that a cylinder $A$ of degree $n$ is \emph{$G$-stable of degree $n$} if, moreover, $C=\theta_n(A)$ is closed under the action of $G$, and for any
integer $N\geq n$, the truncation map
$(\theta^{N+1}_n)^{-1}(C)\rightarrow (\theta^N_n)^{-1}(C)$ is piecewisely a $G$-equivariant affine bundle of rank $m$ with affine $G$-action.
\end{defn}

\begin{rem}\label{G-stabel G-equivariant}
Assume that $X_\infty$ is smooth over $R$.
 Then a cylinder $A\subset \Gr(X_\infty)$ of degree $n$ is $G$-stable if and only if it is $G$-invariant.
 This holds, because by Proposition \ref{Gaction}
the truncation map $\theta_n$ is $G$-equivariant,
so $\theta_n(A)$ is closed under the action of $G$ if and only if $A$ is $G$-invariant,
and by
Proposition~\ref{prop structure map} the truncation map
$\theta_N^{N+1}:\Gr_{N+1}(X_\infty)\to \Gr_N(X_\infty)$
is a $G$-equivariant affine bundle of degree $m$ with affine $G$-action for all $N\geq n$.
\end{rem}

\noindent
For every $G$-stable cylinder $A$ of degree $n$,
$\theta_n(A)$ is a constructable $G$-invariant subset of
the finite type $X_0$-scheme $\Gr_n(X_\infty)$,
and hence defines an element
of $K_0^G(\Var_{X_0})$
by Remark \ref{class constructible}.
This leads us to the following definition.

\begin{defn}\label{defn measure}
Let $A\subset \Gr(X_\infty)$ be a $G$-stable cylinder of degree $n$. Then
\[
\mu^G_{X_0}(A):=[\theta_n (A)]\LL^{-(n+1)m}\in \mathcal{M}^G_{X_0}
\]
is the \emph{naive $G$-equivariant motivic measure} of $A$
on $\Gr(X_\infty)$.

\end{defn}

\begin{rem}
Note that if $A$ is a $G$-stable cylinder of degree $n$, then it is a
$G$-stable cylinder of degree $n'$, for any $n'\geq n$,
because $\theta_n^{n'}$ is $G$-equivariant.
But still $\mu_{X_0}^G$ only depends on $A$ and not on $n$.
This is true, because if we view $A$
as a cylinder of degree $n'$
with $n'\geq n$,
then using Remark \ref{compute [V]}
we get
\[
 [\theta_{n'}(A)]\mathbb{L}^{-(n'+1)m}=[\theta_n(A)]\mathbb{L}^{(n'-n)m}\mathbb{L}^{-(n'+1)m}=[\theta_n(A)]\mathbb{L}^{-(n+1)m}\in K_0^G(\Var_{X_0}).
\]
\end{rem}

\begin{defn}\label{motintegral}
We call a function 
$\alpha:\Gr(X_\infty)\rightarrow\Z$ ,
i.e.~a map from all points of $\Gr(X_\infty)$ to $\mathbb{Z}$,
\emph{naively $G$-integrable}, if $\alpha$ takes
only a finite number of values, and if $\alpha^{-1}(i)$ is a
$G$-stable cylinder for each $i\in \N$. In this case, we define
the \emph{motivic integral} of $\alpha$ by
\[
\int_{X_\infty}\hspace{-10pt}\LL^{-\alpha}d\mu^G_{X_0}:=\hspace{-3pt}\sum_{i\in \mathbb{Z}}\mu^G_{X_0}(\alpha^{-1}(i))\LL^{-i} \in \mathcal{M}^G_{X_0}.
\]
\end{defn}

\begin{rem}\label{additive}
It is clear from the definition that $\mu^G_{X_0}$ is additive, i.e.~if
we can write a $G$-stable cylinder $A$ as a union
$A_1\cup A_2$ of $G$-stable cylinders $A_i$, then
\[
\mu^G_{X_0}(A)=\mu^G_{X_0}(A_1)+\mu^ G_{X_0}(A_2)-\mu_{X_0} G(A_1\cap A_2).
\]
It follows that if
$\alpha$ and $\beta$ are naively $G$-integrable,
then $\alpha+\beta$ is naively $G$-integrable, too, and
\[
\int_{X_\infty}\hspace{-10pt}\LL^{-(\alpha+\beta)}d\mu^G_{X_0}:=\hspace{-5pt}\sum_{i,j\in \mathbb{Z}}\mu^G_{X_0}(\alpha^{-1}(i)\cap \beta^{-1}(j))\LL^{-(i+j)} \in \mathcal{M}^G_{X_0}  .
\]
Moreover,
if $\{X_\infty^l\}_{l\in L}$ is a finite $G$-invariant cover of $X_\infty$ by opens, we have that
\[
 \int_{X_\infty}\hspace{-10pt}\LL^{-\alpha}d\mu^G_{X_0}=\hspace{-7pt}\sum_{\emptyset\neq \mathcal{L}\subset L}\hspace{-7pt}(-1)^{\lvert \mathcal{L}\rvert -1}\hspace{-5pt}\int_{\cap_{l\in \mathcal{L}}X_\infty^l}\hspace{-25pt}\LL^{-\alpha}d\mu^G_{X_0}.
\]

\end{rem}

\begin{rem}
As in \cite{sebag1} and \cite{NiSe2}, one could
define a bigger class of $G$-closed measurable subsets of
$\Gr(X_\infty)$, endowed with a $G$-equivariant motivic measure
taking values in an appropriate completion of
$\mathcal{M}^G_{X_0}$. We will not need such a construction for our
purposes. 
\end{rem}

\subsection{The equivariant change of variables formula}
\label{eq change of variables}
Let $G$ be again a finite group, acting nicely on a discrete valuation ring $R$.
Let furthermore $X_\infty$ and $Y_\infty$ in $ (stft/R,G)$ be smooth and of relative dimension $m$ over $R$,
and let $h: Y_\infty \to X_\infty$ be a $G$-equivariant morphism.
To simplify the notation,
we write $h$ also for the induced maps $\Gr(h): \Gr(Y_\infty)\to G(X_\infty)$
and ${\Gr_n(h): \Gr_n(X_\infty)\to \Gr_n(Y_\infty)}$.

\begin{rem}\label{alpha h}
Let $\alpha: \Gr(X_\infty)\to \mathbb{Z}$ be a naively $G$-integrable function.
Then
\[
 \alpha\circ h: \Gr(Y_\infty)\to \mathbb{Z}
\]
is also naively $G$-integrable.
This can be seen as follows:
as the image of $\alpha$ is finite, the same holds for the image of $\alpha \circ h$.
Moreover for all $i\in \mathbb{Z}$, $A_i:=\alpha^{-1}(i)\subset \Gr(X_\infty)$ is a $G$-stable cylinder of degree $n$
for some $n \in \mathbb{N}$.
As $h$ is $G$-equivariant, the same is true for the induced map on the Greenberg schemes,
so 
\[
(\alpha \circ h )^{-1}(i)= h^{-1}(A_i)= \theta_{n}^{-1}({h}^{-1}(\theta_{n}(A_i)))\subset \Gr(Y_\infty)
\]
is $G$-invariant, and a cylinder, because
$h^{-1}(\theta_{n}(A))$ is constructable due to the fact that
it is the inverse image of the constructable set $\theta_n(A)$.
So, as $Y_\infty $ is smooth, by Remark~\ref{G-stabel G-equivariant}, $(\alpha \circ h)^ {-1}(i)$ is $G$-stable cylinder.
\end{rem}

\noindent
The aim of this section is to compare the motivic integrals of $\alpha$
and $\alpha\circ h$. 
In the change of variable formula,
the difference will be described using the order of the Jacobian of $h$,
see Definition \ref{def ordjac}.
Before we proof the change of variable formula, we first need to show the following lemma
about the order of the Jacobian.

\begin{lem}\label{jac}
The fibers of the function
$\Ord(\Jac_h): \Gr(Y_\infty) \to \mathbb{N}$
are $G$-invariant.
\end{lem}

\begin{proof}
%
Take $g\in G$,
and let $g_{X_\infty}\in \Aut(X_\infty)$ and $g_{Y_\infty}\in \Aut(Y_\infty)$ be the corresponding automorphisms.
As $Y_\infty,X_\infty\in (stft/R,G)$,
the natural maps
\[
g_{Y_\infty}^*\Omega^m_{Y_\infty/R}\to \Omega^m_{Y_\infty/R} \text{ and } g_{X_\infty}^*\Omega^m_{X_\infty/R}\to \Omega^m_{X_\infty/R}
\]
are isomorphisms.
Hence we get the following commutative diagram:
\[
 \xymatrix{
   g_{Y_\infty}^*h^* \Omega^m_{X_\infty/R}\ar[r]\ar[d]_\cong & g_{Y_{\infty}}^*\Omega^m_{Y_\infty/R}\ar[d]^\cong \\
   h^* \Omega^m_{X_\infty/R}\ar[r] & \Omega^m_{Y_\infty/R}
 }
\]
Here we used that as $h$ is $G$-equivariant,
$h\circ g_{Y_\infty}=g_{X_\infty}\circ h$.
Take a closed point in $\Gr(Y_\infty)$ with residue field $F$,
corresponding to an element $\psi \in Y_\infty(R')$
with $R'=\mathcal{R}(F)$.
Pulling back all the maps in the commutative diagram with $\psi$ we get 
that the cokernels of the two maps
\begin{align*}
 s_g:\psi^*g^*_{Y_\infty} h^* \Omega^m_{X_\infty/R}\to \psi^* g_{Y_\infty}^*\Omega^m_{Y_\infty/R}\text{ and }
 s:\psi^*h^*\Omega^m_{X_\infty/R}\to \psi^* \Omega^m_{Y_\infty/R}
\end{align*}
are isomorphic.
Recall that the $G$-action on $R$ induces canonically a $G$-action on $R'$.
Let $g_{R'}\in \Aut(\Spf(R'))$ be the automorphism corresponding to $g\in G$.
Now pulling back $s_g$ via $g_{R'}^{-1}$, we get that the cokernel of $s$ is also isomorphic to the cokernel of
\[
\bar{s}_g: {{g_{R'}^{-1}}}^* \psi^* g_{Y_\infty}^*h^* \Omega^m_{X_\infty/R}\to  {g_{R'}^{-1}}^* \psi^* g_{\infty}^*\Omega^m_{Y_\infty/R}.
\]
Note that ${g_{R'}^{-1}}^* \psi^* g_{Y_\infty}^*=(g_{Y_\infty} \circ \psi \circ g_{R'}^{-1})^* $.
Now assume that $\psi$ is corresponding to a point in $ J_e:=\Ord(\Jac_h)^{-1}(e)$,
hence by Remark \ref{compute ordjac} the cokernel of $s$ has length $e$,
so the same holds for the cokernel of $\bar{s}_g$,
and the point corresponding to 
$g\circ \psi \circ g_{R'}^{-1}$ lies in $J_e$, too.
By Remark \ref{action on closed points}, 
the action of $G$ on $\Gr(Y_\infty)$ maps $\psi$ to $g\circ \psi \circ g_{R'}^{-1}$ for all $g\in G$,
so $J_e$ is closed under the action of $G$
for all $e\in \mathbb{N}$.
%
\end{proof}

\begin{rem}\label{nonsmooth jac}
 Assume that $X_\infty$ is only generically smooth.
 Take a point $y$ in $\Gr(Y_\infty) $ corresponding to $\psi\in Y_\infty(R')$.
 Then by Remark \ref{compute ordjac}, $\Ord(\Jac_h)(y)$
 is given by the length of the cokernel of
$s: \psi^*h^*\Omega^m_{X_\infty/R}/(\text{torsion})\rightarrow \psi^*\Omega^m_{Y_\infty/R}$.
Hence dividing out torsion in the proof above gives us a proof of Lemma \ref{jac} in the case that $X_\infty$ is only generically smooth.
\end{rem}

\noindent
Now we are ready to state and proof the change of variables formula for equivariant motivic integrals.
The main ingredient of the proof is Proposition \ref{structure h}.
To be able to use it, we need to put some extra assumptions on $G$ and $R$.

\begin{thm}[Equivariant change of variables formula]\label{changevar}
Assume that $G$ is a finite abelian group,
and acts tamely on a complete discrete valuation ring of equal characteristic $R$,
whose residue field contains all roots of unity.
Let $X_\infty,Y_\infty$ in $(stft/R,G)$ be smooth and of pure dimension over $R$,
and let $h:Y_\infty\rightarrow X_\infty$ be a $G$-equivariant morphism,
such that $h_\eta:Y_\eta\rightarrow X_\eta$
is an open immersion, and the induced map
${Y_\eta (K')\to X_\eta(K')}$ is a bijection for all unramified extensions $K'$ of $K$,
the quotient field of $R$.

If $\alpha$ is a naively $G$-integrable function on
$\Gr(X_\infty)$, then $\alpha\circ h+\Ord(\Jac_h)$ is naively
$G$-integrable on $\Gr(Y_\infty)$, and
\[
 \int_{X_\infty}\hspace{-10pt}\LL^{-\alpha}d\mu^G_{X_0} = \hspace{-2pt}\int_{Y_\infty}\hspace{-10pt}\LL^{-(\alpha\circ h + \Ord(\Jac_h))}d\mu^G_{X_0}\in\mathcal{M}^G_{X_0}. 
\]
\end{thm}

\begin{proof}
By Remark \ref{alpha h}, $\alpha \circ h$
is naively $G$-integrable.
As $h_\eta$ is an open immersion,
by \cite[Chapter 4, 3.2.4]{CNS} $\Ord(\Jac_h)$ is constant on $\Gr(Y_\infty^i)$ for every connected component $Y_\infty^i$ of $Y_\infty$.
Hence for all $e\in \mathbb{N}$, $J_e:=\Ord(\Jac_h)^{-1}(e)$ is the union of the $\Gr(Y_\infty^i)$
such that $\Ord(\Jac_h)\rvert_{\Gr(Y_\infty^i)}$ has value $e$,
which are by construction cylinders of degree $0$.
By Lemma \ref{jac}, $J_e$ is $G$-invariant, hence, as $Y_\infty$ is smooth, by Remark~\ref{G-stabel G-equivariant} a $G$-stable cylinder,
so $\Ord(\Jac_h)$ is naively $G$-integrable.
By Remark~\ref{additive}
the same holds also for the sum of the two considered functions.

Set $A_i:=\alpha^{-1}(i)$ for all $i\in \mathbb{Z}$.
Then
$h^{-1}(A_i)\cap J_e$ is a $G$-invariant cylinder.
As the map $h: \Gr(Y_\infty)\to \Gr(X_\infty)$ is $G$-equivariant,
$h(h^{-1}(A_i) \cap J_e)=A_i\cap h(J_e)$
is $G$-closed. 
By \cite[7.2.2]{sebag1} it is a cylinder,
hence using Remark \ref{G-stabel G-equivariant} it is a $G$-stable cylinder.
Now consider
\[
 h:{\theta_n(h^{-1}(A_i) \cap J_e)}\to \theta_n(A_i \cap h(J_e))
\]
for some $n\geq 2e$.
For every point $x_n \in \theta_n(A_i \cap h(J_e))$ with stabilizer $G_x$, the induced map
$h: (h^{-1}(x_n))^{\text{red}}\to x_n$ is a
$G_x$-equivariant affine bundle of rank $e$ with affine $G_x$-action, see Proposition \ref{structure h}.
Hence by Lemma \ref{lemma local global},
\[
 [\theta_n(h^{-1}(A_i) \cap J_e)]=[\theta_n(A_i\cap h(J_e))]\mathbb{L}^e \in K_0^G(\Var_{X_0}).
\]
This implies that
\begin{align*}
 \int_{Y_\infty}\hspace{-10pt}\LL^{-(\alpha\circ h + \Ord(\Jac_h))}d\mu^G_{X_0}
 & = \hspace{-3pt}\sum_{i,e\in \mathbb{Z}}\mu^G_{X_0}(h^{-1}(A_i)\cap J_e)\LL^{-(i+e)}\\
 & = \hspace{-3pt}\sum_{i,e\in \mathbb{Z}}\mu^G_{X_0}(A_i\cap h(J_e))\LL^{e}\LL^{-(i+e)}\\
 & =\hspace{-2pt}\sum_{i\in \mathbb{Z}}\mu^G_{X_0}(A_i)\LL^{-i}
 =\hspace{-2pt}\int_{X_\infty}\hspace{-10pt}\LL^{-\alpha}d\mu^G_{X_0}.
\end{align*}
Here we used
that if $Y_\eta(K')\to X_\eta(K')$ is a bijection for every unramified extension $K'/K$,
then the map $h: \Gr(Y_\infty)\to \Gr(X_\infty)$ is a bijection, too, see \cite[Lemma~2.4.1]{MR2885338}.
Hence $h(J_e)\cap h(J_j)$ is empty for $e\neq j$,
and 
\[
\bigcup_{e\in \mathbb{N}} h(J_e)=h(\bigcup_{e\in \mathbb{Z}}J_e)=h(\Gr(Y_\infty))=\Gr(X_\infty).
\]
\end{proof}

\begin{rem}
 If $R$ has unequal characteristic,
 by Remark \ref{mixed wrong}
 we do not get Proposition \ref{structure h} in the usual Grothendieck ring,
 even in the non-equivariant case.
 Still it might be possible to have a similar result in the some modified equivariant Grothendieck ring,
 where we divide out purely inseparable maps.
\end{rem}

\section{Group actions on weak N\'eron models}
\label{weak neron}
\noindent
In order to be able to define and compute the equivariant integral of a gauge form
of a possibly non-smooth formal scheme with group action in Section~\ref{eq poincare series},
we will make use of weak N\'eron models with group actions,
which will be studied in this section.

\subsection{Equivariant N\'eron smoothenings}
\label{neron smoothening}
Let $G$ be a finite group,
fix a nice $G$-action on a complete discrete valuation ring $R$,
and take $X_\infty \in (stft/R,G)$ flat over $R$.
Denote by $X_\eta$ the generic fiber in the category of rigid varieties.

\begin{defn}[\cite{formner}] \label{def wnm}
A \textit{weak N\'eron $R$-model for $X_\eta$} is a smooth
formal scheme $U_\infty\in (stft/R)$, whose generic fiber is an
open rigid
 subspace of $X_\eta$, and which has the property that the natural
maps
$U_\infty(R')\rightarrow X_\eta(K')$ are bijective for any finite
unramified extension $K'$ of $K$, where $R'$ denotes the normalization
of $R$ in $K'$.
\end{defn}

\begin{defn}
We say that a morphism $f:U_\infty \rightarrow X_\infty$ in
$(stft/R,G)$ is a \emph{$G$-equivariant N\'eron $R$-smoothening} for
$X_\infty$, if it satisfies the following properties:
\begin{enumerate}  \item there exists a morphism
 $X'_\infty\rightarrow X_\infty$  in $(stft/R,G)$,
 inducing an isomorphism $X'_\eta\rightarrow X_\eta$ on the generic
 fibers, such that
 $f$ factors through
a $G$-equivariant open immersion $U_\infty\hookrightarrow X'_\infty$,
\item $U_\infty$ is a weak N\'eron $R$-model for $X_\eta$.
\end{enumerate}
\end{defn}

\noindent
Note that
$U_\infty=\Sm(X'_\infty)$, since any closed point of the special fiber of $\Sm(X'_\infty)$
lifts to a section in $X'_\infty(R')$ for some finite unramified
extension $R'$ of $R$.
The action of $G$ on $X_\infty'$ automatically restricts to $\Sm(X_\infty ')$,
because smooth points are mapped to smooth points by automorphisms.

\begin{rem}\label{not unique wnm}
Note that a $G$-equivariant N\'eron smoothening is in general not unique.
If we have a $G$-equivariant N\'eron smoothening $f:U_\infty\to X_\infty$
of $X_\infty$ given by $X'_\infty\to X_\infty$,
then blowing up $X'_\infty$ in the orbit of a closed point in the image of the special fiber of $U_\infty$
gives rise to a different $G$-equivariant N\'eron smoothening of $X_\infty$.
\end{rem}

\begin{thm}\label{neronsmooth}
Every generically smooth, flat  formal scheme $X_\infty\in (sftf/R,G)$ admits a $G$-equivariant N\'eron smoothening.
\end{thm}

\begin{proof}
Let  $\mathcal{I}$
be any ideal sheaf on $X_\infty$, which contains the uniformizing parameter $t$ of $R$ and is closed under
the action of $G$. Let $h:X'_\infty\rightarrow X_\infty$ be the
formal blow-up of $X_\infty$ at $\mathcal{I}$.
Fix $g\in G$,
and denote by $g$ also the corresponding automorphism of $X_\infty$. 
Then by flat base change for formal blow-ups, see \cite[Proposition 2.16]{MR2461256}, we
get a Cartesian square
\[
\xymatrix{
X_\infty''\ar[r]^{g'}\ar[d]_{h'}& X'_\infty\ar[d]^h \\
 X_\infty \ar[r]^g & X_\infty
}
\]
where $h'$ is the formal blow-up of $X_\infty$ at
$g^*\mathcal{I}$. Since $g^{*}\mathcal{I}=\mathcal{I}$ by assumption,
$X_\infty''=X_\infty'$ and $h=h'$, hence we have a
natural morphism $g':X'_\infty\rightarrow X'_\infty$ lying over $g$. 
Doing so for every $g\in G$, this defines
an action of $G$ on $X'_\infty$
such that $h$ is $G$-equivariant. The fact that this action is good
follows, because $h$ is projective.

By \cite[\S 3, Theorem 3.1]{formner}
every quasi-compact formal $R$-scheme,
hence in particular every stft formal $R$-scheme,
admits a N\'eron smoothening by means of admissible blow-ups, i.e.~by formal blow-ups with center in the special fiber of
$G$-closed ideal sheaves.
From the argument above it follows that it suffices to show that these ideal sheaves are $G$-closed.
The canonical
smoothening for the algebraic case constructed in \cite[\S 3, Theorem 2]{neron}
is given by a sequence of blow-ups in $G$-closed ideal sheaves, which was shown in \cite[Lemma 2.10]{Abi1}.
As the construction of the ideal sheaves in the formal setting,  see \cite[\S 3, Lemma 3.4]{formner},
works completely
analogously, the same proof can be used in the formal setting.
\end{proof}

%

\begin{cor}\label{dominate}
Take $X_\infty\in (stft/R,G)$, and 
let $f_i: U_\infty^i\to X_\infty$, $i\in \{1;2\}$,
be two $G$-equivariant N\'eron $R$-smoothenings of $X_\infty$.
Then there is a third $G$-equivariant N\'eron $R$-smoothening
$h: V_\infty\to X_\infty$, and two $G$-equivariant maps 
$h_i: V_\infty \to U^i_\infty$ with $f_i\circ h_i=h$, which are generically open immersions.
\end{cor}

\begin{proof}
Let $f_i': X_\infty'^i\to X_\infty$ be the $G$-equivariant morphism
one gets from the definition of a $G$-equivariant N\'eron smoothening,
i.e.~the $f_i'$ induce isomorphisms on the generic fibers, and there are
$G$-equivariant immersion
$i_i: U_\infty^i\to X_\infty'^i$ such that $f_i=f_i'\circ i_i$.
Take $\tilde{Y}_\infty=X_\infty'^1\times_{X_\infty}X_\infty'^2$
in $(stft/R,G)$.
Let $V_\infty \to \tilde{Y}_\infty$ be a $G$-equivariant N\'eron smoothening of $\tilde{Y}_\infty$,
which exists due to Theorem \ref{neronsmooth}.
Since by \cite[Corollary~4.6]{formrigI} the fiber product commutes with taking
generic fibers, 
$\tilde{Y}_\eta\cong X_\eta$, which implies that the
induced map $h: V_\infty\to X_\infty$ is a $G$-equivariant N\'eron smoothening of $X_\infty$.

We still have to show that $h_i': V_\infty\to X_\infty'^i$ factors through $U_\infty^i$.
If yes, then by construction it is automatically an open immersion on the generic fiber.
Hence assume that there is no such factorization, hence $V_\infty\setminus h_i'^{-1}(U_\infty^i)$
is not the empty formal scheme.
Then there is in particular a closed point in the special fiber of $V_\infty$ 
with residue field $F$ not mapped to $U_\infty^i$.
As $V_\infty$ is by assumption smooth over $R$,
this point extends to a $R'$-point of $V_\infty$.
Here $R'$ is an unramified extension of $R$ with residue field $F$.
Hence we have a $R'$-point of $X_\infty$ and hence a $K'$-point of $X_\eta$
not coming from an $R'$-point of $U_\infty^i$.
This contradicts to the assumption that $U_\infty^i$ is a weak N\'eron model of $X_\infty$.
\end{proof}

\subsection{Equivariant weak N\'eron models for ramifications}
\label{weak neron ramification}
The aim of this subsection is to explicitly construct equivariant N\'eron smoothenings
for some ramification models.
They will be used to explicitly compute equivariant Poincar\'e
series in Section \ref{computation},
and to compare them with Denef and Loeser's motivic zeta function in Section \ref{comparison}.

Throughout this subsection, we assume that $R$ is a complete discrete valuation ring of equicharacteristic zero, and that its residue field contains all roots of unity.
 If we do not assume that $k$ contains all roots of unity,
 then we need to consider actions of group schemes instead of abstract groups.
 In order to keep everything as simple as possible, we do not consider this case.
Moreover, we fix a regular $stft$ formal $R$-scheme $X_\infty$,
whose special fiber $X_0$ is a simple normal crossing divisor
$\sum_{i\in I}N_i E_i$ with $I=\{1,\dots,r\}$.

\begin{nota}\label{noation EJ}
Let $D=\sum_{i\in I}N_i E_i$ be a simple normal crossing divisor.
For any subset $J\subset I$,
we consider the non-singular varieties 
\[
E_J=\hspace{-2pt}\bigcap _{j\in J}\hspace{-2pt}E_j, \text{ and } E_J^o:=E_J\hspace{-3pt}\setminus\hspace{-6pt} \bigcup_{i\in I\setminus J}\hspace{-6pt} E_i.
\]
If $J=\{i\}$, we set $E_i^o:=E_J^o$.
Set moreover $m_J:=\text{gcd}\{N_i\mid i\in J\}$.
\end{nota}

\begin{defn}[\cite{NiSe}, Chapter 4]
\label{def tildeE}
For each non-empty subset $J\subset I$, we
can cover $E_J^o\subset X_\infty$  by finitely many affine open
formal subschemes $U_\infty=\mathrm{Spf}(V)$ of $X_\infty$, such
that on $U_\infty$, $t=u\prod _{i\in J}x_i^{N_i}$, with $t$ a uniformizing parameter of $R$, $u$ a unit in $V$, and the $x_i$ are local coordinates.
 The
restrictions over $E_J^o$ of the \'etale covers
$U_\infty':=\mathrm{Spf}(V\{T\}/(uT^{m_J}-1))$
of $U_\infty$
glue together to an \'etale cover $\widetilde{E}_J^o$ of $E_J^o$.

Let $\mu_{m_J}$ be the abstract group of $m_J$-th roots of unity.
This group acts on ${U}_\infty'$ by sending $T$ to $\xi T$ for every $\xi\in \mu_d$.
Note that these actions glue to a good action of $\mu_{m_J}$ on $\widetilde{E}_J^o$.

Take any integer $d$ such that $m_J$ divides $d$,
and let $\mu_d$ be the group of $d$-th roots of unity.
Then the quotient map $\mu_d\to \mu_{m_J}$ defines an action of $\mu_d$ on $\widetilde{E}_J^o$.
It is explicitly given by sending $T$ to $\xi^{\frac{d}{m_J}}T$ for all $\xi\in \mu_d$.
\end{defn}

\begin{rem}\label{tildee dl}
 Let $X$ be a smooth variety over a field $k$ of characteristic $0$,
 and let $f: X\to \mathbb{A}_k^1$ be a non-constant  morphism of $k$-varieties.
 Assume that $X_0:=f^{-1}(0)$ is a simple normal crossing divisor $\sum_{i\in I}N_i E_i$.
 Take $J\subset I$, and let $\mu_{m_J}$ be again the group of $m_J$-th roots of unity. 

 In this setting, \cite[3.3]{DL3} introduce an unramified Galois cover $\tilde{E}_J^o$ of $E_J^o$ with Galois group 
$\mu_{m_J}$ as follows:
 $E_J^o$ can be covered by such affine open subset $U$ of $X$,
such that, on $U$, $f=uv^{m_J}$,
with $u$ a unit on $U$ and $v$ a morphism from $U$ to $\mathbb{A}_k^1$.
Then the restriction of $\tilde{E}_J^o$ above $E_J^o\cap U$,
denoted by $\tilde{E}_J^o\cap U$, is defined as 
\[
 \{(z,y)\in \mathbb{A}_k^1\times(E_J^o\cap U)\mid z^{m_J}=u^{-1}\}.
\]
Gluing together the $\tilde{E}_J^o\cap U$ in the obvious way,
we obtain the cover $\tilde{E}_J^o$ of $E_J^o$,
which has a natural $\mu_{m_J}$-action
(obtained by multiplying the $z$-coordinate with the elements of $\mu_{m_J}$).

It is easy to see that
if  $X_\infty$ is the formal completion of $X$ at $X_0$,
then this definition of $\tilde{E}_J^o$ agrees
with Definition \ref{def tildeE}.
\end{rem}

\begin{defn}\label{defX0lin}
Let $J$ be a non-empty subset of $I$. We say that an integer $d\geq1$ is \emph{$J$-linear}
if there exists for every $j\in J$ an integer $\alpha_j\geq 1$ such that
$d=\sum _{j\in J}\alpha_jN_j$.
We say that $d$ is \emph{$X_0$-linear} if there exists a non-empty subset $J\subset I$
with $\lvert J \rvert>1$ and $E_J^o\neq \emptyset$, such that $d$ is $J$-linear.
\end{defn}
%

\noindent
Let us fix some more notation:
let $d>0$ be an integer,
and denote by $R(d)$ a finite totally ramified extension of $R$,
on which $\Gal(K(d)/K)\cong \mu_d$ acts nicely, see Example~\ref{ex action on r}.
Here $\mu_d$ is again the abstract group of $d$-th roots of unity.
Moreover, consider $X_\infty(d):=X_\infty\times_R R(d)\in (stft/R(d),G)$
as in Example~\ref{stan ex}.

Furthermore, denote by $n:\widetilde{X_\infty(d)}\rightarrow X_\infty(d)$ the normalization
of the formal scheme $X_\infty(d)$, see \cite[Chapter 2.1]{conrad},
and set
\[
 \widetilde{E}(d)_i^o;=E_i^o\times_{X_\infty(d)}\widetilde{X_\infty(d)}.
\]
We now describe explicitly the $\mu_d$-equivariant N\'eron smoothening for $X_\infty(d)$,
in the case that $X_0$ is not $d$-linear.
This theorem was already proved without group actions in
\cite[Theorem~4.5]{NiSe}.

\begin{thm}\label{neron}
There is a unique good $\mu_d$-action on $\widetilde{X_\infty(d)}$
such that $n$ is $\mu_d$-equivariant.
If $d$ is not $X_0$-linear, then
\[
n:Sm(\widetilde{X_\infty(d)})\rightarrow X_\infty(d)
\]
is a $\mu_d$-equivariant
N\'eron $R$-smoothening of $X_\infty(d)$. Moreover,
\[
Sm(\widetilde{X_\infty(d)})\times_R k=\hspace{-3pt}\bigsqcup_{N_i|d}\hspace{-3pt}\widetilde{E}(d)^o_i.
\]
For every $i\in I$ with $N_i|d$, $\widetilde{E}(d)^o_i$ is
$\mu_d$-closed, and there exists a $\mu_d$-equivariant
isomorphism $\widetilde{E}(d)^o_i\to \widetilde{E}^o_i$ over
$E_i^o$.
\end{thm}

\begin{proof}
Take any $\xi\in \mu_d$, and denote by $\xi$ also the corresponding automorphism of $X_\infty(d)$.
Then $\xi\circ n: \widetilde{X_\infty(d)}\to X_\infty(d)$
is also a normalization of $X_\infty (d)$ by \cite[Theorem 2.1.2]{conrad}.
Hence by the universal property of a normalization we get a unique induced automorphism $\xi'$ of $\widetilde{X_\infty(d)}$
with $n\circ \xi'=\xi\circ n$.
Doing so for every $\xi\in \mu_d$, we get a unique
$\mu_d$-action on $\widetilde{X_\infty(d)}$ such that $n$ is $\mu_d$-equivariant.
This action is good, because
inverse images under the normalization morphism of affine open subsets are affine by construction of the normalization.

As the action of $\mu_d$ on  $\widetilde{X_\infty(d)}$ restricts to $\Sm(\widetilde{X_\infty(d)})$, it follows together with 
\cite[ Theorem~4.5]{NiSe} that $h$ is in fact a $\mu_d$-equivariant N\'eron smoothening.
Also the decomposition of the special fiber of $\Sm(\widetilde{X_\infty(d)})$ was already proved there.

Denote by $\pi: X_\infty(d)\to X_\infty$ the projection map,
which is $\mu_d$-equivariant for the Galois action on $X_\infty(d)$ and the trivial action on $X_\infty$.
Now $\tilde{E}(d)_i^o=(\pi\circ n)^{-1}(E_i^o)$, and hence it is as inverse of a $\mu_d$-invariant subscheme $\mu_d$-invariant.

For any $i\in I$ with $N_i|d$, we can cover $E_i^o$ by affine open
formal subschemes $U_\infty=\mathrm{Spf}(V)$ such that on
$U_\infty$, we can write $t=ux_i^{N_i}$, with $u$ a unit, and $x_i$ local coordinates.
It was shown in the proof of \cite[Lemma 4.4]{NiSe} that
$\widetilde{U_\infty(d)}$ is given by
\[
\widetilde{U_\infty(d)}=\mathrm{Spf}(V(d)\{T\}/(t(d)^{d/N_i}T-x_i,uT^{N_i}-1)),
\]
where $V(d)$ denotes $V\otimes_R R(d)$, and $t(d)$ is a uniformizing parameter of $R(d)$.
Let $\mu_d$ act on $\widetilde{U(d)}_\infty$ by sending $t(d)$ to $\xi t(d)$
and $T$ to $\xi^{-d/N_i}T$ for all $\xi\in \mu_d$.
One can easily check that this is well defined.
Moreover ${\widetilde{U(d)}_\infty\to \Spf(V(d))}$ is $\mu_d$-equivariant,
because the action on $V(d)=V\otimes_R R(d)$ is given by sending $t(d)$ to $\xi t(d)$
for all $\xi\in \mu_d$. 
Using the uniqueness of the $\mu_d$-action on $\widetilde{X_\infty(d)}$,
this coincides with the action constructed above.

Now restrict the action to $\widetilde{U(d)}_\infty\times_Rk\cong \Spec(V[T]/(x_i, uT^{N_i}-1))$.
It is given by by sending $T$ to $\xi^{-d/N_i}T$ for all $\xi\in \mu_d$.

Now $U_0:=U_\infty\times_R \Spec(k)=\Spec(V/(t))$, and as $E_i^o$ is reduced and $u^{-1}t=x_i^{N_i}$,
the restriction of $E_i^o$ to $U_0$ is $\Spec(V/(x_i))$.
This implies that the restriction of $\tilde{E}_i^o$ to $U$ is given by
$\Spec(V[T]/(x_i,uT^{N_i}-1))$, on which
 $\mu_d$ acts by sending $T$
to $\xi^{\frac{d}{N_i}}T$ for all $\xi\in \mu_d$, see Definition \ref{def tildeE}.
Hence in particular there
exists a $\mu_d$-equivariant isomorphism $\widetilde{E}(d)^o_i\cong
\widetilde{E}^o_i$ over $E_i^o$.
\end{proof}

\noindent
To get rid of the assumption that $d$ is not $X_0$-linear later on,
we will need the following technical lemma.
This lemma was already proved in 
\cite[Lemma 7.5]{NiSe} without group action.
To make the two results comparable, we stick to the notation in \cite{NiSe}.

\begin{lem}\label{correction nonlin}
Let $J\subset I$ with $\lvert J \rvert >1$,
and let $\pi_X: X_\infty '\to X_\infty$ the formal blow-up with center $E_J$.
Denote the exceptional divisor by $E_0'$,
and the strict transform of $E_i$ by $E_i'$ for all $i\in I$.
Then for each subset $K\subset I$ with $J\!\setminus\! K \neq \emptyset$ we have
\begin{align}\label{annoying formula}
 [\tilde{E}'^o_{K\cup \{0\}}]=(\mathbb{L}_{X_0}-[X_0])^{\lvert J\setminus K\rvert -1}[\tilde{E}^o_{J\cup K}]\in K_0^{\mu_{m_{K\cup \{0\}}}}(\Var_{X_0}).
\end{align}
\end{lem}

\begin{proof}
To simplify notation, set $\mu:=\mu_{m_{K\cup \{0\}}}$.
To prove the lemma, we go along the lines of \cite[Lemma 7.5]{NiSe} to examine the actions of $\mu$.
As we will be in particular interested in fibers over points in $\tilde{E}^o_{J\cup K}$,
we can replace $X_\infty$
by affine opens $U_\infty$ such that
$t=u\prod_{j\in J\cup K}x_j^{N_j}$,
with $u$ a unit, and the $x_j$ defining the $E_J$.
As $J\setminus K\neq \emptyset$, we may assume that $x_1\in J\!\setminus\! K$.
Set $J^-:=J\setminus \{1\}$,
and set $x_j'=x_j/x_1$ for $j\in J^-$, and $x_j'=x_j$ for $j\in \{1\}\cup K\setminus J^-$.
Then we can write
\[
 t=\pi_X^*(u)(x_1')^{N_0}\hspace{-10pt}\prod_{j\in J^-\cup K}\hspace{-10pt}(x_j')^{N_j}
\]
on $X_\infty'\setminus E_1'$. Here $N_0:=\sum_{i\in J}N_i$.
With this notation we get as in \cite{NiSe} that
\begin{align*}
 &\tilde{E}^o_{J\cup K}=E_{J\cup K}^o[v]/(uv^{m_{J\cup K}}-1),\\
 &G:=\tilde{E}^o_{J\cup K}\times_{E_{J\cup K}^o}E'^o_{K\cup \{0\}}=E'^o_{K\cup \{0\}}[\omega]/(\pi_X^*(u)\omega^{m_{J\cup K}}-1),\text{ and}\\
  &\tilde{E}'^o_{ K\cup \{0\}}=E'^o_{K\cup \{0\}}[z]/(\pi_X^*(u)\hspace{-10pt}\prod _{j\in J^-\setminus K}\hspace{-10pt}z^{m_{K\cup \{0\}}}-1).
\end{align*}
To simplify notation, set from now on $m:=m_{J\cup K}$ and $n:=m_{K\cup \{0\}}$.
Using Definition \ref{def tildeE} we get that on $\tilde{E}^o_{J\cup K}$,
$\mu_m$ acts by multiplying $v$ with the elements of $\mu_{m}$,
and on $\tilde{E}'^o_{K\cup \{0\}}$, $\mu$ acts by multiplying $z$ with elements of $\mu$.
Moreover, $\mu_{m}$ acts on G by multiplying $\omega$ with the elements in $\mu_{m}$.
The projection map from $G$ to $\tilde{E}_{J\cup K}^o$ is equivariant with this $\mu_m$-actions.
As $m=\gcd\{N_j\mid j\in J\cup K\} $
divides $n=\gcd\{\sum_{i\in J}N_i,N_j\mid j\in K\}$,
we can view the $\mu_{m}$-actions
as $\mu$-action.
This is done by sending $\xi\in \mu$
to $\xi^{\frac{n}{m}}\in \mu_m$.

Now we can define an \'etale morphism $\varphi$ from $\tilde{E}'^o_{K\cup \{0\}}$ to $G$
given by
\[
\omega \mapsto  z^{\frac{n}{m}}\hspace{-10pt}\prod_{j\in J^-\setminus K}\hspace{-10pt}(x_j')^{\frac{N_j}{m}}.
\]
One checks easily that $\varphi$ is equivariant with the given $\mu$-actions.
Hence we get a $\mu$-equivariant isomorphism
\[
 \tilde{E}'^o_{K\cup \{0\}}\cong G[z]/(z^{\frac{n}{m}}- \omega \hspace{-10pt}\prod_{j\in J^-\setminus K}\hspace{-10pt}(x_j')^{-\frac{N_j}{m}}).
\]
Now take any $x\in \tilde{E}^o_{J\cup K}$ with residue field $k(x)$.
Then as shown in \cite[Lemma~7.5]{NiSe} the fiber $E_x$ over $x$ in $\tilde{E}'^o_{K\cup \{0\}}$
is isomorphic to
\begin{align*}
\Spec(k(x)[z,x_j',x_j'^{-1}]_{j\in{J^-\setminus K}}/(z^{\frac{n}{m}}- \omega(x) \hspace{-10pt}\prod_{j\in J^-\setminus K}\hspace{-10pt}(x_j')^{-\frac{N_j}{m}})).
\end{align*}
Here $\omega(x)$ is the value of the function $\omega$ in $k(x)^\times$.
Consider the stabilizer $\mu_x$ of $x$, which is a subgroup of $\mu$.
Note that $\mu_x$ might act non-trivially on $k(x)$.
As the map $\tilde{E}'^o_{K\cup \{0\}}\to \tilde{E}_{J\cup K}^o$ is $\mu$-equivariant and hence $\mu_x$-invariant,
$\mu_x$ acts on $E_x$.
This action is given by sending $z$ to $\xi z$ and $\omega(x)$ to $\xi^{-\frac{ n}{m}}\omega(x)$ for all $\xi\in \mu_x\subset \mu$.
The action on $k(x)$ agrees with the action on $x=\Spec(k(x))$.

We can now describe $E_x$ using \cite[Lemma 7.4]{NiSe}:
by B\'ezout there exist $\alpha$ and $\beta_i$
such that $\alpha \frac{n}{m}-\sum_{j\in J^-\setminus K}\beta_j \frac{N_j}{m}=1$.
So we get
\[
\omega(x)^{-\alpha \frac{n}{m}} (z^{\frac{n}{m}}- \omega(x) \hspace{-10pt}\prod_{j\in J^-\setminus K}\hspace{-10pt}(x_j')^{-\frac{N_j}{m}})
=(z\omega(x)^{-\alpha})^{\frac{n}{m}}- \hspace{-10pt}\prod_{j\in J^-\setminus K}\hspace{-10pt}(x_j'\omega(x)^{\beta_j})^{-\frac{N_j}{m}},
\]
hence we can replace $z$ by $y:=\omega(x)^{-\alpha}z$
and $x_j'$ by $y_j:=\omega(x)^{\beta_j}x_j'$.
Hence we get
\[
 E_x=\Spec(k(x)[y,y^{-1}\hspace{-3pt},y_j,y_j^{-1}]_{j\in{J^-\setminus K}}/(1-  y^{\frac{n}{m}}\hspace{-10pt}\prod_{j\in j^-\setminus K}\hspace{-10pt}(y_j)^{-\frac{N_j}{m}})),
\]
and the action of $\mu_x$ is given by sending 
$y$ to $\xi^{-\frac{n}{m}\alpha}y$ and $y_i$ to $\xi^{\frac{n}{m}\beta_i}y_i$
for all $\xi\in \mu_x$.
As shown in \cite[Lemma~7.4]{NiSe}, a change of variables gives us
\[
 E_x=\Spec(k(x)[\tilde{y}_j,\tilde{y}_j^{-1}]_{j\in J^-\setminus K})=\Spec(\hspace{-9pt}\bigotimes_{j\in J^-\setminus K} \hspace{-9pt}k(x)[\tilde{y}_j,\tilde{y}_j^{-1}]),
\]
with $y_j=y^{a_j}\prod_{i\in J -\setminus K}y_j^{a_{ji}}$ for some $a_j,a_{ij}\in \mathbb{Z}$.
One computes that the action of $\mu_x$ is given by sending $\tilde{y}_i$ to $\xi^{n_i}\tilde{y}_i$ for some $n_i\in \mathbb{N}$ for all $\xi \in \mu_x$.
Altogether $E_x\cong \mathbb{A}_{k(x)}^{\lvert J^-\setminus K\rvert}\setminus \{0\}$,
and the action on $E_x$ is the restriction of a linear action on $\mathbb{A}_{k(x)}^{\lvert J^-\setminus K\rvert}$.
With a proof analogous to that in Lemma \ref{lemma local global} it follows that
\begin{align*}
 [\tilde{E}'^o_{K\cup \{0\}}]=(\mathbb{L}_{X_0}-[X_0])^{\lvert J^-\setminus K\rvert}[\tilde{E}^o_{J\cup K}]\in K_0^{\mu}(\Var_{X_0}),
\end{align*}
and hence the claim, because $\lvert J^-\setminus K\rvert =\lvert J\setminus K\rvert$.
\end{proof}

\section{The equivariant volume Poincar\'e series}
\label{eq poincare series}
\noindent
The aim of this section to give definitions of equivariant versions
of the integral of a gauge form, \cite[Theorem-Definition 4.1.2]{motrigid},
the motivic Serre invariant, \cite[Section 4]{motrigid},
the volume Poincar\'e series and the Serre Poincar\'e series, \cite[Definition 7.2 and Definition 7.3]{NiSe}.
In Section \ref{computation} we give formulas for the equivariant
volume Poincar\'e series and the equivariant Serre Poincar\'e series,
which imply in particular that these series are rational functions.

Throughout this section, let $R$ be a complete discrete valuation ring of equal characteristic,
whose residue field $k$ contains all roots of unity.
Let $G$ be a finite abelian group, acting nicely on $R$.
Assume moreover that the action of $G$ on $R$ is tame.
These assumption will allow us to use the change of variables formula, Theorem~\ref{changevar}.

\subsection{The order of a $G$-closed gauge form}
\label{gauge}
Suppose that $X_{\infty}$ is an $stft$ formal
$R$-scheme of pure relative dimension $m$, generically smooth,
and that $\omega$ is a global section of $\Omega^m_{X_\eta/K}$,
i.e.~a \emph{gauge form} on $X_\eta$. As
$\Omega^m_{X_\eta/K}\cong \Omega^m_{X_\infty/R}\otimes_R K$, see \cite[1.5]{formrigIII},
and  $X_\infty$ is quasi-compact, we can find an integer $a\geq 0$
such that $t^a\omega$, with $t$ a uniformizer of $R$,
extends to a global section $\omega'$ of
$\Omega^m_{X_\infty/R}$.
Fix such $\omega'$ and $a$.

\begin{defn}[\cite{NiSe}, Definition~6.3]\label{dfn ordgauge}
Let $R'$ be an extension of $R$ of
ramification index one, and $\psi$ a section in
$X_\infty(R')$. The module
$M:=\psi^*\Omega^m_{X_\infty}/(\mathrm{torsion})$ is a free
$R'$-module of rank one. We define $\Ord(\omega')(\psi)$ as the
length of the $R'$-module $M/R'(\psi^*\omega')$, and set
$\Ord(\omega)(\psi):=\Ord(\omega')(\psi)-a$.
This definition
does not depend on $\omega'$ and $a$.
Identifying points
of $\Gr(X_\infty)$ with sections $\psi\in X_\infty(R')$
for some unramified extension $R'$ of $R$, we obtain
a map
$\Ord(\omega):\Gr(X_\infty)\rightarrow \Z$.
\end{defn}
%
%

\begin{defn}
We say that a gauge form $\omega$ is \emph{$G$-closed} if the
fibers of the map
$\Ord(\omega)$ are $G$-closed sets.
\end{defn}

\begin{lem}\label{Gclosed}
Take $X_\infty, Y_\infty \in (stft/R,G)$ be of pure relative dimension $m$.
Assume that $Y_\infty$ is smooth and $X_\infty$ is generically smooth.
Let $h:Y_\infty\rightarrow X_\infty$ be a $G$-equivariant morphism.
If a gauge form $\omega$ on $X_\eta$ is $G$-closed, then
 $h^*\omega$ is a $G$-closed gauge form on $Y_\eta$.
\end{lem}

\begin{proof}
By
\cite[Lemma 6.4]{NiSe},
$\Ord(h^*\omega)=\Ord(\omega)\circ h+\Ord(\Jac_h)$.
By Lemma \ref{jac} and Remark \ref{nonsmooth jac}, all fibers of $\Ord(\Jac_h)$ are $G$-closed.
As the fibers of $\Ord(\omega)$ are $G$-closed by assumption, the same holds for $\Ord(\omega)\circ h$,
because $h$ is $G$-equivariant.
Hence also the fibers of the sum, and hence of
$\Ord(h^*\omega)$, are $G$-closed sets.
\end{proof}

%
%

\subsection{The equivariant integral of a gauge form}
\label{integral of gauge}
We are now going to investigate the existence of a $G$-equivariant
N\'eron smoothening, Theorem \ref{neronsmooth},
to define the equivariant integral of a global gauge form.

\begin{theorem-definition}
Let $X_\infty\in (stft/R,G)$ be generically smooth, flat, and
of pure relative dimension over $R$, and let $\omega$ be a $G$-closed
gauge form on $X_\eta$. We set
\[
\int_{X_\infty}\hspace{-10pt}\lvert\omega\rvert:=\hspace{-3pt}\int_{U_\infty}\hspace{-10pt}\LL_{X_0}^{-\Ord(f^*\omega)}d\mu_{X_0}\in\mathcal{M}^G_{X_0},
\]
where $f:U_\infty\rightarrow X_\infty$ is
any $G$-equivariant N\'eron smoothening of $X_\infty$. This
integral is well defined, in particular it does not depend on the choice of $f$.
\end{theorem-definition}

\begin{proof}
By Theorem \ref{neronsmooth}, we know that
there exists a $G$-equivariant N\'eron smoothening ${f:U_\infty\to X_\infty}$.
By Lemma \ref{Gclosed}, $f^*\omega$ is $G$-closed, i.e.~the fibers of $\Ord(f^*\omega)$ are $G$-closed.
It follows from \cite[Proposition 2.3.8]{MR2885338}
that $\Ord(f^*\omega)$ takes only finitely many values and its fibers are cylinders.
As $U_\infty$ is smooth, Remark \ref{G-stabel G-equivariant} implies that $\Ord(f^*\omega)$ is naively $G$-integrable, so
$\int_{U_\infty}\hspace{-5pt}\LL^{-\Ord(f^*\omega)}d\mu_{X_0}$
is well defined.

Recall that a $G$-equivariant N\'eron smoothening
is not unique in general, see Remark \ref{not unique wnm}.
Hence we still need to show that the definition does not depend on the N\'eron smoothening.
Take two $G$-equivariant N\'eron smoothenings
${f_i:U_\infty ^i\to X_\infty}$, $i\in \{1;2\}$.
By Corollary~\ref{dominate},
we may assume that there is a $G$-equivariant map ${h:U_\infty^2\to U_\infty^1}$, which is generically an open immersion,
such that $f_1\circ h=f_2$.
As both $U_\infty^1$ and $U_\infty^2$ are weak N\'eron models of $X_\infty$,
the induced map ${U_\eta^2(K')\to U_\eta^1(K')}$ is a bijection for every unramified extension $K'/K$.
Hence we can apply the change of variables formula, Theorem \ref{changevar},
to $h$.
Recall moreover that by
\cite[Lemma~6.4]{NiSe}
$\Ord(f_2^*\omega)=\Ord(h^*f_1^*\omega)=\Ord(f_1^*\omega)\circ h+\Ord(\Jac_h)$.
Hence altogether we get that
\begin{align*}
 \int_{U^{1}_\infty}\hspace{-10pt}\LL^{-\Ord(f_1^*\omega)}d\mu_{X_0}
 =\hspace{-3pt}\int_{U^{2}_\infty}\hspace{-10pt}\LL^{-(\Ord(f_1^*\omega)\circ h+ \Ord(\Jac_h))}d\mu_{X_0}
 = \hspace{-3pt}\int_{U^{2}_\infty}\hspace{-10pt}\LL^{-\Ord(f_2^*\omega)}d\mu_{X_0}.
\end{align*}
\end{proof}

%

\subsection{The equivariant motivic Serre invariant}
\label{eq serre}

\begin{theorem-definition}\label{defn serre invariant}
Let $X_\infty \in (stft/R,G)$ be generically smooth, flat and of pure dimension $m$ over $R$.
We define the \emph{equivariant
motivic Serre invariant $S^G(X_\infty)$ of $X_\infty$} by
\[
S^G(X_\infty):=[U_0]\in K_0^G(\Var_{X_0})/(\LL-1),
\]
where $f:U_\infty\rightarrow X_\infty$ is
any $G$-equivariant N\'eron smoothening of $X_\infty$,
and $U_0$ is the special fiber of $U_\infty$. This
definition does not depend on the choice of $f$.
%
\end{theorem-definition}

\begin{proof}
Take two $G$-equivariant N\'eron smoothenings
$f_i:U_\infty ^i\to X_\infty$, $i\in \{1;2\}$.
By Corollary~\ref{dominate},
we may assume that there is a $G$-equivariant map $h:U_\infty^2\to U_\infty^1$, which is generically an open immersion,
such that $f_1\circ h=f_2$.
As both $U_\infty^1$ and $U_\infty^2$ are weak N\'eron models of $X_\infty$,
the induced map $U_\eta^2(K')\to U_\eta^1(K')$ is a bijection for every unramified extension $K'/K$.
This implies in particular that the induced map $h: \Gr(U_\infty^2)\to \Gr(U^1_\infty)$ is a bijection,
and hence $h: \Gr_n(X_\infty)\to \Gr_n(X_\infty)$ is a surjection for all $n$.

Let ${n\geq 2 \max\{\Ord(\Jac_h)\}}$, 
which exists, because, as shown in the proof of Theorem \ref{changevar}, $\Ord(\Jac_h)$ is naively $G$-integrable.
Set $J_e:=\theta_n(\Ord(\Jac_h)^{-1}(e))$ for every $e\in \mathbb{N}$.
Take any $x_n\in h(J_e)\subset\Gr_n(U^1_\infty)$ with stabilizer $G_x$.
Then it follows from Proposition~\ref{structure h} that
$(h^{-1}(x_n))^{\text{red}}$ is a $G_x$-equivariant affine bundle of rank $e$ over $x_n$ with affine $G_x$-action.
Hence  by Lemma~\ref{lemma local global}
\[
 [J_e]=[h(J_e)]\mathbb{L}_{X_0}^{e}\in K_0^G(\Var_{X_0}).
\]
As $U_\infty^1$ and $U_\infty^2$ are smooth, we can use Proposition \ref{prop structure map} to get that
\[
 [U_0^i]=[\Gr_0(U_\infty^i)]=[\Gr_n(U_\infty ^i)]\mathbb{L}_{X_0}^{-nm}\in K_0^G(\Var_{X_0}).
\]
Now we set $\mathbb{L}$ equal to $1=[X_0]$ in $K_0^G(\Var_{X_0})$, and get
\begin{align*}
 [U_0^1]=[\Gr_n(U_\infty^1)]=\hspace{-3pt}\sum [J_e]
 =\hspace{-3pt}\sum[h(J_e)]=[\Gr_n(U_\infty^2)]=[U_0^2]\in K_0^G(\Var_{X_0})\!/\!(\LL-1).
\end{align*}
\end{proof}

\begin{rem}\label{rem specializing}
Assume that  $X_\eta$ admits a $G$-closed global gauge form $\omega$, and
let $f: U_\infty\to X_\infty$ be a $G$-equivariant N\'eron smoothening of $X_\infty$. 
By \cite[4.3.1]{motrigid}
the function $\Ord(f^*\omega)$ is constant on
$\theta_0^{-1}(D)$ for every connected component $D$ of $U_0$.
As $\Ord(f^*\omega)$ is $G$-closed, this implies 
that it is constant with value $\Ord_C(f^*\omega)$ on the $G$-stable cylinder
$\theta_0^{-1}(C)$, where $C$ is the orbit of $D$.
Denote by $GC(U_0)$ the set of orbits on the connected components of $U_0$. With this notation we get that
\[
 \int_{X_\infty}\hspace{-10pt}\lvert \omega \rvert=\LL^{-m}\hspace{-12 pt}\sum_{C\in G\mathcal{C}(U_0)}\hspace{-12pt}[C]\LL^{-\Ord_{C}(f^*\omega)}\in \mathcal{M}^G_{X_0}.
\]
Hence $S^G(X_\infty)$ is the image of
$\int_{X_\infty}\hspace{-3pt}|\omega|$ under the projection morphism
\[
\mathcal{M}^G_{X_0}\rightarrow\mathcal{M}^G_{X_0}/(\LL-1)\cong K_0^G(\Var_{X_0})/(\LL-1).
\]
\end{rem}

\subsection{Equivariant Poincar\'e series}
\label{eq poincare}

We suppose now that $k$ has characteristic zero.
Let $X_\infty$ be a
generically smooth, $stft$ formal $R$-scheme of pure dimension $m$.
Recall that for any integer $d>0$, 
$\Gal(K(d)/K)\cong\mu_d$, the group of $d$-th roots of unity,
acts on $R(d)$ and $X_\infty(d)$,
see Example \ref{ex action on r} and Example \ref{stan ex}.
If $\omega$ is a gauge form on $X_\eta$, we denote
by $\omega(d)$ the pullback of $\omega$ to the generic fiber ob $X_\infty(d)$.
By construction $\omega(d)$ is a $\mu_d$-closed gauge form. 

Recall that the groups $\mu_d$ form a projective system with respect to the quotient maps $\mu_d'\to \mu_d$ which we have whenever $d$ divides $d'$.
We denote by $\hat{\mu}$ the projective limit of the $\mu_d$.
By construction $\mu_d$ is a quotient of $\hat{\mu}$
for all $d$.
Hence we can view the integral
$\int_{X_\infty(d)}\lvert\omega(d)\rvert$ and the motivic Serre invariant
$S^{\mu_d}(X_\infty(d))$ as elements in
$\mathcal{M}^{\hat{\mu}}_{X_0}$ and
$K^{\hat{\mu}}_0(\Var_{X_0})/(\LL-1)$, respectively.
Here $\hat{\mu}$ acts trivially on $X_0$, the special fiber of $X_\infty$,
which is also the special fiber of $X_\infty(d)$.

\begin{defn} \label{volume}
For any integer $d>0$,
we put
\[
F(X_\infty,\omega;d):=\int_{X_\infty(d)}\hspace{-21pt}\lvert\omega(d)\rvert\in \mathcal{M}^{\hat{\mu}}_{X_0}.
\]
This defines a function
$F(X_\infty,\omega)\,:\,\N\rightarrow \mathcal{M}^{\hat{\mu}}_{X_0}$
which we call the \emph{equivariant local singular series}
associated to the pair $(X_\infty,\omega)$.
The \emph{equivariant volume Poincar\'e series}
$S(X_\infty,\omega;T)$ of the pair $(X_\infty,\omega)$ is the
generating series
\[
S(X_\infty,\omega;T)=\sum_{d>0}F(X_\infty,\omega;d)T^d\in \mathcal{M}^{\hat{\mu}}_{X_0}\Pol T \Por.
\]
\end{defn}

\begin{defn}\label{serreser}
The \emph{equivariant Serre Poincar\'e series} $S(X_\infty;T)$ of $X_\infty$ is the
generating series
\[
S(X_\infty;T)=\sum_{d>0}S^{\mu_d}(X_\infty(d))T^d\in K^{\hat{\mu}}_0(\Var_{X_0})/(\LL-1)\Pol T \Por.
\]
\end{defn}

\begin{rem}\label{rem specializing2}
Definition \ref{serreser} does not require that $X_\eta$ admits a
global gauge form, see Theorem-Definition \ref{defn serre invariant}.
If it does,
then by Remark \ref{rem specializing} the series $S(X_\infty,\omega;T)$ specializes to the Serre
Poincar\'e series $S(X_\infty;T)$ under the morphism
\[
\mathcal{M}^{\hat{\mu}}_{X_0}\Pol T \Por\rightarrow\mathcal{M}^{\hat{\mu}}_{X_0}/(\LL-1)\Pol T \Por\cong K^{\hat{\mu}}_0(\Var_{X_0})/(\LL-1)\Pol T \Por.
\]
\end{rem}

%

\subsection{Computation of the equivariant Poincar\'e series}
\label{computation}
The aim of this subsection is to give explicit formulas for the equivariant Poincar\'e series
and the equivariant Poincar\'e series.
We will need these formulas to compare the equivariant Poincar\'e series with
Denef and Loeser's motivic zeta function in Section \ref{applications}.
To get the formulas, we will use Section \ref{weak neron ramification},
in particular the explicit weak N\'eron model constructed in Theorem~\ref{neron}.
Note that similar formulas were already proved in \cite[Theorem~7.6 and Corollary 7.7]{NiSe} in the non-equivariant case.

We will use the same assumptions and notations as in Section \ref{eq poincare}.
Moreover, we fix  an \emph{embedded resolution} $h:{X'_\infty\rightarrow X_\infty}$ of $X_\infty$,
i.e.~a morphism of flat $sftf$ formal $R$-schemes inducing
an isomorphism on the generic fiber, such that $X_\infty'$ is regular and 
such that the
special fiber $X'_0=\sum_{i\in I} N_i E_i$
is a simple normal crossing divisor.
For all $J\subset I$, let $\tilde{E}_J^o$ and $\tilde{E}_i^o$ be defined as in
Definition \ref{def tildeE}.

\begin{thm}\label{explicit}
Let $\omega$ be a
gauge form on $X_\eta$, and let
$\mu_i$ be the order of $\omega$ of any point in $E_i$.
For any integer
$d>0$, we have that
\[
F(X_\infty,\omega;d)
 =\LL^{-m}\hspace{-5pt}\sum_{\emptyset\neq J\subset I}\hspace{-5pt}(\LL-1)^{|J|-1}[\widetilde{E}_J^{o}]
(\hspace{-15pt}\sum_{\stackrel{k_i\geq1,i\in J}{\sum_{i\in J} k_iN_i=d}}\hspace{-15pt}\LL^{-\sum_i k_i \mu_i}\,)
\in \mathcal{M}^{\hat{\mu}}_{X_0}.
\]
Moreover, the equivariant volume Poincar\'e series is explicitly given by
\[
S(X_\infty,\omega;T)
=\LL^{-m}\hspace{-5pt}\sum_{\emptyset\neq J\subset I}\hspace{-5pt}(\LL-1)^{|J|-1}[\widetilde{E}_J^{o}]
\prod_{i\in J}\frac{\LL^{-\mu_i}T^{N_i}}{1-\LL^{-\mu_i}T^{N_i}}\in \mathcal{M}^{\hat{\mu}}_{X_0}\Pol T \Por.
\]
\end{thm}

\begin{proof}
We go along the lines of the proof of the non-equivariant case, \cite[Theorem 7.6]{NiSe}, and show that it remains valid if we
take the $\hat{\mu}$-action into account. 

Assume that $d$ is not $X_0'$-linear.
Then we can use Theorem \ref{neron} to get that $\Sm(\widetilde{X_\infty'(d)})\to X_\infty'(d)$
is an equivariant N\'eron smoothening of $X'_\infty(d)$, and hence also of $X_\infty(d)$,
and the class of the special fiber of $\Sm(\widetilde{X'_\infty(d)})$
agrees with the sum $\sum_{i\in I, N_i\mid d}[\tilde{E}_i^o]$ in $K^{\hat{\mu}}_0(\Var_{X_0})$.
By \cite[Lemma 6.3]{NiSe} the pullback of $\omega$ to $\Sm(\widetilde{X'_\infty(d)})$
has value $\mu_id/N_i$ on $\widetilde{E}(d)_i^o$,
the pullback of $E_i$ to $\Sm(\widetilde{X'_\infty(d)})$.
Moreover, if $N_i$ divides $d$, then $\tilde{E}_i^o\cong\tilde{E}(d)^o_i$.
Hence
\begin{align}\label{eq xolin}
F(X_\infty,\omega;d)&=\mathbb{L}^{-m}\hspace{-7pt}\sum_{i\in I,N_i\mid d}\hspace{-7pt}[\tilde{E}_i^o]\LL^{\mu_id/N_i}\nonumber \\
&=\LL^{-m}\hspace{-5pt}\sum_{\emptyset\neq J\subset I}\hspace{-5pt}(\LL-1)^{|J|-1}[\widetilde{E}_J^{o}](\hspace{-15pt}\sum_{\stackrel{k_i\geq1,i\in J}{\sum_{i\in J} k_iN_i=d}}\hspace{-15pt}\LL^{-\sum_i k_i \mu_i}\,)\in \mathcal{M}^{\hat{\mu}}_{X_0}.
\end{align}
In the second equation it was used again that $d$ is not $X_0$-linear.

By \cite[Lemma 5.17]{NiSe2} we can always find a map $X_\infty''\to X_\infty'$ constructed 
by a sequence of blowups of strata $E_J$ for some $J\subset I$, such that $d$ is $X_0''$-linear and $X_\infty''\to X_\infty$
is an embedded resolution.
As by Lemma \ref{correction nonlin} Formula (\ref{annoying formula})
holds in $K_0^{\mu_{m_{J\cup K}}}(\Var_{X_0})$,
and hence also in $K_0^{\hat{\mu}}(\Var_{X_0})$ and $\mathcal{M}_{X_0}^{\hat{\mu}}$,
we can show with the same computation as in
\cite[Theorem 7.6]{NiSe} that the right hand side of Equation (\ref{eq xolin}) is invariant under blow-ups of strata $E_J$.
This implies the first part of the theorem.
The second part follows from this result with exactly the same computation
as in the prove of \cite[Corollary 7.7]{NiSe}.
\end{proof}

\noindent
We also get a similar formula for the equivariant Serre invariant and the equivariant Serre Poincar\'e series.
If $X_\eta $ admits a global gauge form, this formula follows immediately from Theorem \ref{explicit} using Remark \ref{rem specializing} and Remark \ref{rem specializing2}, respectively.
But it holds without assuming the existence of a global gauge form.

\begin{thm}\label{explicit serre}
For any integer
$d>0$, we have that
\begin{align*}
S^{\mu_d}(X_\infty(d))&=\hspace{-7pt}\sum_{i\in I, N_i|d}\hspace{-7pt}[\widetilde{E}_i^{o}]\in K^{\hat{\mu}}_0(\Var_{X_0})/(\LL-1) \text { and}\\
S(X_\infty;T)&=\sum_{i\in I}[\widetilde{E}^o_i]\frac{T^{N_i}}{1-T^{N_i}}\in K^{\hat{\mu}}_0(\Var_{X_0})/(\LL-1)\Pol T \Por.
\end{align*}
\end{thm}

\begin{proof}
Assume first that $d$ is not $X_0'$-linear.
 As in the proof of Theorem \ref{explicit}, we can use Theorem \ref{neron}
 and get that
 \[
S^{\mu_d}(X_\infty(d))=\hspace{-7pt}\sum_{i\in I, N_i|d}\hspace{-7pt}[\widetilde{E}_i^{o}]\in
K^{\hat{\mu}}_0(\Var_{X_0})/(\LL-1).
\]
Take a map $X_\infty''\to X_\infty'$ as above
such that $d$ is not $X_0''$-linear.
It follows from Lemma \ref{correction nonlin} that the classes of $\tilde{E}_i'^o$ in $K^{\hat{\mu}}_0(\Var_{X_0})/(\LL-1)$
are zero if they are coming from an exceptional divisor $E'_i$,
hence we get rid of the assumption that $X_0'$ is not $d$-linear, and the first claim follows.
This implies that
\begin{align*}
S(X_\infty;T)&=\sum_{d\geq 0}T^d\hspace{-7pt}\sum_{i\in I, N_i|d}\hspace{-7pt}[\widetilde{E}^o_i]=\sum_{i\in I}[\widetilde{E}^o_i]\sum_{d'>0}T^{d'N_i}\\
&=\sum_{i\in I}[\widetilde{E}^o_i]\frac{T^{N_i}}{1-T^{N_i}}\in K^{\hat{\mu}}_0(\Var_{X_0})/(\LL-1)\Pol T \Por.
\end{align*}
\end{proof}

\noindent
By \cite[Proposition 2.5]{NiSe}, every affine generically smooth stft formal $R$-scheme admits an embedded resolution.
Here one needs that $k$ has characteristic zero.
Hence as without group actions, see \cite[Corollary 7.8]{NiSe},
we have the following corollary.

\begin{cor}\label{rationality}
Let $X_\infty$ be a generically smooth $stft$ formal $R$-scheme,
of pure relative dimension, that admits a global gauge form $\omega$ on $X_\eta$.
Then there exists a
finite subset $S$ of $\Z\times \N^{\ast}$ such that
$S(X_\infty,\omega;T)$ belongs to the ring
\[
\mathcal{M}^{\hat{\mu}}_{X_0}[T]\left[\frac{\LL^{a}T^b}{1-\LL^{a}T^b} \right]_{(a,b)\in S}\subset\mathcal{M}^{\hat{\mu}}_{X_0}\Pol T \Por.
\]
\end{cor}

\noindent
Hence in particular $S(X_\infty,\omega;T)$ is a rational function.
Similarly one gets that the equivariant Serre Poincar\'e series $S(X_\infty;T)$ is a rational function.

\section{Application to Denef and Loeser's motivic zeta functions}
\label{applications}
\noindent
Throughout this section,
let $X$ be an irreducible smooth variety of dimension $m+1$ over a field of characteristic zero $k$ containing all roots of unity,
together with a dominant map $f: X\to \mathbb{A}^ 1_{k}$.
Denote by $X_0$ the special fiber $f^{-1}(0)$ of $f$ over the point $0\in \mathbb{A}^1_k$.
Denote by $X_\infty$ the formal completion of $X$ along $X_0$.
This is a generically smooth $sftf$ formal scheme of relative dimension $m$ over $R:=k\Pol t \Por$.
Denote by $X_\eta$ the generic fiber of $X_\infty$.

The aim of this section is to recover Denef and Loeser's motivic zeta function of $X$ from
a special equivariant Poincare series of $X_\infty$,
namely from the equivariant motivic Weil generating series.
Moreover we define and examine the equivariant motivic volume of a formal $R$-scheme,
from which we can recover the motivic nearby cycles $\mathcal{S}_f$ of $f$.
Before we do so, we recall some definitions and fix notations.

\subsection{Jet schemes}
\label{jets}
 As for example in \cite[2.1]{DL3}, we define for any integer
$d>0$, the \emph{$d$-th jet scheme} $\mathcal{L}_d(X)$ to be the $k$-scheme representing the
functor
\[
(k-alg)\rightarrow (Sets);\ A\mapsto X(A[t]/(t^{d+1}))=\Hom_k(\Spec(A[t]/(t^{d+1})),X).
\]
Following \cite[3.2]{DL3}, we denote by $\mathcal{X}_d$ and
$\mathcal{X}_{d,1}$ the $X_0$-varieties
\begin{align*}
\mathcal{X}_{d}&:=\{\psi\in\mathcal{L}_{d}(X)\mid \Ord_tf(\psi(t))=d\} \text{ and}\\
\mathcal{X}_{d,1}&:=\{\psi\in\mathcal{L}_{d}(X)\mid f(\psi(t))=t^d\,\text{mod}\,t^{d+1}\},
\end{align*}
where the structural morphisms to
$X_0$ are given by reduction modulo $t$. Let $\mu_d$, the group of $d$-th roots of unity, act
on $\mathcal{X}_{d,1}$ by sending $\psi(t)\in \mathcal{X}_{d,1}$ to $\psi(\xi t)$ for any $\xi$ in
$\mu_d$. Hence $\mathcal{X}_{d,1}$ can be viewed as an
$X_0$-variety with good $\hat{\mu}$-action,
where $\hat{\mu}$ denotes again the projective limit of the $\mu_d$.

\begin{rem}\label{relation xd and xd1}
 Take any $d>0$.
As explained in \cite[3.2]{DL3}, we can connect $\mathcal{X}_d$ and $\mathcal{X}_{d,1}$ as follows: 
Look at the map $\varphi: \mathcal{X}_{d,1}\times \mathbb{G}_{m,k}\to \mathcal{X}_d$
given by sending $(\psi(t),a)$ to $\psi(at)$.
Let $\mu_d$ act on $\mathcal{X}_{d,1}\times \mathbb{G}_{m,k}$
by sending $(\psi(t),a)$ to $(\psi(\xi t), \xi^{-1}a)$ for all $\xi\in \mu_d$.
As $\varphi(\psi(t),a)=\psi(at)=\psi(\xi^{-1}a\xi t)=\varphi(\psi(\xi t),\xi^{-1}a)$,
$\varphi$ factors through a map $\tilde{\varphi}: (\mathcal{X}_{d,1}\times \mathbb{G}_{m,k})/\mu_d\to \mathcal{X}_d$,
which is in fact an isomorphism.
As the action on $\mathcal{X}_{d,1}\times \mathbb{A}_{k}^1$ extending the action on $\mathcal{X}_{d,1}\times \mathbb{G}_{m,k}$
is linear over the base $\mathcal{X}_{d,1}$,
we get that
\[
 [\mathcal{X}_{d,1}\times \mathbb{G}_{m,k}]=[\mathcal{X}_{d,1}]\mathbb{L}_{X_0}-[\mathcal{X}_{d,1}]\in \mathcal{M}_{X_0}^{\mu_d}.
\]
This implies, using that the quotient map on $\mathcal{M}_{X_0}^{\mu_d}$ is well defined by \cite[Corollary 8.4]{abi2}, that
\begin{equation*}
[\mathcal{X}_d]=(\LL_{X_0}-1)[\mathcal{X}_{d,1}/\mu_d] \in\mathcal{M}_{X_0}.
\end{equation*}
\end{rem}

%
%
%

\subsection{Motivic zeta functions}
\label{motivic zeta}
In \cite[3.2.1]{DL3}, the \emph{motivic zeta function} $Z(f;T)$ of $f$ is
defined as
\[
Z(f;T):=\sum_{d=1}^{\infty}[\mathcal{X}_{d,1}]\LL^{-(m+1)d}T^{d}\in \mathcal{M}^{\hat{\mu}}_{X_0}\Pol T \Por,
\]
and the \emph{naive motivic zeta function}
$Z^{naive}(T)$ is defined as
\[
Z^{naive}(f;T):=\sum_{d=1}^{\infty}[\mathcal{X}_{d}]\LL^{-(m+1)d}T^{d}\in \mathcal{M}_{X_0}\Pol T \Por.
\]

\begin{rem}\label{naive quotient}
 Using Remark \ref{relation xd and xd1} one can recover the naive motivic zeta function
 from the motivic zeta function as follows:
 \[
  Z^{\text{naive}}(f;T)=(\mathbb{L}-1)Z(f;T)/{\hat{\mu}}\in \mathcal{M}_{X_0}\Pol T \Por,
 \]
where $Z(f;T)/\hat{\mu}$ is the image of $Z(f,T)$ under the map
$q: \mathcal{M}^{\hat{\mu}}_{X_0}\Pol T\Por\to \mathcal{M}_{X_0}\Pol T \Por$
sending $\sum a_iT^i$ to $\sum a_i/\hat{\mu}\ T^i$, where $a_i/\hat{\mu}$
denotes the image of $a_i$ under the quotient map given in \cite[Corollary 8.5]{abi2}.
\end{rem}

\noindent
Let $h:X'\rightarrow X$ be an \emph{embedded resolution} for $f$, i.e.~$h$ is a proper morphism
inducing an isomorphism $Y\setminus X'_0\to X\setminus X_0$, $Y$ is smooth, and $X'_0=\sum_{i\in I}N_i E_i$ is a simple normal crossing divisor.
Let $K_{X'/X}=\sum_{i\in I}(\xi_i-1)E_i$ be the relative canonical divisor of $f$.
By \cite[Theorem~3.3.1]{DL3} we have
\begin{align}\label{formula mzf}
Z(f;T)=\hspace{-5pt}\sum_{\emptyset\neq J\subset I}\hspace{-5pt}(\LL-1)^{|J|-1}[\widetilde{E}_J^{o}]\prod_{i\in J}\frac{\LL^{-\xi_i}T^{N_i}}{1-\LL^{-\xi_i}T^{N_i}} \in \mathcal{M}^{\hat{\mu}}_{X_0}\Pol T \Por.
\end{align}
Here $\tilde{E}_J^o$ is given
by Definition \ref{def tildeE}, which agrees with the definition by Denef and Loeser, see Remark \ref{tildee dl}.
By Remark \ref{naive quotient} and the fact that $\tilde{E}^o_{J}/\hat{\mu}=E^o_J$, with $E^o_J$  as in Notation \ref{noation EJ}, Equation (\ref{formula mzf}) implies that
\begin{align*}
Z^{naive}(f;T)
=\hspace{-5pt}\sum_{\emptyset\neq J\subset I}\hspace{-5pt}(\LL-1)^{|J|}[E_J^{o}] 
\prod_{i\in J}\frac{\LL^{-\xi_i}T^{N_i}}{1-\LL^{-\xi_i}T^{N_i}} \in \mathcal{M}_{X_0}\Pol T \Por.
\end{align*}
Inspired by the $p$-adic case,
Denef and Loeser defined the \emph{motivic
nearby cycles $\mathcal{S}_f$} by taking formally the
limit of $-Z(f;T)$ for $T\to\infty$ in $\mathcal{M}_{X_0}^{\hat{\mu}}$.
By Equation (\ref{formula mzf}) this limit is well defined, and 
\begin{align}
\label{formula sf}
\mathcal{S}_f= \hspace{-5pt}\sum_{\emptyset\neq J\subset I}\hspace{-5pt}(1-\LL)^{|J|-1}[\widetilde{E}_J^{o}]\in \mathcal{M}^{\hat{\mu}}_{X_0}.
\end{align}

\subsection{Recovering the motivic zeta function}
\label{comparison}
Assume for this subsection, that $X_\eta$ admits a global gauge form $\omega$.
As in \cite[9.5]{NiSe}, we can associate to it its \emph{Gelfand-Leray form}
$\frac{\omega}{df}$.

\begin{defn}\label{local}
We define the \emph{equivariant motivic Weil generating series associated to $f$} by
$
S(f;T):=
S(X_\infty,\frac{\omega}{df};T)
\in\mathcal{M}^{\hat{\mu}}_{X_0}\Pol T \Por.
$
\end{defn}

\begin{thm}\label{comparzeta}
Let $X$ be a smooth irreducible variety over $k$ of dimension $m+1$,
and let $f:X\rightarrow \A^1_k$ be a
 dominant morphism.
 Assume that there exists a global gauge form on $X_\eta$. Then
\[
 S(f;T)=\LL^{-m}Z(f;\LL T)\in\mathcal{M}^{\hat{\mu}}_{X_0}\Pol T\Por.
\]
\end{thm}

\begin{proof}
Let $h: X'\to X$ be an embedded resolution 
of $f$.
Let ${X_0'=\sum_{i\in I}N_iE_i}$ be its special fiber 
and $K_{X'/X}=\sum_{i\in I}(\xi_i-1)E_i$ its relative canonical divisor.
Then by \cite[Lemma 9.6]{NiSe}
\[
 \Ord_{E_i}(\frac{h^*\omega}{d(f\circ h)})=\xi_i-N_i.
\]
With this fact the theorem follows immediately from
Theorem \ref{explicit} and Formula~(\ref{formula mzf}).
\end{proof}

\noindent
Using Remark \ref{naive quotient}, we can also recover the naive
motivic zeta function from $S(f,T)$.

\subsection{Recovering the motivic nearby cycles}\label{rec mnc}
Using Theorem \ref{comparzeta}
we can also recover $S_f$ from $S(f,T)$
by taking formally the limit of $-S(f,T)$ for $T\to\infty$
and multiplying it with $\LL^{m}$.
Due to Corollary \ref{rationality},
this limit also makes sense without assuming that the formal scheme $X_\infty$ comes from a
morphism $f$, which
leads us to the following definition, which was given in \cite[Definition 8.3]{NiSe} in the non-equivariant case.

\begin{defn}\label{eq mot vol1}
 Let $X_\infty$ be a $sftf$ formal scheme of pure relative dimension $m$ over $R$
 with smooth generic fiber $X_\eta$, which admits a gauge form $\omega$.
 We define \emph{equivariant motivic volume} $\mathcal{S}_{X_\infty}\in \mathcal{M}_{X_0}^{\hat{\mu}}$
 to be the formal limit of $-S(X_\infty, \omega; T)$ for $T\to \infty$.
\end{defn}

\noindent
Take any embedded resolution of $X_\infty$ with special fiber $\sum_{i\in I}N_iE_i$,
and let $\tilde{E}^o_J$ be given as in Definition \ref{def tildeE}.
Then Theorem \ref{explicit} implies that $\mathcal{S}_{X_\infty}$ satisfies the following formula:
\begin{align}\label{formula mot volume}
 \mathcal{S}_{X_\infty}=\mathbb{L}^{-m}\hspace{-5pt}\sum_{\emptyset\neq J\subset I}\hspace{-5pt}(1-\LL)^{|J|-1}[\widetilde{E}_J^{o}]\in \mathcal{M}^{\hat{\mu}}_{X_0}.
\end{align}
In particular the definition of $\mathcal{S}_{X_\infty}$ does not depend on $\omega$.

Now take any cover $\{X_\infty^l\}_{l\in L}$ of $X_\infty$ by open formal subschemes.
Then by construction $\{X_\infty^l(d)\}_{l\in L}$ is a $\mu_d$-invariant cover
of $X_\infty(d)$ for all $d$.
Hence Remark~\ref{additive} implies
that 
\[
F(X_\infty,\omega; d)=\hspace{-5pt}\sum_{\emptyset\neq\mathcal{L}\subset L}\hspace{-5pt}(-1)^{\lvert \mathcal{L}\rvert - 1}F(X_\infty^{\mathcal{L}},\omega;d)\in \mathcal{M}_{X_0}^{\hat{\mu}},
\]
where $X_\infty^\mathcal{L}:=\cap_{l\in \mathcal{L}}X_\infty^l$ for all $\mathcal{L}\subset L$.
Hence summing up the $F(X_\infty,\omega;d)$ we get that 
the analog equation holds also for $-S(X_\infty,\omega;T)$.
Taking formally the limit for $T$ against $\infty$,
we get
\begin{align}\label{sxinfty}
  \mathcal{S}_{X_\infty}=\hspace{-5pt}\sum_{\emptyset\neq\mathcal{L}\subset L}\hspace{-5pt}(-1)^{\lvert \mathcal{L}\rvert - 1}\mathcal{S}_{X_\infty^{\mathcal{L}}}\in \mathcal{M}_{X_0}^{\hat{\mu}}.
\end{align} 
Inspired by this equation, we can,
as in the non-equivariant case, see \cite[Section~8, Remark]{NiSe},
define $\mathcal{S}_{X_\infty}$ 
without assuming the existence of a global gauge form on $X_\eta$.
Here we use that $X_\eta$ admits a gauge form locally, because $X_\infty$ is generically smooth.

\begin{defn}\label{eq mot vol2}
Let $X_\infty$ be a $sftf$ formal generically smooth $R$-scheme of dimension $m$ over $R$.
Fix any finite cover $\{X_\infty^l\}_{l\in L}$ of $X_\infty$ by open formal subschemes,
such that $X_\eta^l$ admits a global gauge form $\omega_l$.
For all $\mathcal{L}\subset L$ set $X_\infty^{\mathcal{L}}:=\cap_{l\in \mathcal{L}}X_\infty^l$.
We define the \emph{equivariant motivic volume} of $X_\infty$ by
\[
 \mathcal{S}_{X_\infty}:=\hspace{-5pt}\sum_{\emptyset\neq\mathcal{L}\subset L}\hspace{-5pt}(-1)^{\lvert \mathcal{L}\rvert - 1}\mathcal{S}_{X_\infty^{\mathcal{L}}}\in \mathcal{M}_{X_0}^{\hat{\mu}},
\]
where $\mathcal{S}_{X_\infty^{\mathcal{L}}}$ is given by Definition \ref{eq mot vol1}.
\end{defn}

\noindent
Using Equation (\ref{sxinfty}), Definition \ref{eq mot vol1} and Definition \ref{eq mot vol2} agree
in the case that $X_\eta$ admits a global gauge form.
Moreover, if we have two covers
$\{X_\infty^l\}_{l\in L}$ and
$\{Y_\infty^l\}_{l\in L'}$ of $X_\infty$
we can compare them via the common refinement $\{X_\infty^l\cap Y_\infty^r \}_{l,r\in L\times L'}$,
so Definition~\ref{eq mot vol2} does not depend on the chosen cover.

%
If we use Formula (\ref{formula mot volume}) to compute the $\mathcal{S}_{X_\infty^\mathcal{L}}$, we get
\begin{align*}
 \mathcal{S}_{X_\infty}&=\hspace{-5pt}\sum_{\emptyset\neq\mathcal{L}\subset L}\hspace{-5pt}(-1)^{\lvert \mathcal{L}\rvert - 1}\mathbb{L}^{-m}\hspace{-5pt}\sum_{\emptyset\neq J\subset I}\hspace{-5pt}(1-\LL)^{|J|-1}[\widetilde{E}_J^{o}\times_{X_0}X_\infty^{\mathcal{L}}]\\
 &=\mathbb{L}^{-m}\hspace{-5pt}\sum_{\emptyset\neq J\subset I}\hspace{-5pt}(1-\LL)^{|J|-1}[\widetilde{E}_J^{o}]\in \mathcal{M}_{X_0}^{\hat{\mu}}.
\end{align*}
This formula and Formula (\ref{formula sf}) imply the following proposition:

\begin{prop}
 Let $X$ be a smooth, irreducible variety over $k$ of dimension $m+1$, and let $f:X\to \mathbb{A}^1_k$
 be a non-constant morphism.
 Let $X_\infty$ be the formal completion of $X$ along $X_0=f^{-1}(0)$.
 Then
 \[
  \mathcal{S}_f=\mathbb{L}^m\mathcal{S}_{X_\infty}\in \mathcal{M}_{X_0}^{\hat{\mu}}.
 \]
\end{prop}

\noindent
As done in \cite[Section 9]{abi2} with $\mathcal{S}_f/\hat{\mu}$, we can now study
the quotient $\mathcal{S}_{X_\infty}/\hat{\mu}$.
Using Formula (\ref{formula mot volume}) we 
we can in particular deduce the following result
with the same proof as in \cite[Proposition 9.5]{abi2}.

\begin{cor}\label{special fiber modulo l}
 Let $X_\infty$ be a $sftf$ formal scheme of relative dimension $m$ over $R$
 with smooth generic fiber.
Then the class of 
$ X_0'$ modulo $\mathbb{L}$ in $\mathcal{M}_{X_0}$ does not depend on the choice of 
an embedded resolution $h:X_\infty'\to X_\infty$.
\end{cor}

\noindent
For a discussion of this result we refer to \cite[Section 9]{abi2}.

\medskip
\noindent
Finally remark that
modulo $\mathbb{L}-1$ we can recover $\mathcal{S}_f$ also from the equivariant
Serre Poincar\'e series.
This follows from Theorem \ref{explicit serre} and Formula (\ref{formula sf}).
More concrete, we have the following proposition:

\begin{prop}
Let $X$ be a smooth, irreducible variety over $k$, let ${f: X\to \mathbb{A}_k^1}$ be a dominant morphism,
and let $X_\infty$ be the formal completion of $X$ along ${X_0=f^{-1}(0)}$.
Then
 the limit
 of $-S(X_\infty;T) \mathbb{L}^m$ for $T\to \infty$ agrees with $\mathcal{S}_f$ in ${K_0^{\hat{\mu}}(\Var_{X_0})/(\mathbb{L}-1)}$.
\end{prop}

\bibliographystyle{babalpha}

	\bibliography{refgaction}

\end{document}